\documentclass[11pt, reqno]{amsart}
\usepackage{geometry}                                                      
\usepackage{amsmath, amssymb, latexsym, wasysym, amsthm, amsfonts, multicol, color, soul, graphicx, tikz, tkz-euclide, array}    
\usetikzlibrary{decorations.pathreplacing, arrows, positioning, shapes.geometric, calc, through}
\usepackage{mleftright}  
\usepackage{mathtools}           
\usepackage{verbatim}                          
                                                                             
\usepackage{parskip}   
\usepackage{qtree}
\usepackage{caption}
\usepackage{subcaption}
\usepackage{adjustbox}
\usepackage{float}
\usepackage{enumerate}
\usepackage{etex}
\usepackage{hyperref} 

\newtheorem{theorem}{Theorem}[section]
\newtheorem{corollary}[theorem]{Corollary}
\newtheorem{lemma}[theorem]{Lemma}
\newtheorem{proposition}[theorem]{Proposition}

\newtheorem{definition}[theorem]{Definition}
\newtheorem{remark}[theorem]{Remark}

\newtheorem{question}[theorem]{Question}
  
\def\multiset#1#2{\ensuremath{\left(\kern-.3em\left(\genfrac{}{}{0pt}{}{#1}{#2}\right)\kern-.3em\right)}}
  
\title{RNA, Local moves on plane trees, and transpositions on tableaux}

\thanks{The authors acknowledge the National Science Foundation (DMS--1143716) and Smith College for their support of the Center for Women in Mathematics.  The third author was also partially supported by NSF grant DMS--1248171.}

\author{Laura Del Duca}
\address{Department of Mathematics and Statistics, Smith College, Northampton, MA 01063 U.S.A.}
\email{lauraseegerer@gmail.com}

\author{Jennifer Tripp}
\address{Department of Mathematics and Statistics, Smith College, Northampton, MA 01063 U.S.A.}
\email{jentripp@comcast.net}

\author{Julianna Tymoczko}
\address{Department of Mathematics and Statistics, Smith College, Northampton, MA 01063 U.S.A.}
\email{jtymoczko@smith.edu}

\author{Judy Wang}
\address{Department of Mathematics and Statistics, Smith College, Northampton, MA 01063 U.S.A.}
\email{jwang@alumnae.smith.edu}


\begin{document}

\begin{abstract}
We define a collection of functions $s_i$ on the set of plane trees (or standard Young tableaux).  The functions are adapted from transpositions in the representation theory of the symmetric group and almost form a group action.  They were  motivated by {\em local moves} in combinatorial biology, which are maps that represent a certain unfolding and refolding of RNA strands.  One main result of this study identifies a subset of local moves that we call $s_i$-local moves, and proves that $s_i$-local moves correspond to the maps $s_i$ acting on standard Young tableaux.  We also prove that the graph of $s_i$-local moves is a connected, graded poset with unique minimal and maximal elements.  We then extend this discussion to functions $s_i^C$ that mimic reflections in the Weyl group of type $C$.  The corresponding graph is no longer connected, but we prove it has two connected components, one of symmetric and the other of asymmetric plane trees.  We give open questions and possible biological interpretations.
\end{abstract}

\maketitle

\section{Introduction}

This paper analyzes a combinatorial question inspired by biology, specifically the mathematical structure of RNA.  RNA  has primary structure (a sequence of letters A, U, C, and G), secondary structure (a partial matching of the letters in the primary structure, indicating how the RNA strand has folded and bonded to itself), and a tertiary structure (how this folding occurs in 3-dimensional space).  All of these structures contribute to the function of the RNA strand in ways that are still being uncovered. While our mathematical model of RNA is motivated by biology, this paper focuses on the model's combinatorial properties rather its direct relationship to biology.

There are many combinatorial models for the secondary structure of RNA, including plane trees and standard Young tableaux of shape $(n,n)$.  We will compare two important operations on these combinatorial objects, one from biological applications and the other from representation theory.

The first operation is called a {\em local move}.  Defined by Condon, Heitsch, and Hoos (and in Definition \ref{definition: local moves}), local moves model unfolding an RNA strand and refolding it differently \cite{Hei}.  Heitsch described key combinatorial statistics of the graph whose vertices are plane trees on $n$ edges and whose edges are local moves; she also showed how this graph is related to other important graphs like an analogous graph whose vertices are noncrossing partitions \cite{Hei}.

The second operation comes from constructions of representations of the symmetric group $S_n$.  One classical construction of representations of $S_n$ uses Young diagrams, which are staircase-shaped collections of boxes.  The symmetric group acts naturally on the set of all fillings of a Young diagram with the integers $1,2,\ldots,n$ (without repeating numbers) just by permuting the numbers.  It turns out that this action on filled Young diagrams gives rise to irreducible representations of $S_n$ (see, e.g.  \cite{Ful97, Sag01} for more).  

We restrict our attention to ``standard" Young tableaux, which are fillings that increase along both rows and columns. These tableaux are known to index bases for the irreducible representations of $S_n$ as well as other quantities of combinatorial interest. It is therefore natural to ask whether the symmetric group can be modified to also act on standard Young tableaux. The answer is yes and no.  In Section \ref{section: maps on tableaux} we define a collection of maps that act on standard Young tableaux and agree as much as possible with the action of the simple transpositions $(i,i+1)$ on arbitrary fillings of Young diagrams.  More precisely, the map corresponding to the simple reflection $(i,i+1)$ simply exchanges $i$ and $i+1$ in the tableau when doing so makes sense.  The maps do not induce a group action of $S_n$ because composition of functions does not agree with multiplication in $S_n$.  Thus these maps cannot directly give information about $S_n$-representations.  However the maps are involutions, as we confirm in Proposition \ref{proposition: involutions}.  Moreover similar maps arise in other parts of combinatorial representation theory, including Vogan's generalized tau invariants \cite{Vog79, HRT}. 

We further restrict our study to the standard Young tableaux corresponding to the partition $(n,n)$.  This partition is an especially important one in applications from geometry \cite{Fun03} to knot theory \cite{Kho04}, as well as the biological applications discussed here.  In Theorem \ref{theorem: graph local moves} we prove that our maps actually correspond to certain local moves, whose defining conditions are shown in Figure~\ref{fig:diffrow}.  We call the local moves that arise in this way {\em{$s_i$-local moves.}} 

Note that not all local moves correspond to the action of permutations of the form $(i,i+1)$.  In particular the graph $G^A$ whose vertices are plane trees and whose edges correspond to $s_i$-local moves is different from the graph whose edges are {\em all} local moves.  The graph of {\em all} local moves is a connected graded poset for which the cardinalities of the ranks form a symmetric, unimodal sequence (see, e.g., \cite{Hei}).  Section \ref{section: graph in type A} proves that the graph $G^A$ is still a
\begin{itemize}
\item connected (Proposition \ref{proposition: connected})
\item graded poset (Proposition \ref{proposition: ranked poset})
\item with a unique minimal element and a unique maximal element (Proposition \ref{proposition: max/min elements}).
\end{itemize}
However the grading of the graph of $s_i$-local moves does not coincide with that of the graph of all local moves, nor does the graph of $s_i$-local moves satisfy symmetry of ranks that the graph of local moves does (see Remark \ref{remark: not unimodal}).

Our $s_i$-local moves were constructed by analogy with the symmetric group $S_n$.  Thus we finish by extending the analogy to Weyl groups of other classical types, which we can do by considering these groups as subgroups of $S_n$. Our main focus is Weyl groups of type $C$, which give rise to type $C$ local moves.  Like Heitsch for local moves, we find that the plane tree model is particularly natural for type $C$ local moves.  Indeed we prove in Corollary \ref{corollary: two components} that the graph $G^C$ of plane trees under type $C$ local moves contains exactly two connected components: one consisting of symmetric plane trees and the other consisting of asymmetric plane trees.

We conclude with a brief discussion of extending $s_i$-local moves to types $D$ and $B$, as well as possible biological interpretations of all the local moves we describe.  We give open questions throughout the manuscript.

Throughout this manuscript  $Y$ denotes standard Young tableaux and $T$ denotes plane trees.

\section{Maps on tableaux corresponding to simple transpositions}\label{section: maps on tableaux}

In this section we describe a set of involutions on the set of standard Young tableaux of shape $(n,n)$ that are indexed by simple reflections.  Our maps are inspired by a well-known $S_n$-action from classical representation theory that gives all irreducible representations of the symmetric group.  Our maps do not generate a group action, as we show in Remark \ref{remark: not group action}.  However because they are involutions, our maps induce a graph whose vertices are the set of standard Young tableaux of shape $(n,n)$ and whose edges correspond to the image under each map.  We define this graph in this section.  In subsequent sections we study combinatorial properties of the graph, prove that these maps agree with operations on plane trees from combinatorial biology, and discuss how to change the Lie type of our maps.

To begin we recall the definition of Young tableaux and sketch their relationship to the representation theory of the symmetric group $S_n$. 

\begin{definition}
Let $\lambda$ be a partition of $n$.  A Young diagram of shape $\lambda$ is a collection of $\lambda_1$ boxes in the top row, $\lambda_2$ boxes in the second row, and so on, aligned on the top and the left.  A standard Young tableau $Y$ of shape $\lambda$ is a filling of the Young diagram with the integers $\{1,2,\ldots, n\}$ without repetition so that each row increases left-to-right and each column increases top-to-bottom.
\end{definition}

The {\em Specht module} for a partition $\lambda$ is generated as a complex vector space by vectors $v_T$ indexed by standard tableaux $Y$ of shape $\lambda$.  The dimension of the irreducible representation of {{$S_n$}} corresponding to $\lambda$ is also the number of standard Young tableaux of shape $\lambda$.  A reasonable question arises: is there an action of {{$S_n$}} on standard Young tableaux under which the Young tableaux themselves can be the basis for the irreducible representation?  Sadly the answer is generally no: the vectors $v_Y$ in the Specht module are linear combinations of terms corresponding to different fillings of $\lambda$.   (See Fulton's text \cite{Ful97} or Sagan's text \cite{Sag01} for more.)  The problem is that {{$S_n$}} ``should" act by permuting the entries of $Y$ but permuting the entries of $Y$ usually doesn't produce another standard tableau.

In the following family of maps, we modify the permutation action on all fillings so that it produces standard tableaux.  We define the maps on standard Young tableaux for arbitrary partitions; in later sections we specialize to the case when $\lambda = (n,n)$ and the maps correspond to elements of $S_{2n}$.  

\begin{definition}\label{definition: group action}
Suppose that $Y$ is a standard Young tableau with $n$ boxes and $s_i = (i, i+1)$ where $i=1,\ldots,{n-1}$ is a simple reflection.  If $i, i+1$ are not in the same row or in the same column of $Y$ then define $s_i(Y)$ to be the tableau with $i$ and $i+1$ exchanged.  If $i, i+1$ are in the same row or in the same column of $Y$ then define $s_i(Y)$ to be $Y$.  Define an arbitrary word $s_{i_1}s_{i_2}\cdots s_{i_k}(Y)$ to be the tableau obtained by composition of maps.
\end{definition}

Others have considered an analogous action on 3-row tableaux \cite{HRT} that comes from Vogan's generalized tau invariant \cite{Vog79}.

The next result shows that these operations always give well-defined maps on standard tableaux (of arbitrary but fixed shape).

\begin{proposition}
For each $i=1,\ldots,n-1$ the map $s_i$ is well-defined and the image $s_i(Y)$ is a standard Young tableau of the same shape as $Y$.
\end{proposition}

\begin{proof}
By construction $s_i$ preserves the shape of $Y$.  By definition the boxes containing $i$ and $i+1$ inside the standard Young tableau $Y$ have numbers less than $i$ to the left and above and have numbers greater than $i+1$ to the right and below.  Hence if $s_i$ exchanges $i$ and $i+1$ then the result $s_i(Y)$ is also a standard Young tableau.
\end{proof}

Moreover these maps have a convenient property.

\begin{proposition} \label{proposition: involutions}
Definition \ref{definition: group action} produces a well-defined involution on the set of standard Young tableaux of shape $\lambda$.
\end{proposition}

\begin{proof}
We check that for all $i$ we have $s_i^2=e$ using two cases:
\begin{enumerate}
\item 
If $i$ and $i+1$ are in the same row then by definition $s_i(Y) = Y$ so the claim holds. 
\item 
If $i$ and $i+1$ are in different rows then $s_i$ swaps the positions of $i$ and $i+1$. Applying $s_i$ twice brings $i$ and $i+1$ back to their original positions.
\end{enumerate}
\end{proof}

This leads us to construct a graph whose vertices are standard Young tableaux of shape $\lambda$ and whose edges describe the maps $s_i$.  The edges are undirected precisely because the maps $s_i$ are involutions for each $i$.

\begin{definition} \label{definition: graph on tableaux}
Let $G_{\lambda}=(V,E)$ be the edge-labeled graph whose vertices $V$ are the set of standard Young tableaux of shape $\lambda$. An edge labeled $s_i$ connects tableaux $Y$ and $Y'$ when $s_i(Y)=Y'$.  We call $G_{\lambda}$ the graph of $s_i$-local moves for $\lambda$.
\end{definition}

As an example, Figure~\ref{fig:S2} shows the graph $G_{(2,2)}$ corresponding to the partition $(2, 2)$. 

%%%%%%%%%%%%%%%%%%%%%%%%%%%%%%%%%%%%%%%%%%%%%%%
%%%%%%%%%%%%%%%% S2 diagram begins %%%%%%%%%%%%%%%%%%
%%%%%%%%%%%%%%%%%%%%%%%%%%%%%%%%%%%%%%%%%%%%%%%
\begin{figure}[H]
\begin{center}
\begin{tikzpicture}[align=center, anchor=base]
    % set up node
    \node[minimum size=8em] (a) at (0, 0) {
    \tiny
    \begin{tabular}{ | c | c |}
    \hline
    $1$ & $2$ \\ \hline
    $3$ & $4$ \\ \hline
    \end{tabular}
    };
    % set up node
    \node[minimum size=8em] (b) at (4, 0) {
    \tiny
    \begin{tabular}{ | c | c |}
    \hline
    $1$ & $3$ \\ \hline
    $2$ & $4$ \\ \hline
    \end{tabular}
    };
    
    % Graph edges
    % Left loops
    \draw (-1, 0.25) edge[dashed, loop, distance=2cm, in=135, out=180] node (s1L) [left] {\tiny $s_{1}$} 
    (-1, 0.5); 
    \draw (-1, -0.25) edge[dashed, loop, distance=2cm, in=225, out=180] node (s3L) [left] {\tiny $s_{3}$} 
    (-1, -0.5); 
    
    % Middle edge
    \draw (1, 0) edge[-] node (s2) [above] {\tiny $s_{2}$} (3, 0);
    
    % Right loops
    \draw (5, 0.25) edge[dashed, loop, distance=2cm, in=45, out=0] node (s1R) [right] {\tiny $s_{1}$} 
    (5, 0.5); 
    \draw (5, -0.25) edge[dashed, loop, distance=2cm, in=315, out=0] node (s3R) [right] {\tiny $s_{3}$} 
    (5, -0.5); 
\end{tikzpicture}
\end{center}
\caption{Graph of $s_{i} = (i, i + 1)$ on standard Young tableaux of shape $(2, 2)$ \label{fig:S2}}
\end{figure}
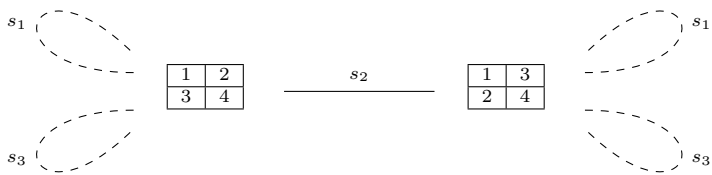
%%%%%%%%%%%%%%%%%%%%%%%%%%%%%%%%%%%%%%%%%%%%%%%
%%%%%%%%%%%%%%%%%% S2 diagram ends %%%%%%%%%%%%%%%%%%%%
%%%%%%%%%%%%%%%%%%%%%%%%%%%%%%%%%%%%%%%%%%%%%%%

\begin{remark}\label{remark: not group action}
Note that the maps $s_i$ do not induce a group action of the symmetric group on the standard Young tableaux even for the shape $(n,n)$.  For a counterexample, inspect Figure~\ref{fig:S2}.  On the one hand $s_2s_3s_2(Y)=Y$ for each standard tableau $Y$ of shape $(2,2)$.  On the other hand $s_3s_2s_3(Y)$ is the opposite tableau of shape $(2,2)$.  Since $s_2s_3s_2=s_3s_2s_3$ in the symmetric group, we conclude that the maps $s_i$ do not define a group action.
\end{remark}

\begin{remark}
We typically omit all edges corresponding to fixed points $Y = s_i(Y)$ (represented in Figure~\ref{fig:S2} as dashed self-edges) from our drawings of $G_{\lambda}$. In later sections we restrict to the case $\lambda = (n,n)$ and so omit $\lambda$ from our notation.  We will also modify the maps $s_i$ that define the edges, so we often write $G^A$ to denote the graph with the precise edges in Definition \ref{definition: graph on tableaux} or write $G^C$ to denote the modified graph in Section \ref{section: type C}.
\end{remark}

\begin{question}
In subsequent sections we analyze the graph $G_{(n,n)}$.  What can be said about the graph $G_{\lambda}$ for arbitrary partitions?
\end{question}

%Add questions: what can we say about graphs $G_{\lambda}$ for other shapes $\lambda$?  Are the leading terms of $v_T$ related to $w*T$ in the two-row case?

\section{$S_n$-action and local moves on plane trees}

This section relates the functions defined in the previous section to an operation on plane trees called {\em local moves}.  Condon, Heitsch, and Hoos defined local moves to represent an unfolding-and-refolding process on a strand of RNA.  Heitsch then proved many combinatorial properties of a graph whose vertices are plane trees and whose edges come from local moves, for instance that the graph is symmetric and unimodal \cite{Hei}.  She also showed that under one natural modification to the edges, we obtain the graph whose vertices are noncrossing partitions and whose edges come from Kreweras complementation \cite{Hei}.

We extend these results in a different direction, showing that many local moves correspond naturally to the action of the maps $s_i$ on standard Young tableaux.  Since we specialize to Young diagrams of shape $(n,n)$ we also specialize to the permutations $S_{2n}$ in this section.

We begin by recalling the definition of plane trees and local moves.

\begin{definition} \label{definition: plane tree}
A plane tree is a rooted tree whose subtrees at any vertex are linearly ordered. % (Heitsch)
\end{definition}

Our convention for a plane tree is that the root is at the top and that the subtrees are linearly ordered from left to right. In figures, the root is drawn with an open circle and ordinary vertices are drawn with solid circles.

Plane trees are related to Young diagrams, noncrossing matchings, and other fundamental combinatorial objects that are also counted by Catalan numbers.  To see this, we interpret each edge of a plane tree with $n$ edges as a pair of two half-edges, each of which is indexed with one of the integers from $1$ to $2n$.  The half-edges are labeled in increasing order counterclockwise from the root.  We write $e(i,j)$ to denote the edge whose left half-edge is labeled $i$ and whose right half-edge is labeled $j$.  Given this setup, the half-edges $i$ and $j$ in the edge $e(i,j)$ satisfy many constraints, including $i<j$.

The next definition describes {\em local moves}, which are operations on plane trees that are central to this paper. We denote the collection of plane trees with $n$ edges by $\mathcal{T}_n$.

\begin{definition} \label{definition: local moves}
A local move on a plane tree $T \in \mathcal{T}_n$ converts a pair of adjacent edges in one of two ways:
\begin{enumerate}
\item \label{type A local move} If $i<i'<j'<j$ then replace $e(i,j)$ and $e(i',j')$ with $e(i,i')$ and $e(j',j)$.  This is a local move {\em of type \eqref{type A local move}}:
%%%%%%%%%%%%%%%%%%%%%%%%%%%%%%%%%%%%%%%%%%%%%%%
%%%%%%%%%%%%%%%%%% Type (1) local move diagram begins %%%%%%%%%%%
%%%%%%%%%%%%%%%%%%%%%%%%%%%%%%%%%%%%%%%%%%%%%%%
\begin{figure}[H]
\begin{center}
$
\begin{aligned}
  \begin{tikzpicture}[level distance=1cm,
level 1/.style={sibling distance=2cm},
level 2/.style={sibling distance=1cm}]
\tikzstyle{every node}=[circle, draw, scale=.8, inner sep=2pt]
\vspace{10pt}
\node (Root)[fill] {}
    child {
    node[fill] {}
    child { node[fill] {} 
	edge from parent 
	node[left, draw=none] {\tiny $i'$}
	node[right,draw=none]  {\tiny $j'$}    
    }
	edge from parent 
	node[left, draw=none] {\tiny $i$}
	node[right, draw=none]  {\tiny $j$}
};
\end{tikzpicture}
\end{aligned}
%%%%%%%%%%%%%%%%%%%% Arrow begins %%%%%%%%%%%%%%%%%%
\quad 
\begin{aligned}
\begin{tikzpicture}
  [scale = 1.5]
  \node (n1) at (0, 0) {};
  \node (n2) at (1, 0) {};
  \draw[->] (n1) -- (n2) node[midway, above]  {\tiny Type (1) local move};
\end{tikzpicture}
\end{aligned}
\quad
%%%%%%%%%%%%%%%%%%%% Arrow ends %%%%%%%%%%%%%%%%%%%%
 \begin{aligned}
  \begin{tikzpicture}[sloped, level distance=1cm,
level 1/.style={sibling distance=2cm},
level 2/.style={sibling distance=1cm}]
\tikzstyle{every node}=[circle, draw, scale=.8, inner sep=2pt]
\vspace{10pt}
\node[inner sep=2pt] (Root)[fill] {}
    child {
    node[fill] {} 
    	edge from parent 
	node[ellipse, above, draw=none] {\tiny $i$}
	node[ellipse, below, draw=none]  {\tiny $i'$}  
}
child {
    node[fill] {}
    	edge from parent 
	node[ellipse, below, draw=none] {\tiny $j'$}
	node[ellipse, above, draw=none]  {\tiny $j$}  
};
\end{tikzpicture}
\end{aligned}
$
\end{center}
%\caption{Type (1) local move \label{fig:type1}}
\end{figure}
%%%%%%%%%%%%%%%%%%%%%%%%%%%%%%%%%%%%%%%%%%%%%%%
%%%%%%%%%%%%%%%%%% Type (1) local move diagram ends %%%%%%%%%%%%
%%%%%%%%%%%%%%%%%%%%%%%%%%%%%%%%%%%%%%%%%%%%%%%
\item \label{type B local move} If $i<j<i'<j'$ then replace $e(i,j)$ and $e(i',j')$ with $e(i,j')$ and $e(j,i')$.  This is a local move {\em of type \eqref{type B local move}}: 
%%%%%%%%%%%%%%%%%%%%%%%%%%%%%%%%%%%%%%%%%%%%%%%
%%%%%%%%%%%%%%%%%% Type (2) local move diagram begins %%%%%%%%%%%
%%%%%%%%%%%%%%%%%%%%%%%%%%%%%%%%%%%%%%%%%%%%%%%
\begin{figure}[H]
\begin{center}
$
\begin{aligned}
\begin{tikzpicture}[sloped, level distance=1cm,
level 1/.style={sibling distance=2cm},
level 2/.style={sibling distance=1cm}]
\tikzstyle{every node}=[circle, draw, scale=.8, inner sep=2pt]
\vspace{10pt}
\node[inner sep=2pt] (Root)[fill] {}
    child {
    node[fill] {} 
    	edge from parent 
	node[ellipse, above, draw=none] {\tiny $i$}
	node[ellipse, below, draw=none]  {\tiny $j$}  
}
child {
    node[fill] {}
    	edge from parent 
	node[ellipse, below, draw=none] {\tiny $i'$}
	node[ellipse, above, draw=none]  {\tiny $j'$}  
};
\end{tikzpicture}
\end{aligned}
%%%%%%%%%%%%%%%%%%%% Arrow begins %%%%%%%%%%%%%%%%%%
\quad 
\begin{aligned}
\begin{tikzpicture}
  [scale = 1.5]
  \node (n1) at (0, 0) {};
  \node (n2) at (1, 0) {};
  \draw[->] (n1) -- (n2) node[midway, above]  {\tiny Type (2) local move};
\end{tikzpicture}
\end{aligned}
\quad
%%%%%%%%%%%%%%%%%%%% Arrow ends %%%%%%%%%%%%%%%%%%%%
 \begin{aligned}
 \begin{tikzpicture}[level distance=1cm,
level 1/.style={sibling distance=2cm},
level 2/.style={sibling distance=1cm}]
\tikzstyle{every node}=[circle, draw, scale=.8, inner sep=2pt]
\vspace{10pt}
\node (Root)[fill] {}
    child {
    node[fill] {}
    child { node[fill] {} 
	edge from parent 
	node[left, draw=none] {\tiny $j$}
	node[right,draw=none]  {\tiny $i'$}    
    }
	edge from parent 
	node[left, draw=none] {\tiny $i$}
	node[right, draw=none]  {\tiny $j'$}
};
\end{tikzpicture}
 \end{aligned}
$
\end{center}
%\caption{Type (2) local move \label{fig:type2}}
\end{figure}
%%%%%%%%%%%%%%%%%%%%%%%%%%%%%%%%%%%%%%%%%%%%%%%
%%%%%%%%%%%%%%%%%% Type (2) local move diagram ends %%%%%%%%%%%%
%%%%%%%%%%%%%%%%%%%%%%%%%%%%%%%%%%%%%%%%%%%%%%%
\end{enumerate}

\end{definition}

The following map provides a natural bijection between plane trees with $n$ edges and standard Young tableaux of shape $(n,n)$.

\begin{definition}\label{definition: plane trees and tableaux}
Let $Y_{(n,n)}$ denote the set of standard Young tableaux of shape $(n,n)$.  Define a map $\phi: \mathcal{T}_n \rightarrow Y_{(n,n)}$ by the rule that for each $T \in \mathcal{T}_n$ the Young tableau $\phi(T)$ has the labels of the left-half-edges of $T$ on its top row and the labels of the right-half-edges of $T$ on its bottom row.
\end{definition}

The following proposition confirms that the map $\phi$ is bijective.  Both image and domain are sets that are known to index the Catalan numbers \cite[Chapter 6, Problem 19(e) and (ww)]{Sta99}; we include the following proof  to confirm that the specific map $\phi$ is a direct bijection.

\begin{proposition} \label{proposition: bijection}
The map $\phi: \mathcal{T}_n \rightarrow Y_{(n,n)}$ is a well-defined bijection.
\end{proposition}

\begin{proof}
The half-edges of a plane tree are labeled counterclockwise, so for each $k$ there are at least as many left-edges $i$ with $i \leq k$ as right edges $j$ with $j \leq k$.  Thus if $i$ is above $j$ in a column of the Young tableau $\phi(T)$ then $i<j$.  It follows that $\phi$ is well-defined.

If $\phi(T)=\phi(T')$ then both $T$ and $T'$ have the same set of left half-edges and the same set of right half-edges.  Since by definition every subtree of a plane tree is linearly ordered, the indexing of the half-edges determines the plane tree.  So $\phi$ is injective.

The sets $\mathcal{T}_n$ and $Y_{(n,n)}$ have the same cardinality so the map $\phi$ is a bijection, as desired.
\end{proof}
 
In order to prove our main result, we need more precise information about the fragments of a plane tree that correspond to the boxes filled with $i$ and $i+1$ in a standard Young tableau.  The next lemma compiles this information.

\begin{lemma} \label{lemma: edges for i,i+1}
Consider a standard Young tableau $Y$ of shape $(n,n)$ and its preimage $\phi^{-1}(Y)$ under the bijection in Definition \ref{definition: plane trees and tableaux}.  The half-edges corresponding to $i$ and $i+1$ are in one of the following relative positions:
\begin{enumerate}[(i)]
\item The numbers $i$ and $i+1$ are on the same row in $Y$ if and only if $i$ and $i+1$ label the half-edges of $\phi^{-1}(Y)$ in one of the following ways (Figure~\ref{fig:samerow}).

%%%%%%%%%%%%%%%%%%%%%%%%%%%%%%%%%%%%%%%%%%%%%%%%
%%%%%%%%%%%%%%%%%% Same row diagram begins %%%%%%%%%%%%%%%%%%%
%%%%%%%%%%%%%%%%%%%%%%%%%%%%%%%%%%%%%%%%%%%%%%%%
\begin{figure}[H]
\begin{center}
\begin{subfigure}[b]{0.45\textwidth}  
\centering
\begin{tikzpicture}[level distance=1cm,
level 1/.style={sibling distance=2cm},
level 2/.style={sibling distance=1cm}]
\tikzstyle{every node}=[circle, draw, scale=.8, inner sep=2pt]
\vspace{10pt}
\node (Root)[fill] {}
    child {
    node[fill] {}
    child { node[fill] {} 
	edge from parent 
	node[left, draw=none] {\tiny $i + 1$}
	node[right,draw=none]  {\tiny $j'$}    
    }
	edge from parent 
	node[left, draw=none] {\tiny $i$}
	node[right, draw=none]  {\tiny $j$}
};
\end{tikzpicture}  
\caption{$i$ and $i+1$ label left half-edges \label{fig:samerowleft}}   
\end{subfigure}
\begin{subfigure}[b]{0.45\textwidth}
\centering
  \begin{tikzpicture}[level distance=1cm,
level 1/.style={sibling distance=2cm},
level 2/.style={sibling distance=1cm}]
\tikzstyle{every node}=[circle, draw, scale=.8, inner sep=2pt]
\vspace{10pt}
\node (Root)[fill] {}
    child {
    node[fill] {}
    child { node[fill] {} 
	edge from parent 
	node[left, draw=none] {\tiny $j'$}
	node[right, draw=none]  {\tiny $i$}    
    }
	edge from parent 
	node[left, draw=none] {\tiny $j$}
	node[right, draw=none]  {\tiny $i+1$}
};
\end{tikzpicture}
\caption{$i$ and $i+1$ label right half-edges \label{fig:samerowright}}   
\end{subfigure}
\end{center}
\caption{$i$ and $i+1$ on same row}\label{fig:samerow}
\end{figure}
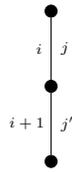
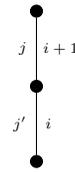
%%%%%%%%%%%%%%%%%%%%%%%%%%%%%%%%%%%%%%%%%%%%%%%%
%%%%%%%%%%%%%%%%%% Same row diagram ends %%%%%%%%%%%%%%%%%
%%%%%%%%%%%%%%%%%%%%%%%%%%%%%%%%%%%%%%%%%%%%%%%%

\item The numbers $i$ and $i+1$ are on opposite rows in $Y$ if and only if in $\phi^{-1}(Y)$ either $i$ and $i+1$ label a leaf (Figure~\ref{fig:diffrowleaf}) or the interior of a peak (Figure~\ref{fig:diffrowV}).

%%%%%%%%%%%%%%%%%%%%%%%%%%%%%%%%%%%%%%%%%%%%%%%%
%%%%%%%%%%%%%%%%%% Different row diagram begins %%%%%%%%%%%%%%%
%%%%%%%%%%%%%%%%%%%%%%%%%%%%%%%%%%%%%%%%%%%%%%%%
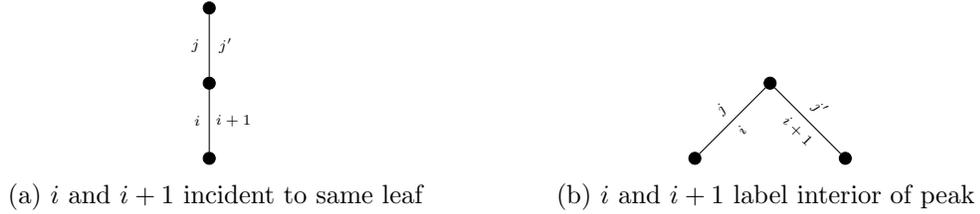
\begin{figure}[H]
\begin{center}
\begin{subfigure}[b]{0.5\textwidth}  
\centering
  \begin{tikzpicture}[level distance=1cm,
level 1/.style={sibling distance=2cm},
level 2/.style={sibling distance=1cm}]
\tikzstyle{every node}=[circle, draw, scale=.8, inner sep=2pt]
\vspace{10pt}
\node (Root)[fill] {}
    child {
    node[fill] {}
    child { node[fill] {} 
	edge from parent 
	node[left, draw=none] {\tiny $i$}
	node[right,draw=none]  {\tiny $i + 1$}    
    }
	edge from parent 
	node[left, draw=none] {\tiny $j$}
	node[right, draw=none]  {\tiny $j'$}
};
\end{tikzpicture}
\caption{$i$ and $i+1$ incident to same leaf \label{fig:diffrowleaf}}
\end{subfigure}
\begin{subfigure}[b]{0.45\textwidth}  
\centering
  \begin{tikzpicture}[sloped, level distance=1cm,
level 1/.style={sibling distance=2cm},
level 2/.style={sibling distance=1cm}]
\tikzstyle{every node}=[circle, draw, scale=.8, inner sep=2pt]
\vspace{10pt}
\node[inner sep=2pt] (Root)[fill] {}
    child {
    node[fill] {} 
    	edge from parent 
	node[ellipse, above, draw=none] {\tiny $j$}
	node[ellipse, below, draw=none]  {\tiny $i$}  
}
child {
    node[fill] {}
    	edge from parent 
	node[ellipse, below, draw=none] {\tiny $i + 1$}
	node[ellipse, above, draw=none]  {\tiny $j'$}  
};
\end{tikzpicture}
\caption{$i$ and $i+1$ label interior of peak \label{fig:diffrowV}} 
\end{subfigure}
\end{center}
\caption{$i$ and $i + 1$ on different rows \label{fig:diffrow}}
\end{figure}
%%%%%%%%%%%%%%%%%%%%%%%%%%%%%%%%%%%%%%%%%%%%%%%%
%%%%%%%%%%%%%%%%%% Different row diagram ends %%%%%%%%%%%%%%%%
%%%%%%%%%%%%%%%%%%%%%%%%%%%%%%%%%%%%%%%%%%%%%%%%

\item The numbers $i$ and $i+1$ are on the same column in $Y$ if and only if $i$ and $i+1$ label a leaf incident to the root in $\phi^{-1}(Y)$ (Figure~\ref{fig:samecolumn}).
%%%%%%%%%%%%%%%%%%%%%%%%%%%%%%%%%%%%%%%%%%%%%%%%
%%%%%%%%%%%%%%%%%% Same column diagram begins %%%%%%%%%%%%%%
%%%%%%%%%%%%%%%%%%%%%%%%%%%%%%%%%%%%%%%%%%%%%%%%
\begin{figure}[H]
\begin{center}
  \begin{tikzpicture}[level distance=1cm,
level 1/.style={sibling distance=2cm},
level 2/.style={sibling distance=1cm}]
\tikzstyle{every node}=[circle, draw, scale=.8, inner sep=2pt]
\vspace{10pt}
\node (Root)[] {}
    child { node[fill] {} 
	edge from parent 
	node[left, draw=none] {\tiny $i$}
	node[right, draw=none]  {\tiny $i + 1$}    
    };
\end{tikzpicture}
\end{center}
\caption{$i$ and $i + 1$ on same column \label{fig:samecolumn}}
\end{figure}
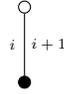
%%%%%%%%%%%%%%%%%%%%%%%%%%%%%%%%%%%%%%%%%%%%%%%%
%%%%%%%%%%%%%%%%%% Same column diagram ends %%%%%%%%%%%%%%%
%%%%%%%%%%%%%%%%%%%%%%%%%%%%%%%%%%%%%%%%%%%%%%%%
\end{enumerate}
In no case is there an additional half-edge incident to the vertex between $i$ and $i+1$.
\end{lemma}

\begin{proof}
By convention, plane trees are labeled counterclockwise from the root. Hence there can be no edges or half-edges on the vertex incident to both $i$ and $i+1$. We think of each edge $e(i,j)$ as having a left half-edge labeled $i$ and a right half-edge labeled $j$.

\begin{enumerate}[(i)]
\item
Consider the case where the numbers $i$ and $i+1$ are on the same row in $Y$. By definition of $\phi$ the top row of the Young tableau has the labels on the left half-edges of the corresponding plane tree while the bottom row has the labels on the right half-edges. Suppose $i$ and $i+1$ are on the top row of the Young tableau. Then $i$ and $i+1$ are left half-edges and must be in the configuration shown in Figure~\ref{fig:samerowleft}. Suppose $i$ and $i+1$ are on the bottom row of the Young tableau. Then $i$ and $i+1$ are right half-edges and must be in the configuration shown in Figure~\ref{fig:samerowright}.

\item
Consider the case where the numbers $i$ and $i+1$ are on different rows in $Y$. Suppose $i$ is on the top row and $i+1$ is on the bottom row. Then $i$ is a left half-edge and $i+1$ is a right half-edge. That means these two numbers will label the same leaf in the tree, as shown in Figure~\ref{fig:diffrowleaf}. Now suppose $i+1$ is in the top row and $i$ is in the bottom row of $Y$. Then $i$ labels a right half-edge and $i+1$ labels a left half-edge. In a plane tree, this configuration must be a peak with $i$ and $i+1$ labeling the interior, as shown in Figure~\ref{fig:diffrowV}.

\item
The numbers $i$ and $i+1$ are on the same column of $Y$ if and only if the first $\frac{i-1}{2}$ columns of $Y$ form a standard Young tableau of size $(\frac{i-1}{2}, \frac{i-1}{2})$ and filled with the numbers $1, 2, \ldots, i-1$.  By restricting $\phi$ to plane trees on $\frac{i-1}{2}$ edges we note that the first $\frac{i-1}{2}$ edges of the plane tree $\phi^{-1}(Y)$ form a subtree with the same root as $\phi^{-1}(Y)$.  This is equivalent to saying that $i-1$ labels the right half of an edge incident to the root, which is true if and only if $i$ and $i+1$ label the half-edges of a leaf incident to the root, as shown in Figure~\ref{fig:samecolumn}.
\end{enumerate}
\end{proof}

In the next theorem we use Lemma~\ref{lemma: edges for i,i+1} to show that if $i$ and $i+1$ are in different rows (but not in the same column) of a standard Young tableau then the action of the map $s_i$ on the tableau corresponds to a local move on the corresponding plane tree.  Henceforth the maps $s_i$ vary from $i=1$ to $i=2n-1$ since there are $2n$ boxes in the Young diagram.

\begin{theorem} \label{theorem: graph local moves}
Consider a plane tree $T$ and its image $Y=\phi(T)$ under the bijection in Definition \ref{definition: plane trees and tableaux}.  The half-edges in $T$ labeled $i$ and $i+1$ are in one of the two relative positions in Figure~\ref{fig:diffrow} if and only if the local move on edges with half-edges $j<i<i+1<j'$ produces the plane tree $\phi^{-1}(s_i (Y))$. 
\end{theorem}

\begin{proof}
Lemma \ref{lemma: edges for i,i+1} showed  that $i$ and $i+1$ are on different rows and different columns exactly when $i$ and $i+1$ are in the configurations in Figure~\ref{fig:diffrow}. In fact, local moves exchange these two configurations because $j<i<i+1<j'$ as shown in Figure~\ref{fig:simpleLM}.  
%%%%%%%%%%%%%%%%%%%%%%%%%%%%%%%%%%%%%%%%%%%%%%%
%%%%%%%%%%%%%%%%%% Simple local move diagram begins %%%%%%%%%%%
%%%%%%%%%%%%%%%%%%%%%%%%%%%%%%%%%%%%%%%%%%%%%%%
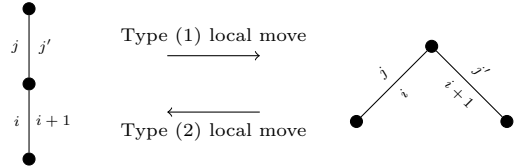
\begin{figure}[H]
\begin{center}
$
\begin{aligned}
  \begin{tikzpicture}[level distance=1cm,
level 1/.style={sibling distance=2cm},
level 2/.style={sibling distance=1cm}]
\tikzstyle{every node}=[circle, draw, scale=.8, inner sep=2pt]
\vspace{10pt}
\node (Root)[fill] {}
    child {
    node[fill] {}
    child { node[fill] {} 
	edge from parent 
	node[left, draw=none] {\tiny $i$}
	node[right,draw=none]  {\tiny $i + 1$}    
    }
	edge from parent 
	node[left, draw=none] {\tiny $j$}
	node[right, draw=none]  {\tiny $j'$}
};
\end{tikzpicture}
\end{aligned}
%%%%%%%%%%%%%%%%%%%% Arrow begins %%%%%%%%%%%%%%%%%%
\quad 
\begin{aligned}
\begin{tikzpicture}
  [scale = 1.5]
  \node (n1a) at (0, 0) {};
  \node (n2a) at (1, 0) {};
  \node (n1b) at (0, -0.5) {};
  \node (n2b) at (1, -0.5) {};
  \draw[->] (n1a) -- (n2a) node[midway, above]  {\tiny Type (1) local move};
  \draw[->] (n2b) -- (n1b) node[midway, below]  {\tiny Type (2) local move};
\end{tikzpicture}
\end{aligned}
\quad
%%%%%%%%%%%%%%%%%%%% Arrow ends %%%%%%%%%%%%%%%%%%%%
 \begin{aligned}
  \begin{tikzpicture}[sloped, level distance=1cm,
level 1/.style={sibling distance=2cm},
level 2/.style={sibling distance=1cm}]
\tikzstyle{every node}=[circle, draw, scale=.8, inner sep=2pt]
\vspace{10pt}
\node[inner sep=2pt] (Root)[fill] {}
    child {
    node[fill] {} 
    	edge from parent 
	node[ellipse, above, draw=none] {\tiny $j$}
	node[ellipse, below, draw=none]  {\tiny $i$}  
}
child {
    node[fill] {}
    	edge from parent 
	node[ellipse, below, draw=none] {\tiny $i + 1$}
	node[ellipse, above, draw=none]  {\tiny $j'$}  
};
\end{tikzpicture}
\end{aligned}
$
\end{center}
\caption{Edges of plane tree under local moves \label{fig:simpleLM}}
\end{figure}
%%%%%%%%%%%%%%%%%%%%%%%%%%%%%%%%%%%%%%%%%%%%%%%
%%%%%%%%%%%%%%%%%% Simple local move diagram ends %%%%%%%%%%%%
%%%%%%%%%%%%%%%%%%%%%%%%%%%%%%%%%%%%%%%%%%%%%%%
Let $T'$ denote the image of $T$ under the allowed local move on half-edges $j<i<i+1<j'$ and let $Y'=\phi(T')$.  Comparing $T$ and $T'$ in Figure \ref{fig:simpleLM} shows that $i$ and $i+1$ change from a left half-edge to a right half-edge or vice versa.  Thus $i$ is on the opposite row in $Y$ as it is in $Y'$ and similarly for $i+1$.  By inspection of Figure \ref{fig:simpleLM}, both $j$ and $j'$ stay on the same respective halves of their shared edge in $T$ and $T'$. By definition, a local move changes only the two edges involved in the local move. Thus all other numbers remain on the same rows in the corresponding Young tableau, and so every other integer is in the same row in $Y$ as it is in $Y'$.  Finally $i$ and $i+1$ are on opposite rows in $Y$ by the hypotheses of the theorem together with Lemma \ref{lemma: edges for i,i+1}.  Thus $Y' = s_i(Y)$.  

Conversely suppose there is a local move involving the half-edges $j<i<i+1<j'$.  The configurations in Figure \ref{fig:simpleLM} are the only possibilities listed in Lemma~\ref{lemma: edges for i,i+1} that satisfy these inequalities.  The claim follows.
\end{proof}

\begin{samepage}
\begin{remark}\label{remark: not all local moves are perms}
Not every local move corresponds to one of the maps $s_i$. If $i$ and $i+1$ are on the same row or column of a tableau then $s_i$ fixes the tableau. Otherwise $s_i$ describes the local moves in Figure~\ref{fig:simpleLM} . But when $n >2$ a local move may be described by a transposition between $i$ and $i+k$ with $1<k<2n-i$ in the tableau. Figure~\ref{fig:commutative} gives an example.  The original tableau $Y$ has $1, 2, 3$ along its top row, so every transposition except $s_3$ fixes $Y$.  However the associated plane tree has a local move affecting the half-edges $1$, $2$, $5$, $6$ that corresponds to exchanging 2 and 5 in the tableau.  The tableau resulting from this local move differs both from the original tableau $Y$ and from $s_3(Y)$. 
\end{remark}

%%%%%%%%%%%%%%%%%%%%%%%%%%%%%%%%%%%%%%%%%%%%%%%
%%%%%%%%%%%%%%%% Commutative diagram begins %%%%%%%%%%%%%%%%
%%%%%%%%%%%%%%%%%%%%%%%%%%%%%%%%%%%%%%%%%%%%%%%
\begin{figure}[H]
\begin{center}
\begin{tikzpicture}[scale = 0.85, every node/.style={scale=0.85}]
    %%%%%%%%%%%%%%%%%%%%
    % set up node for path tree
    \node[minimum size=8em] (Mtree) at (0, 0) {
    \begin{tikzpicture}[level distance=1cm,
    level 1/.style={sibling distance=2cm},
    level 2/.style={sibling distance=1cm}]
    \tikzstyle{every node}=[circle, draw, scale=.8, inner sep=2pt, minimum size=0]
    \vspace{10pt}
    \node (Root)[] {}
    child {
    node[fill] {}
    child { 
    node[fill] {} 
    child {
    node[fill] {}
	edge from parent 
	node[left, draw=none] {\tiny $3$}
	node[right, draw=none]  {\tiny $4$}        
    }
	edge from parent 
	node[left, draw=none] {\tiny $2$}
	node[right, draw=none]  {\tiny $5$}    
    }
	edge from parent 
	node[left, draw=none] {\tiny $1$}
	node[right, draw=none]  {\tiny $6$}
    };
    \end{tikzpicture}
    };
    %%%%%%%%%%%%%%%%%%%%
    % set up node for leggy tree
    \node[minimum size=8em] (Rtree) at (6, 0) {
    \begin{tikzpicture}[sloped, level distance=1cm,
    level 1/.style={sibling distance=2cm},
    level 2/.style={sibling distance=2cm}]
    \tikzstyle{every node}=[circle, draw, scale=.8, inner sep=2pt, minimum size=0]
    \vspace{10pt}
    \node[inner sep=2pt] (Root)[] {}
    child {
    node[fill] {} 
    	child {
    	node[fill] {}
    	edge from parent 
	node[ellipse, above, draw=none]  {\tiny $2$}
	node[ellipse, below, draw=none] {\tiny $3$}  
    	} 
	child {
    	node[fill] {} 
    	edge from parent 
	node[ellipse, below, draw=none] {\tiny $4$}
	node[ellipse, above, draw=none]  {\tiny $5$}  
    	}
    	edge from parent 
	node[left, draw=none, rotate=90] {\tiny $1$}
	node[right, draw=none, rotate=90]  {\tiny $6$}  
    };
    \end{tikzpicture}
    };
    %%%%%%%%%%%%%%%%%%%%
    % set up node for star tree
    \node[minimum size=8em] (Ltree) at (-6, 0) {
    \begin{tikzpicture}[sloped, level distance=1cm,
    level 1/.style={sibling distance=1.3cm},
    level 2/.style={sibling distance=1cm}]
    \tikzstyle{every node}=[circle, draw, scale=.8, inner sep=2pt, minimum size=0]
    \vspace{10pt}
    \node[inner sep=2pt] (Root)[] {}
    child {
    node[fill] {} 
    	edge from parent 
	node[ellipse, above, draw=none] {\tiny $1$}
	node[ellipse, below, draw=none]  {\tiny $2$}  
    }
    child {
    node[fill] {} 
    	edge from parent 
	node[ellipse, left, draw=none, rotate=90] {\tiny $3$}
	node[ellipse, right, draw=none, rotate=90]  {\tiny $4$}  
    }
	child {
    	node[fill] {}
    	edge from parent 
	node[ellipse, below, draw=none] {\tiny $5$}
	node[ellipse, above, draw=none]  {\tiny $6$}  
    };
    \end{tikzpicture}
    };
    %%%%%%%%%%%%%%%%%%%%
    %%%%%%%%%%%%%%%%%%%%
    \node[minimum size=5em, inner sep = 25] (Mbox) at (0, -4.85) {
    \tiny
    \begin{tabular}{ | c | c | c |}
    \hline
    1 & 2 & 3 \\ \hline
    4 & 5 & 6 \\ \hline
    \end{tabular}
    };
    %%%%%%%%%%%%%%%%%%%%
    \node[minimum size=5em, inner sep = 25] (Rbox) at (6, -4.85) {
    \tiny  
    \begin{tabular}{ | c | c | c |}
    \hline
    1 & 2 & 4 \\ \hline
    3 & 5 & 6 \\ \hline
    \end{tabular}
    };
    %%%%%%%%%%%%%%%%%%%%
    \node[minimum size=5em] (Lbox) at (-6, -4.85) {
    \tiny  
    \begin{tabular}{ | c | c | c |}
    \hline
    1 & 3 & 5 \\ \hline
    2 & 4 & 6 \\ \hline
    \end{tabular}
    };
    %%%%%%%%%%%%%%%%%%%%
    %%%%%%%%%%%%%%%%%%%%
        
    % Commutative edges
    % phi
    \draw[->] (0, -2)--(Mbox) node [midway, right] {\tiny $\phi$}; % Middle
    \draw[->] (6, -2)--(Rbox) node [midway, right] {\tiny $\phi$}; % Right
    \draw[->] (-6, -2)--(Lbox) node [midway, right] {\tiny $\phi$}; % Left
    
    % Graph edges
    % Loop
    \draw (-0.2, -5.7) edge[dashed, loop, distance=2.25cm, out=235, in=305] node (loop) [below] {\tiny $s_{1}$, $s_{2}$, $s_{4}$ ,$s_{5}$} (0.2, -5.7);   
    % s_3
     \draw[-] (Mbox)--(Rbox) node [midway, above] {\tiny $s_{3}$};
    % Type (1) local move
    \draw[shorten >=0.125cm, shorten <=0.125cm, ->] (Mtree)--(Ltree) node [midway, above] (lml) {\tiny Type (1) local move};
    \draw[shorten >=0.125cm, shorten <=0.125cm, ->] (Mtree)--(Rtree) node [midway, above] (lmr) {\tiny Type (1) local move};

\end{tikzpicture}
\end{center}
\caption{A local move that does not correspond to a permutation $s_i$ \label{fig:commutative}}
\end{figure}
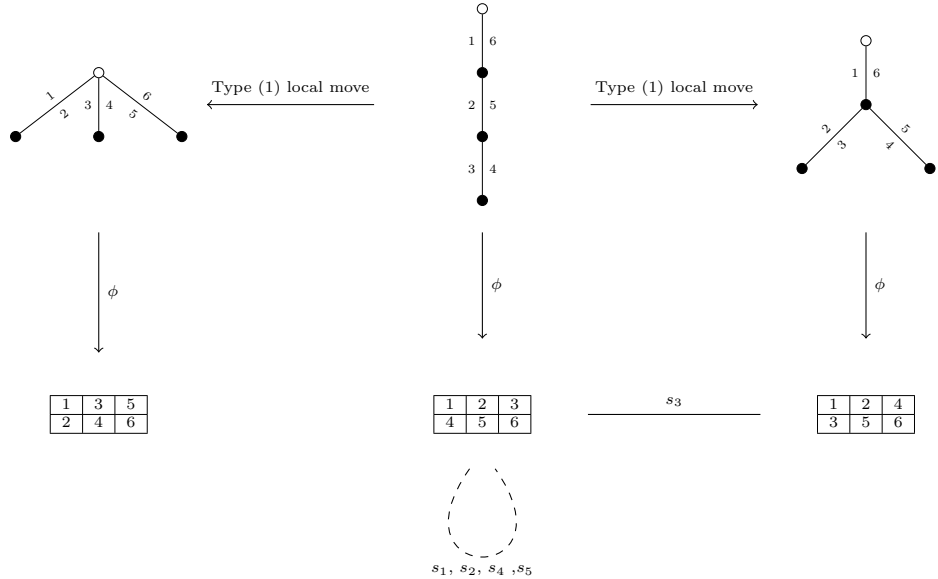
%%%%%%%%%%%%%%%%%%%%%%%%%%%%%%%%%%%%%%%%%%%%%%%
%%%%%%%%%%%%%%%% Commutative diagram ends %%%%%%%%%%%%%%%%%
%%%%%%%%%%%%%%%%%%%%%%%%%%%%%%%%%%%%%%%%%%%%%%%
\end{samepage}

To avoid this ambiguity we have the following definition.  
 
\begin{definition}
Suppose $T$ is a plane tree with $n$ edges whose associated standard Young tableau is $\phi(T)=Y$.
An $s_i$-local move is a local move that is consistent with one of the maps $s_i$ in the sense that the local move sends $T$ to $\phi^{-1}(s_{i}(Y))$ for some $s_i$ with $i=1,2,\ldots,2n-1$. An $s_i$-local move is trivial if $s_i(Y)=Y$.
\end{definition}

We conclude this section with an open question.

\begin{question}
What other types of transpositions $(i,j)$ can also be interpreted as local moves on plane trees?
\end{question}

\section{The graph of $s_i$-local moves in type $A$}\label{section: graph in type A}

Theorem \ref{theorem: graph local moves} showed that the graph whose vertices are plane trees with $n$ edges and whose edges are $s_i$-local moves is isomorphic to the graph in Definition \ref{definition: graph on tableaux} for the partition $(n,n)$.  Remark \ref{remark: not all local moves are perms} demonstrated that this graph is a subgraph (proper subgraph for $n > 2$) of the graph of plane trees under {\em all} local moves.  

Heitsch studied the graph of plane trees under {\em all} local moves and compared it to similar graphs for other combinatorial objects enumerated by Catalan numbers \cite{Hei}.  However when we remove edges from these graphs, many of Heitsch's properties no longer hold.   We explore the statistics of these modified graphs in this section.  We restrict our attention to the partition $(n,n)$ and denote the graph from Definition \ref{definition: graph on tableaux} by $G^A$.  We refer to $G^A$ as the graph of $s_i$-local moves in type $A$.  Note that the permutations whose corresponding maps $s_i$ are defined on this partition are in $S_{2n}$ rather than $S_n$.  (In later sections we look at local moves corresponding to other Weyl groups.)

We begin by proving that the graph of $s_i$-local moves is still connected in type $A$.

\begin{proposition}\label{proposition: connected}
The graph $G^A$ is connected.
\end{proposition}

\begin{proof}

We describe a way to construct a path between any two standard Young tableaux $Y$ and $Y'$ that both have shape $(n,n)$.  If $Y=Y'$ then the path is trivial.  We now induct on the minimum number $i$ that lies on opposite rows in $Y$ and $Y'$.  Suppose that $i$ is the smallest number whose row in $Y$ is different from that in $Y'$.  Suppose further that $i, i+1, i+2, \ldots, i+k$ are all on the same row and $i+k+1$ is on the opposite row in $Y$.  (We allow $k$ to be zero.) \\
\\
We first prove that in $Y$ the number $i+k+1$ is not in the same column as any of $i,i+1,\ldots,i+k$.  Indeed if $i$ is on the bottom row then $i+k+1$ must be in a column to the right of $i+k$ in order for $Y$ to be standard.  Now suppose that $i$ is on the top row of $Y$ and thus on the bottom row of $Y'$.  In $Y'$ we know that $i$ is directly below one of $1,2,\ldots,i-1$ in order for $Y'$ to be standard.  Both $Y$ and $Y'$ have $1,2,\ldots,i-1$ in the same positions, so $Y$ has an empty box in the bottom row below one of $1,2,\ldots,i-1$.  This must be the box occupied by $i+k+1$. \\
 \\
Now consider the standard tableau $s_i s_{i+1} s_{i+2} \cdots s_{i+k-1}s_{i+k}(Y)$.  It is connected to $Y$ in the graph $G^A$ by construction. The numbers $1,2,\ldots,i-1$ are in the same positions in $s_i s_{i+1} \cdots s_{i+k}(Y)$ as in $Y$.  Furthermore the number $i$ occupies opposite rows in $s_i s_{i+1} \cdots s_{i+k}(Y)$ and $Y$.  Thus the first $i$ numbers are on the same rows in $s_i s_{i+1} \cdots s_{i+k}(Y)$ as in $Y'$.  If $1,2,\ldots,2n-1$ are all on the same rows in $Y$ as in $Y'$ then $2n$ must also be on the same row in $Y$ and $Y'$.  (Indeed $2n$ is  on the bottom row for all standard tableaux.)  By induction we can find a path from $Y$ to $Y'$ in $G^A$ as desired. 
\end{proof}

The graph of plane trees under local moves has the structure of a graded poset.  This is true for $G^A$ as well, but for a different rank function.  The next two results describe {\em total distance} and {\em total number of descendants}, two functions that rank $G^A$.  Like Heitsch, we find that the language of plane trees characterizes the ranking more naturally than tableaux.  In particular, we show that $s_i$-local moves change both the total distance and the total number of descendants by exactly one.  

\begin{proposition}\label{proposition: total distance}
Fix a plane tree $T$ with root $v_0$.

The total distance of the plane tree $d_T$ is defined as $d_T = \sum_{v \in V(T)} dist(v,v_0)$.

If $T'$ is obtained from $T$ by an $s_{i}$-local move of type \eqref{type A local move} then $d_T-1=d_{T'}$.
If $T'$ is obtained from $T$ by an $s_{i}$-local move of type \eqref{type B local move} then $d_T+1=d_{T'}$.
\end{proposition}

\begin{proof}
The proof follows by comparing the distances in the schematics in Figure~\ref{fig:blob1}.  An $s_{i}$-local move does not change the distance between the root and the vertices in the subtrees $a$, $b$, $c$, $d$, and $e$, each of which can be empty.  In the tree to the left, the leaf between half-edges $i$ and $i+1$ has no descendants.  Moreover this vertex is one edge farther from the root than both ``ankles" of the tree to the right are, changing the total distance by exactly one.
\end{proof}
%%%%%%%%%%%%%%%%%%%%%%%%%%%%%%%%%%%%%%%%%%%%%%%%
%%%%%%%%%%%%%%%%%% Blob diagram begins %%%%%%%%%%%%%%%%%%%
%%%%%%%%%%%%%%%%%%%%%%%%%%%%%%%%%%%%%%%%%%%%%%%%
\begin{center}
\begin{figure}[H]
$$
\begin{aligned}
\tikzset{blob/.style={draw, dashed, kite, rounded corners, minimum size=1.4cm}} % Subtree blob style
\begin{tikzpicture}
[level distance=1.46cm,
level 1/.style={sibling distance=1.5cm},
level 2/.style={sibling distance=3cm}]
\tikzstyle{every node}=[circle, draw, scale=0.8, inner sep=2pt]
\node (root) {\tiny $v_0$}
	child{
                   node[fill] {}
        		edge from parent[draw=none]
                            child { node[fill] {}      
                            child { node[fill] {} 
	edge from parent 
	node[left, draw=none] {\tiny $i$}
	node[right,draw=none]  {\tiny $i + 1$}    
    }
 	edge from parent 
	node[left, draw=none] {\tiny $j$}
	node[right, draw=none]  {\tiny $j'$}
    }
    node[blob, shape border rotate=0, xshift=0cm, yshift=0.3cm] {$a$} % Center root blob
    node[blob, shape border rotate=90, xshift=-1.1cm, yshift=-0.8cm] {$b$} % Left root blob
    node[blob, shape border rotate=270, xshift=1.1cm, yshift=-0.8cm] {$e$} % Right root blob
    node[blob, shape border rotate=90, xshift=-1.1cm, yshift=-2.65cm] {$c$}  % Left stem blob
    node[blob, shape border rotate=270, xshift=1.1cm, yshift=-2.65cm] {$d$}  % Right stem blob
    };
\end{tikzpicture}
\end{aligned}
%%%%%%%%%%%%%%%%%%%% Arrows begin %%%%%%%%%%%%%%%%%%%%%%%%%
\quad
\begin{aligned}
\begin{tikzpicture}
  [scale = 1.5]
  \node (n1) at (0, 0) {};
  \node (n2) at (1.5, 0) {};
  \draw[->] (0, 0.25) -- (1.5, 0.25) node[midway, above] {\tiny Type (1) local move};
   \draw[<-] (0, -0.25) -- (1.5, -0.25) node[midway, below] {\tiny Type (2) local move};
\end{tikzpicture}
\end{aligned}
\quad
%%%%%%%%%%%%%%%%%%%% Arrows end %%%%%%%%%%%%%%%%%%%%%%%%%
\begin{aligned}
\tikzset{blob/.style={draw, dashed, kite, rounded corners, minimum size=1.4cm}} % Subtree blob style
\begin{tikzpicture}
[level distance=1.46cm,
level 1/.style={sibling distance=1.5cm},
level 2/.style={sibling distance=2cm}]
\tikzstyle{every node}=[circle, draw, scale=.8, inner sep=2pt]
\node(root) {\tiny $v_0$}
child{
    node[fill] (top) {}
        	edge from parent[draw=none]
child {
    node[fill] (loleft) {} 
    	edge from parent 
	node[ellipse, above, draw=none, sloped] {\tiny $j$}
	node[ellipse, below, draw=none, sloped]  {\tiny $i$}  
}
child {
    node[fill] (loright) {}
        	edge from parent 
	node[ellipse, below, draw=none, sloped] {\tiny $i + 1$}
	node[ellipse, above, draw=none, sloped]  {\tiny $j'$}  
}
    node[blob, shape border rotate=0, xshift=0cm, yshift=0.3cm] {$a$} % Center root blob
    node[blob, shape border rotate=90, rotate=-15, xshift=-0.9cm, yshift=-0.8cm] {$b$} % Left root blob
    node[blob, shape border rotate=270, rotate=15, xshift=0.9cm, yshift=-0.8cm] {$e$} % Right root blob
    node[blob, shape border rotate=90, rotate=30, xshift=-3.5cm, yshift=-1.675cm] {$c$} % Left leaf blob
    node[blob, shape border rotate=270, rotate=-30, xshift=3.5cm, yshift=-1.675cm] {$d$} % Right leaf blob
}; 
\end{tikzpicture}
\end{aligned}
$$
\caption{Edges with subtrees under $s_i$-local moves \label{fig:blob1}}
\end{figure}
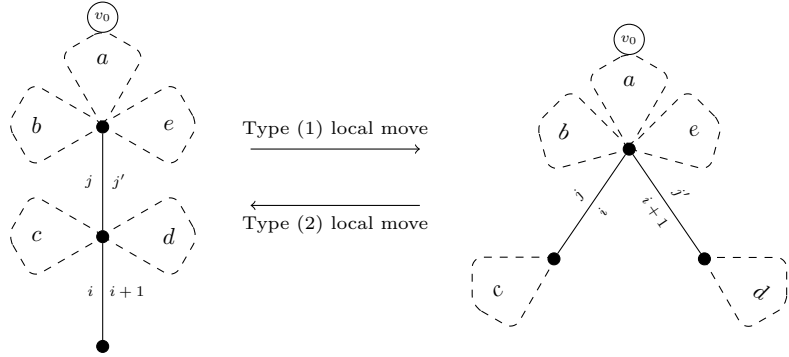
\end{center}
%%%%%%%%%%%%%%%%%%%%%%%%%%%%%%%%%%%%%%%%%%%%%%%
%%%%%%%%%%%%%%%%%% Blob diagram ends %%%%%%%%%%%%%%%%%%%
%%%%%%%%%%%%%%%%%%%%%%%%%%%%%%%%%%%%%%%%%%%%%%%
\begin{proposition}\label{proposition: total number of descendants}
Fix a plane tree $T$ with root $v_0$.

The total number of descendants in $T$ is defined as $des_T = \sum_{v \in V(T)} \big| \left\{ \textup{descendants of } v \right\} \big|$.

If $T'$ is obtained from $T$ by an $s_{i}$-local move of type \eqref{type A local move} then $des_T-1=des_{T'}$.
If $T'$ is obtained from $T$ by an $s_{i}$-local move of type \eqref{type B local move} then $des_T+1=des_{T'}$.
\end{proposition}

\begin{proof}
Consider again the schematic in Figure~\ref{fig:blob1}.  The number of descendants of the root as well as all vertices in the subtrees $a$, $b$, $c$, $d$, $e$ remains the same after each $s_{i}$-local move.  However the length-two path to the left has a total of three descendants while the peak to the right has a total of only two.
\end{proof}

The previous proofs were similar in part because they turn out to count the same quantities, as we prove next.

\begin{proposition}
Let $T$ be a plane tree with root $v_0$. The total distance equals the total number of descendants, namely 
\[d_T = des_T\]
\end{proposition}

\begin{proof}
Vertex $v$ of plane tree $T$ has distance $k$ from the root exactly when the unique path between $v$ and the root has $k+1$ vertices on it.  The $k$ vertices on this path other than $v$ are precisely the vertices in $T$ with $v$ as a descendant.  Thus each vertex $v$ contributes exactly $k$ to $d_T$ and exactly $k$ to $des_T$.
\end{proof}

\begin{remark} 
The notions of total distance and of total number of descendants can be useful in different contexts.  For one example, see the proof of Proposition \ref{proposition: max/min elements}. For another example, note that each descendant in a plane tree corresponds to a nesting of arcs in the associated noncrossing matching.  Thus the total number of descendants in a plane tree corresponds to the total number of nestings within a noncrossing matching. (We do not discuss noncrossing matchings in detail in this manuscript; for more, see e.g. \cite{Rus11, RusTym11}.)
\end{remark}

The next proposition is a direct result of the previous propositions.

\begin{proposition}\label{proposition: ranked poset}
Both total distance and total number of descendants partition the vertices of $G^A$ into the same subsets of plane trees.  

Direct the graph $G^A$ according to the rule that each edge is directed $T \rightarrow T'$ if $T'$ is obtained from $T$ by a local move of type \eqref{type A local move}.  This turns $G_A$ into a graded poset.  Moreover we can impose a rank function $\rho(T)=d_T$ on this graded poset, whose ranks are characterized by the subsets of plane trees with total distance $k$ (respectively total number of descendants $k$).
\end{proposition}

\begin{proof}
The first claim is an immediate corollary of the fact that $d_T=des_T$ for each plane tree $T$.  

The directed graph $G^A$ is acyclic, and thus a poset, because if $T_1 \rightarrow T_2 \rightarrow \cdots \rightarrow T_k$ is any directed path then $d_{T_1} > d_{T_2} > \cdots > d_{T_k}$ and so the endpoint cannot coincide with the initial point of the path.

Finally a function is a rank function if the following two conditions are met:
\begin{enumerate}
\item The function is compatible with the partial order, namely if there is a path $T_1 \rightarrow T_2 \rightarrow \cdots \rightarrow T_k$ then $\rho(T_1) > \rho(T_k)$.  We just confirmed this for total distance (respectively total number of descendants).
\item If $T_1 \rightarrow T_2$ is an edge in the graph then $\rho(T_1)=\rho(T_2)+1$.  This is the content of Proposition \ref{proposition: total distance} (respectively Proposition \ref{proposition: total number of descendants} for total number of descendants).
\end{enumerate}
The final claim follows by definition of the rank function. 
\end{proof}

\begin{remark}\label{remark: not unimodal}
The graph $G^A$ does not satisfy the same kind of symmetries as the graph for all local moves does.  For instance Heitsch proved that the number of plane trees of rank $k$ agrees with those of rank $n-k+1$ for each $k=1,\ldots,n$.  That is clearly false here:  for instance the sequence of the number of tableaux of shape $(3,3)$ of each rank in increasing order is $(1,2,1,1)$ as shown in Figure~\ref{fig:posetA3}. (Note too that this is not the same rank function that Heitsch uses, as our graded poset has fewer edges than hers.) %%%
\end{remark}

%%%%%%%%%%%%%%%%%%%%%%%%%%%%%%%%%%%%%%%%%%%%%%%%%%%
%%%%%%%%%%%%%%%%%%%%%%%% Lattice n = 3 begins %%%%%%%%%%%%%%%%%%%
%%%%%%%%%%%%%%%%%%%%%%%%%%%%%%%%%%%%%%%%%%%%%%%%%%%
\begin{figure}[H]
\begin{tikzpicture}[scale = 0.5, every node/.style={scale=0.6}]
\node (lattice) {
\begin{tabular}{c c c c c}
%%%%%%%%%% Line 1 %%%%%%%%%%%
%%%%%%%%%% Line 1 %%%%%%%%%%%
& %%%%next column
& %%%%next column
%%%%%%%%%%%%%%%%%%% Minimal tree begins %%%%%%%%%%%%%%%%%%%%%
$ \begin{gathered}
\begin{tikzpicture}[level distance=1cm,
level 1/.style={sibling distance=1.5cm},
level 2/.style={sibling distance=1cm}]
\tikzstyle{every node}=[circle, draw, scale=.8, inner sep=2pt]
\node[inner sep=2pt] (root) {}
    child {
    node[fill] {} 
    	edge from parent 
	node[sloped, ellipse, above, draw=none] {\tiny $1$}
	node[sloped, ellipse, below, draw=none]  {\tiny $2$}  
    }
    child {
    node[fill] {}
    	edge from parent 
	node[left, draw=none] {\tiny $3$}
	node[right, draw=none]  {\tiny $4$}  
    }
    child {
    node[fill] {}
    	edge from parent 
	node[sloped, ellipse, below, draw=none] {\tiny $5$}
	node[sloped, ellipse, above, draw=none]  {\tiny $6$}  
    };
\end{tikzpicture} \\
\tiny
\begin{tabular}{ | c | c | c | }
  \hline
  $1$ & $3$ & 5 \\ \hline
  $2$ & $4$ & 6 \\ \hline
\end{tabular}
\end{gathered} $
\vspace{10pt}
%%%%%%%%%%%%%%%%%%% Minimal tree ends %%%%%%%%%%%%%%%%%%%%%
& %%%%next column
& %%%%next column
\\
%%%%%%%%%% Line 2 %%%%%%%%%%%
%%%%%%%%%% Line 2 %%%%%%%%%%%
& %%%%next column
%%%%%%%%%%%%%%%%%%%% Arrow begins %%%%%%%%%%%%%%%%%%%%%%%%
$ \begin{gathered}
\begin{tikzpicture}
  [scale = 1.5]
  \node (n1) at (0, 0) {};
  \node (n2) at (1, 1) {};
  \draw[-] (n1) -- (n2) node[sloped, midway, above] {\small $s_{2}$};
\end{tikzpicture}
\end{gathered} $
\vspace{10pt}
%%%%%%%%%%%%%%%%%%%% Arrow ends %%%%%%%%%%%%%%%%%%%%%%%%
& %%%%next column
& %%%%next column
%%%%%%%%%%%%%%%%%%%% Arrow begins %%%%%%%%%%%%%%%%%%%%%%%%
$ \begin{gathered}
\begin{tikzpicture}
  [scale = 1.5]
  \node (n1) at (0, 0) {};
  \node (n2) at (1, -1) {}; 
  \draw[-] (n1) -- (n2) node[sloped, midway, above] {\small $s_{4}$};
\end{tikzpicture}
\end{gathered} $
\vspace{10pt}
%%%%%%%%%%%%%%%%%%%% Arrow ends %%%%%%%%%%%%%%%%%%%%%%%%
& %%%%next column
\\
%%%%%%%%%% Line 3 %%%%%%%%%%%
%%%%%%%%%% Line 3 %%%%%%%%%%%
%%%%%%%%%%%%%%%%%%%% Left arm begins %%%%%%%%%%%%%%%%%%%%%%%%
$ \begin{gathered}
\begin{tikzpicture}[level distance=1cm,
level 1/.style={sibling distance=2cm},
level 2/.style={sibling distance=1cm}]
\tikzstyle{every node}=[circle, draw, scale=.8, inner sep=2pt]
\node[inner sep=2pt] (Root) {}
    child {
    node[fill] {}
    child { 
    node[fill] {} 
	edge from parent 
	node[left, draw=none] {\tiny $2$}
	node[right, draw=none]  {\tiny $3$}    
    }
	edge from parent 
	node[sloped, ellipse, above, draw=none] {\tiny $1$}
	node[sloped, ellipse, below, draw=none]  {\tiny $4$}  
	}
child {
    node[fill] {}
    	edge from parent 
	node[sloped, ellipse, below, draw=none] {\tiny $5$}
	node[sloped, ellipse, above, draw=none]  {\tiny $6$}  
};
\end{tikzpicture} \\
\tiny
\begin{tabular}{ | c | c | c | }
  \hline
  $1$ & $2$ & 5 \\ \hline
  $3$ & $4$ & 6 \\ \hline
\end{tabular}
\end{gathered} $
\vspace{10pt}
%%%%%%%%%%%%%%%%%%%% Left arm ends %%%%%%%%%%%%%%%%%%%%%%%%
& %%%%next column
& %%%%next column
& %%%%next column
& %%%%next column
%%%%%%%%%%%%%%%%%%%% Right arm begins %%%%%%%%%%%%%%%%%%%%%%%%
$ \begin{gathered}
\begin{tikzpicture}[level distance=1cm,
level 1/.style={sibling distance=2cm},
level 2/.style={sibling distance=1cm}]
\tikzstyle{every node}=[circle, draw, scale=.8, inner sep=2pt]
\node[inner sep=2pt] (Root) {}
child {
    node[fill] {}
    	edge from parent 
	node[sloped, ellipse, above, draw=none] {\tiny $1$}
	node[sloped, ellipse, below, draw=none]  {\tiny $2$}  
}
    child {
    node[fill] {}
    child { 
    node[fill] {} 
	edge from parent 
	node[left, draw=none] {\tiny $4$}
	node[right, draw=none]  {\tiny $5$}    
    }
	edge from parent 
	node[sloped, ellipse, below, draw=none] {\tiny $3$}
	node[sloped, ellipse, above, draw=none]  {\tiny $6$}  
	};
\end{tikzpicture} \\
\tiny
\begin{tabular}{ | c | c | c | }
  \hline
  $1$ & $3$ & 4 \\ \hline
  $2$ & $5$ & 6 \\ \hline
\end{tabular}
\end{gathered} $
\vspace{10pt}
%%%%%%%%%%%%%%%%%%%% Right arm ends %%%%%%%%%%%%%%%%%%%%%%%%
\\
%%%%%%%%%% Line 4 %%%%%%%%%%%
%%%%%%%%%% Line 4 %%%%%%%%%%%
& %%%%next column
%%%%%%%%%%%%%%%%%%%% Arrow begins %%%%%%%%%%%%%%%%%%%%%%%%
$ \begin{gathered}
\begin{tikzpicture}
  [scale = 1.5]
  \node (n1) at (0, 0) {};
  \node (n2) at (1, -1) {};
  \draw[-] (n1) -- (n2) node[sloped, midway, above] {\small $s_{4}$};
\end{tikzpicture}
\end{gathered} $
\vspace{10pt}
%%%%%%%%%%%%%%%%%%%% Arrow ends %%%%%%%%%%%%%%%%%%%%%%%%
& %%%%next column
& %%%%next column
%%%%%%%%%%%%%%%%%%%% Arrow begins %%%%%%%%%%%%%%%%%%%%%%%%
$ \begin{gathered}
\begin{tikzpicture}
  [scale = 1.5]
  \node (n1) at (0, 0) {};
  \node (n2) at (1, 1) {};
  \draw[-] (n1) -- (n2) node[sloped, midway, above] {\small $s_{2}$};
\end{tikzpicture}
\end{gathered} $
\vspace{10pt}
%%%%%%%%%%%%%%%%%%%% Arrow ends %%%%%%%%%%%%%%%%%%%%%%%%
& %%%%next column
\\  
%%%%%%%%%% Line 5 %%%%%%%%%%%
%%%%%%%%%% Line 5 %%%%%%%%%%%
& %%%%next column
& %%%%next column
%%%%%%%%%%%%%%%%%%% Y tree begins %%%%%%%%%%%%%%%%%%%%%%
$ \begin{gathered}
\begin{tikzpicture}[level distance=1cm,
level 1/.style={sibling distance=1cm},
level 2/.style={sibling distance=2cm}]
\tikzstyle{every node}=[circle, draw, scale=.8, inner sep=2pt]
\node[inner sep=2pt] (Root) {}
child {
    child {
    node[fill] {}
	edge from parent 
	node[sloped, ellipse, above, draw=none]  {\tiny $2$}
	node[sloped, ellipse, below, draw=none]  {\tiny $3$}  
}	
    child { 
    node[fill] {} 
	edge from parent 
	node[sloped, ellipse, above, draw=none] {\tiny $5$}
	node[sloped, ellipse, below, draw=none]  {\tiny $4$}    
    }
	    node[fill] {}
	edge from parent 
	node[left, draw=none] {\tiny $1$}
	node[right, draw=none]  {\tiny $6$}    
};
\end{tikzpicture} \\
\tiny
\begin{tabular}{ | c | c | c | }
  \hline
  $1$ & $2$ & 4 \\ \hline
  $3$ & $5$ & 6 \\ \hline
\end{tabular}
\end{gathered} $
\vspace{10pt}
%%%%%%%%%%%%%%%% Y tree ends %%%%%%%%%%%%%%%%%%%%%%%%%%%
& %%%%next column
& %%%%next column
\\  
%%%%%%%%%% Line 6 %%%%%%%%%%%
%%%%%%%%%% Line 6 %%%%%%%%%%%
& %%%%next column
& %%%%next column
%%%%%%%%%%%%%%%%%%%% Arrow begins %%%%%%%%%%%%%%%%%%%%%%%%
$ \begin{gathered}
\begin{tikzpicture}
  [scale = 1.5]
  \node (n1) at (0, 0) {};
  \node (n2) at (0, -1) {};
  \node [overlay] at (0.15, -0.5) {\small $s_{3}$};
  \draw[-] (n1) -- (n2);
\end{tikzpicture}
\end{gathered} $
\vspace{10pt}
%%%%%%%%%%%%%%%%%%%% Arrow ends %%%%%%%%%%%%%%%%%%%%%%%%
& %%%%next column
& %%%%next column
\\ 
%%%%%%%%%% Line 7 %%%%%%%%%%%
%%%%%%%%%% Line 7 %%%%%%%%%%%
& %%%%next column
& %%%%next column
$ \begin{gathered}
\begin{tikzpicture}[level distance=1cm,
level 1/.style={sibling distance=1cm},
level 2/.style={sibling distance=1cm}]
level 3/.style={sibling distance=1cm}]
\tikzstyle{every node}=[circle, draw, scale=.8, inner sep=2pt]
\node[inner sep=2pt] (Root) {}
child {
    child {
    node[fill] {}
    child { 
    node[fill] {} 
	edge from parent 
	node[left, draw=none] {\tiny $3$}
	node[right, draw=none]  {\tiny $4$}    
    }
	edge from parent 
	node[left, draw=none] {\tiny $2$}
	node[right, draw=none]  {\tiny $5$}  
}	
	    node[fill] {}
	edge from parent 
	node[left, draw=none] {\tiny $1$}
	node[right, draw=none]  {\tiny $6$}    
};
\end{tikzpicture} \\
\tiny
\begin{tabular}{ | c | c | c | }
  \hline
  $1$ & $2$ & 3 \\ \hline
  $4$ & $5$ & 6 \\ \hline
\end{tabular}
\end{gathered} $
\vspace{10pt}
%%%%%%%%%%%%%%%% Maximal tree ends %%%%%%%%%%%%%%%%%%%%%%%%
\end{tabular}
};
%%%%%%%%%% Big side arrows begin %%%%%%%%%%
%%%%%%%%%% Big side arrows begin %%%%%%%%%%
\draw [<-] ([xshift = -1.25cm] lattice.north west) -- ([xshift = -1.25cm] lattice.south west) node[sloped, ellipse, midway, above, rotate=180, draw=none] {\small Type (1) local moves};
\draw [->] ([xshift = 1.25cm] lattice.north east) -- ([xshift = 1.25cm] lattice.south east) node[sloped, ellipse, midway, above, draw=none] {\small Type (2) local moves};
%%%%%%%%%% Big side arrows end %%%%%%%%%%
%%%%%%%%%% Big side arrows end %%%%%%%%%%
\end{tikzpicture}
\caption{Graded poset obtained from $G^A$ when $n = 3$ \label{fig:posetA3}}
\end{figure}
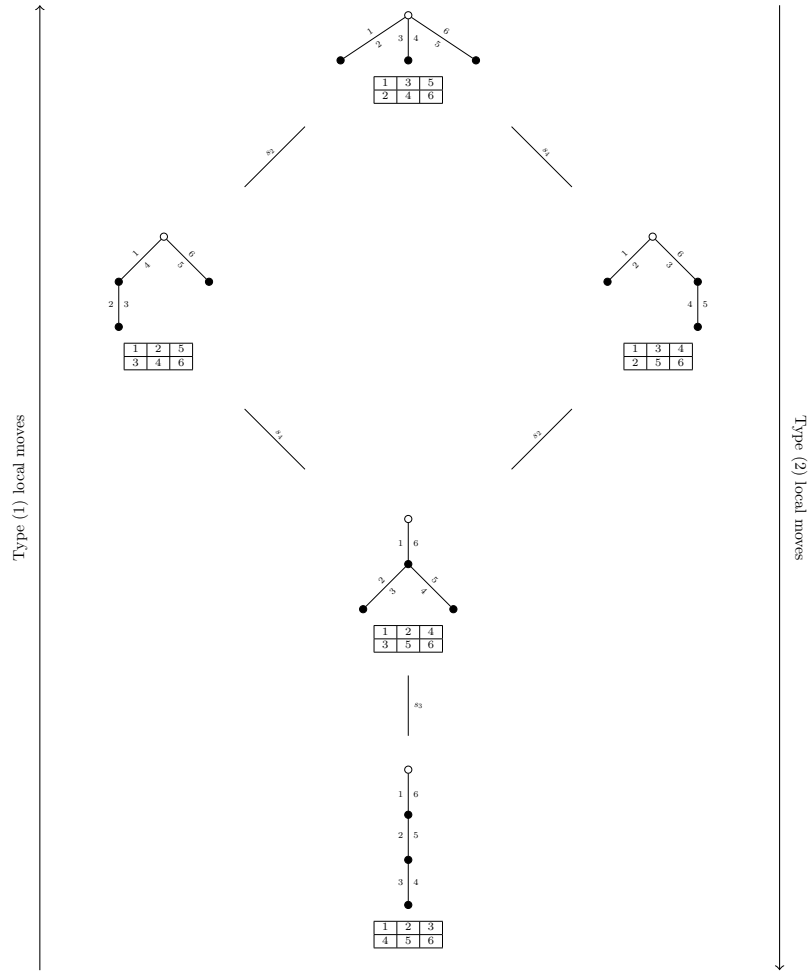
%%%%%%%%%%%%%%%%%%%%%%%%%%%%%%%%%%%%%%%%%%%%%%%%%%%
%%%%%%%%%%%%%%%%%%%%%% Lattice n = 3 ends %%%%%%%%%%%%%%%%%%%
%%%%%%%%%%%%%%%%%%%%%%%%%%%%%%%%%%%%%%%%%%%%%%%%%%%

However we can prove the following.

\begin{proposition}  \label{proposition: max/min elements}
There is a unique element of maximal rank and a unique element of minimal rank.
\end{proposition}

\begin{proof}
Consider the graph $G^A$ whose vertex set is the set $\mathcal{T}_n$ of plane trees with $n$ edges.  If $T$ is a plane tree in $\mathcal{T}_n$ then its root must have $n$ descendants since any other vertex in the graph is a descendant of the root.  So the minimal total number of descendants is $n$.  This is achieved by the star graph in Figure~\ref{fig:subfigmin}.

The plane tree $T$ is connected so there is at least one vertex of each possible distance from the root.  The path graph in Figure~\ref{fig:subfigmax} has just one vertex at each distance from the root and therefore maximizes the total distance.
\end{proof}

%%%%%%%%%%%%%%%%%%%%%%%%%%%%%%%%%%%%%%%%%%%%%%%%
%%%%%%%%%%%%%%%%%% Min/Max diagram begins %%%%%%%%%%%%%%%%%
%%%%%%%%%%%%%%%%%%%%%%%%%%%%%%%%%%%%%%%%%%%%%%%%
\begin{figure}[H]
    \begin{subfigure}[b]{0.49\textwidth}
%%%%%%%%%%%%%%%%%%% Star %%%%%%%%%%%%%%%%%%%%%%%%%%%
\begin{tikzpicture}[level distance=1.5cm,
level 1/.style={sibling distance=2cm},
level 2/.style={sibling distance=1cm}]
\tikzstyle{every node}=[circle, draw, scale=.8, inner sep=2pt]
\node (root) [circle, draw] {}
  child {node[fill] (B) {}
    	edge from parent 
	node[sloped, ellipse, above, draw=none] {\tiny$1$}
	node[sloped, ellipse, below, draw=none]  {\tiny$2$}}  
  child {node[fill] (1) {}
     	edge from parent 
	node[sloped, ellipse, above, draw=none] {\tiny$3$}
	node[sloped, ellipse, below, draw=none]  {\tiny$4$}}  
  child {node[fill] (2) {}
     	edge from parent 
	node[sloped, ellipse, above, draw=none] {\tiny$2n - 2$}
	node[sloped, ellipse, below, draw=none]  {\tiny$2n - 3$}}  
  child {node[fill] (A) {}
     	edge from parent 
	node[sloped, ellipse, above, draw=none] {\tiny$2n$}
	node[sloped, ellipse, below, draw=none]  {\tiny$2n - 1$}} ;
\node [draw=none, anchor=center, yshift=-1.85cm] {\LARGE$\ldots$}; % Dots
\draw [decorate,decoration={brace,amplitude=2mm}] ($(A.west)+(1mm,-3mm)$) -- ($(B.east)+(-1mm,-3mm)$) ; % Bottom brace
\node [draw=none, anchor=center, yshift=-3cm] {\small $des_T = n$} ; % Bottom label
\end{tikzpicture} \\ \\
%%%%%%%%%%%%%%%%%%%% Star table %%%%%%%%%%%%%%%%%%%%%%%
\centering
\tiny
\begin{tabular}{ | c | c | c | c | c |}
  \hline
  $1$ & $3$ &  $\cdots$ & $2n - 3$ & $2n - 1$ \\ \hline
  $2$ & $4$ & $\cdots$ & $2n - 2$ & $2n$ \\ \hline
\end{tabular} \vspace{3mm}
        \caption{Minimal tree (star graph) \label{fig:subfigmin}}
    \end{subfigure} \hfill
%%%%%%%%%%%%%%%%%%%%%%%%%%%%%%%%%%%%%%%%%%%%%%%%%
\begin{subfigure}[b]{0.49\textwidth}
%%%%%%%%%%%%%%%%%%%%%% Path %%%%%%%%%%%%%%%%%%%%%%%%%
\qquad \qquad \qquad
\begin{tikzpicture} [level distance=1cm,
level 1/.style={sibling distance=2cm},
level 2/.style={sibling distance=1cm}]
\tikzstyle{every node}=[circle, draw, scale=.8, inner sep=2pt]
\node (Root) {}
    child { node[fill] {}
    child { node[fill] {} 
    child { node[fill] {}
    child { node[fill] {} 
    child { node[fill] {}     
    	edge from parent 
	node[left, draw=none] {\tiny$n$}
	node[right,draw=none]  {\tiny$n + 1$}    
    }
    	edge from parent 
	node[left, draw=none] {\tiny$n - 1$}
	node[right,draw=none]  {\tiny$n + 2$}    
    }
    	edge from parent[draw=none] % See dashed line
    }
	edge from parent 
	node[left, draw=none] {\tiny$2$}
	node[right,draw=none]  {\tiny$2n - 1$}    
    }
	edge from parent 
	node[left, draw=none] {\tiny$1$}
	node[right, draw=none]  {\tiny$2n$}
};
\node [draw=none, anchor=center, yshift=-3cm] {\LARGE$\vdots$} ; % Dots
\draw [decorate,decoration={brace,amplitude=2mm}] (0.85, 0) -- (0.85, -5) ; % Side brace
\node [draw=none, anchor=center, xshift=2.5cm, yshift=-3.1cm] {\small $d_T = \binom{n+1}{2}$}; % Bottom label
\end{tikzpicture} \\ \\
%%%%%%%%%%%%%%%%%%%% Path table %%%%%%%%%%%%%%%%%%%%%%
\centering
\tiny
\begin{tabular}{ | c | c | c | c | c |}
  \hline
  $1$ & $2$ &  $\cdots$ & $n - 1$ & $n$ \\ \hline
  $n + 1$ & $n + 2$ & $\cdots$ & $2n - 1$ & $2n$ \\ \hline
\end{tabular} \vspace{3mm}
        \caption{Maximal tree (path graph) \label{fig:subfigmax}}   
    \end{subfigure}
\caption{Minimal and maximal trees with associated Young tableaux \label{minmax}}
\end{figure}
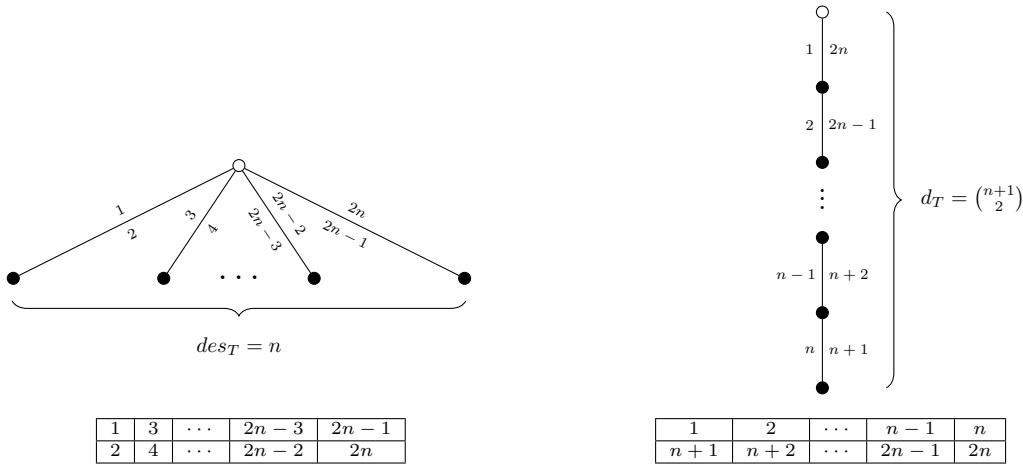
%%%%%%%%%%%%%%%%%%%%%%%%%%%%%%%%%%%%%%%%%%%%%%%%%%
%%%%%%%%%%%%%%%%%% Min/Max diagram ends %%%%%%%%%%%%%%%%%%%%
%%%%%%%%%%%%%%%%%%%%%%%%%%%%%%%%%%%%%%%%%%%%%%%%%%

\begin{corollary}
For the partition $(n, n)$ the number of ranks in the graded poset obtained from $G^A$ and ranked by the total distance function $d_T$ is $\binom{n+1}{2} - n+1$.
\end{corollary}

\begin{proof}
The total distance of the path graph is the binomial coefficient $\binom{n+1}{2}$.  The total distance of the star graph is $n$.  There is at least one plane tree of each rank between these because $G^A$ is connected and each edge changes rank by exactly one.
\end{proof}

Again we close with an open question.

\begin{question}
Is the rank sequence of $G^A$ unimodal for every $n$?
\end{question}

%%%%%%%%%%%%%%%%%%%%%%%%%%%%%%%%%%%%%%%%%%%%%
%%%%%%%%%%%%%%%%%%%% Section 5 %%%%%%%%%%%%%%%%%%%
\section{The graph of $s_i$-Local moves in type $C$}\label{section: type C}

In our description of $s_i$-local moves so far, we relied on an analogy with the generators of the symmetric group.  We now extend the analogy to define maps $s_i^C$ corresponding to the generators of the Weyl group of type $C$.  Intuitively the Weyl group of type $C$ plays the same role for the complex symplectic group $Sp(2n,\mathbb{C})$ that the permutation matrices play for $n \times n$ invertible matrices $GL(n,\mathbb{C})$.  We will represent the Weyl group of type $C$ as a subgroup of the permutations in $S_{2n}$ using generators that we describe below.  

In this section we show that we can easily define maps $s_i^C$ on the standard tableaux of shape $(n,n)$ even when there are no analogous local moves on the corresponding plane trees.  Nonetheless, the geometry of the plane trees is the best way to describe key properties of these maps.  More precisely we prove that restricting to type $C$ $s_i$-local moves identifies symmetry within the plane trees.  The main theorem of this section shows that within the graph whose vertices are plane trees and whose edges are type $C$ $s_i$-local moves, there are precisely two connected components: one composed of {\em symmetric} plane trees and one composed of {\em asymmetric} plane trees.

We define functions analogous to the maps $s_i$ for type $C$ instead of type $A$.  The reader who is not familiar with Weyl groups can take this as a definition of the Weyl group of type $C$.  Like in our earlier treatment, the maps $s_i$ and $s_i^C$ are both permutations in $S_{2n}$.  However note that in type $A$ we have maps $s_i$ for each $i \in \{1,2,\ldots,2n-1\}$ while in type $C$ we only have $s_i^C$ for $i \in \{1,2,\ldots,n\}$.

\begin{definition}
The maps of type $C$ are the involutions on standard tableaux defined by:
\begin{center}
$s_1^C=s_1s_{2n-1}$ corresponding to the reflection $(1,2)(2n-1,2n)$\\
$s_2^C=s_2s_{2n-2}$ corresponding to the reflection $(2,3)(2n-2,2n-1)$\\
\qquad \vdots \\
 $s_{n-1}^C=s_{n-1}s_{n+1}$ corresponding to the reflection $(n-1,n)(n+1,n+2)$\\
 $s_n^C=s_n$ corresponding to the reflection $(n,n+1)$
\end{center}
Using the bijection $\phi: \mathcal{T}_n \rightarrow \left\{\textup{standard Young tableaux of size }(n,n)\right\}$ we also define maps $s_i^C$ on plane trees according to the rule
\[s_i^C(T)=\phi^{-1}(s_i^C(\phi(T))\]
\end{definition}

Generally the simple reflections of type $C$ exchange disjoint pairs of integers according to the product $s_i s_{2n-i}$ of type $A$ reflections.  However note that $s_n$ exchanges just the two integers $n$ and $n+1$.  It is the only simple reflection of type $C$ that exchanges integers between the sets $\{1,2,\ldots,n\}$ and $\{n+1,n+2,\ldots,2n\}$.

\begin{remark}
Note that while we use terminology from earlier in the paper, the maps $s_i^C$ are no longer local moves in the strict sense.  Except for the case when $i=n$ the maps $s_i^C = s_{i}s_{2n-i}$ corresponding to the reflections $(i,i+1)(2n-i,2n-i+1)$ are in fact {\em pairs} of $s_i$-local moves of type $A$. We can perform a pair of $s_i$-local moves on a standard tableau $Y$ simultaneously because the pairs of integers are disjoint: if $i$ and $i+1$ are in the same row or column then $s_{i}$ does nothing; otherwise $s_{i}$ exchanges the positions of $i$ and $i+1$ leaving all the other numbers in their original positions. The same dynamic holds for $s_{2n-i}$ with respect to $2n-i$ and $2n-i+1$. So the standard tableau $s_i^C(Y)$ is always defined.

Our definition for the plane tree $s^C_{i}(T)$ uses the action on the corresponding tableau $\phi(T)$.  This is because often a pair of $s_i$-local moves that would act on a plane tree is {\em not defined} on that plane tree.  Figure~\ref{fig: degenerate local move} provides an example in which the map $s^{C}_{2}$ involves one non-trivial $s_2$-local move and one trivial $s_6$-local move. The heuristic for determining $s^C_i(T)$ directly is to perform all of the local moves $s_i$ and $s_{2n-i}$ that are non-trivial.

%%%%%%%%%%%%%%%%%%%%%%%%%%%%%%%%%%%%%%%%%%%%%%%%
%%%%%%%%%%%%%%%%% Degenerate local move figure begins %%%%%%%%%%%%%%%%
%%%%%%%%%%%%%%%%%%%%%%%%%%%%%%%%%%%%%%%%%%%%%%%%
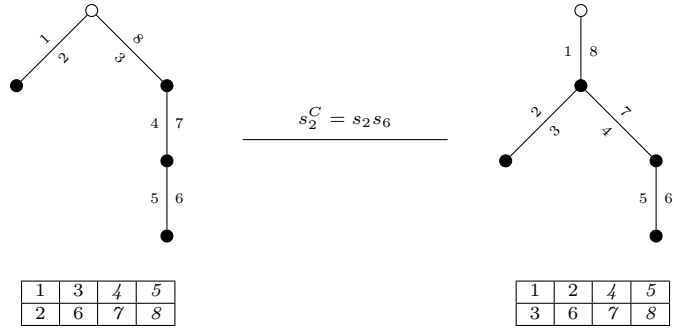
\begin{figure}[H]
\begin{tabular}{ c c c }
%%%%%%%%%%%%%%%%%%% Left tree begins %%%%%%%%%%%%%%%%%%%%%
$ \begin{gathered}
    \begin{tikzpicture}[sloped, 
    level distance=1cm,
    level 1/.style={sibling distance=2cm},
    level 2/.style={sibling distance=1cm}]
    \tikzstyle{every node}=[circle, draw, scale=.8, inner sep=2pt, minimum size=0]
    \node[inner sep=2pt] (Root)[] {}
    child { 
    node[fill] {} 
    	edge from parent 
	node[ellipse, above, draw=none] {\tiny $1$}
	node[ellipse, below, draw=none]  {\tiny $2$}  
    }
    child {
    node[fill] {}
    child {
    node[fill] {}
    child{
    node[fill] {}
    	edge from parent 
	node[ellipse, left, draw=none, rotate=90] {\tiny $5$}
	node[ellipse, right, draw=none, rotate=90]  {\tiny $6$} 
	}
	edge from parent 
	node[ellipse, left, draw=none, rotate=90] {\tiny $4$}
	node[ellipse, right, draw=none, rotate=90]  {\tiny $7$}	
	}
	edge from parent
	node[ellipse, below, draw=none] {\tiny $3$}
	node[ellipse, above, draw=none]  {\tiny $8$}
};
\end{tikzpicture}
\end{gathered} $
%%%%%%%%%%%%%%%%%%% Left tree ends %%%%%%%%%%%%%%%%%%%%%
& %%%%next column
%%%%%%%%%%%%%%%%%%%% Arrow begins %%%%%%%%%%%%%%%%%%%%%%%%
$ \begin{gathered}
\begin{tikzpicture}
  [scale = 1.5]
  \node (n1) at (0, 0) {};
  \node (n2) at (2, 0) {};
  \draw[-] (n1) -- (n2) node[midway, above] {\tiny $s^{C}_{2} = s_{2}s_{6}$};
\end{tikzpicture}
\end{gathered} $
%%%%%%%%%%%%%%%%%%%% Arrow ends %%%%%%%%%%%%%%%%%%%%%%%%
& %%%%next column 
%%%%%%%%%%%%%%%%%%% Right tree begins %%%%%%%%%%%%%%%%%%%%%%
$ \begin{gathered}
\begin{tikzpicture}[level distance=1cm,
level 1/.style={sibling distance=1cm},
level 2/.style={sibling distance=2cm}]
\tikzstyle{every node}=[circle, draw, scale=.8, inner sep=2pt]
\vspace{10pt}
\node[inner sep=2pt] (Root) {}
child {
    child {
    node[fill] {}
	edge from parent 
	node[sloped, ellipse, above, draw=none]  {\tiny $2$}
	node[sloped, ellipse, below, draw=none]  {\tiny $3$}  
}	
    child { 
    node[fill] {} 
         child{
         node[fill]{}
	edge from parent 
	node[left, draw=none] {\tiny $5$}
	node[right, draw=none]  {\tiny $6$}  
	}
	edge from parent 
	node[sloped, ellipse, above, draw=none] {\tiny $7$}
	node[sloped, ellipse, below, draw=none]  {\tiny $4$}   
    }
	node[fill] {}
	edge from parent 
	node[left, draw=none] {\tiny $1$}
	node[right, draw=none]  {\tiny $8$}    
};
\end{tikzpicture} 
\end{gathered} $
\vspace{10pt}
%%%%%%%%%%%%%%%% Right tree ends %%%%%%%%%%%%%%%%%%%%%%%%%%%
\\
%%%%%%%%%% Line 2 %%%%%%%%%%%
%%%%%%%%%% Line 2 %%%%%%%%%%%
%%%%%%%%%%%%%%%%%%% Left tableau begins %%%%%%%%%%%%%%%%%%%%%%
\tiny{
\begin{tabular}{ | c | c | c | c | }
  \hline
  $1$ & $3$ & 4 & 5 \\ \hline
  $2$ & $6$ & 7 & 8 \\ \hline
\end{tabular}
}
%%%%%%%%%%%%%%%% Left tableau ends %%%%%%%%%%%%%%%%%%%%%%%%%%%
& %%%%next column
& %%%%next column
%%%%%%%%%%%%%%%%%%% Right tableau begins %%%%%%%%%%%%%%%%%%%%%%
\tiny{
\begin{tabular}{ | c | c | c | c | }
  \hline
  $1$ & $2$ & 4 & 5 \\ \hline
  $3$ & $6$ & 7 & 8 \\ \hline
\end{tabular}
}
%%%%%%%%%%%%%%%% Right tableau ends %%%%%%%%%%%%%%%%%%%%%%%%%%%
\end{tabular}
\caption{Map $s^{C}_{2}$ involving non-trivial $s_2$-local move and trivial $s_6$-local move \label{fig: degenerate local move}}
\end{figure}
%%%%%%%%%%%%%%%%%%%%%%%%%%%%%%%%%%%%%%%%%%%%%%%%
%%%%%%%%%%%%%%%%% Degenerate local move figure ends %%%%%%%%%%%%%%%%%
%%%%%%%%%%%%%%%%%%%%%%%%%%%%%%%%%%%%%%%%%%%%%%%%

We stress that even though it appears unnatural to define local moves of type $C$ on plane trees (given that the constituent local moves of type $A$ are not necessarily well-defined), the maps $s_i^C$ characterize key geometric properties of the plane trees.  Indeed we think it is a theme of this field that different characterizations of standard tableaux (plane trees, noncrossing matchings, etc.) provide valuable and often complementary information.  
\end{remark}

The maps $s_i^C$ define a graph $G^C$ in the same way that the maps $s_i$ defined a graph $G^A$.

\begin{definition}
The graph $G^C$ is the graph whose vertices are plane trees. An edge connects plane trees $T$ and $T'$ precisely when $T'=s_i^C(T)$ for a map $s_i^C$.  We call $G^C$ the graph of plane trees under $s_i^C$-local moves (read {\em $s_i$-local moves of type $C$}).
\end{definition}

The following definition formalizes our notion of symmetric and asymmetric plane trees.

\begin{definition}
Let $T$ be a plane tree. We say that $T$ is symmetric if and only if for each edge $e(i,j)$ in $T$ the {\em mirror image} $e(2n-j+1, 2n-i+1)$ is also an edge in $T$. A plane tree is asymmetric if it is not symmetric.
\end{definition}

We will prove that the graph of plane trees $G^C$ under the $s_i^C$-local moves has two connected components: one consisting of symmetric plane trees and one consisting of asymmetric plane trees.  Our proof uses several steps.  First we show that no connected component contains both a symmetric plane tree and an asymmetric plane tree.

\begin{lemma}
Each connected component of $G^C$ consists either entirely of symmetric plane trees or entirely of asymmetric plane trees.
\end{lemma}

\begin{proof}
We will show that if $s_i^C$ is a generator of the Weyl group of type $C$ and $T$ is a symmetric plane tree then $s_i^C(T)$ is also symmetric.  It follows that the connected component of $G^C$ containing any symmetric plane tree consists entirely of other symmetric plane trees. Since every tree is either symmetric or asymmetric, it follows further that the connected component of $G^C$ containing any asymmetric plane tree must consist entirely of other asymmetric plane trees.  

Given a subtree $T'$ of symmetric plane tree $T$ we call the edges in $T$ that are symmetric to $T'$ the {\em mirror image} of $T$.   

Consider the half-edges labeled by $i$ and $i+1$.  A priori there are four possibilities: they could both be left-half-edges, they could both be right-half-edges, they could form a leaf, or they could form the interior of a peak.

Table ~\ref{tab:mirrorchart} shows these four possibilities, the mirror image of these possibilities, the $s_i^C$-local move on the original and its mirror image in each case, and the mirror image of the $s_i^C$-local move on the original in each case. Note that in the first two possibilities $i$ and $i + 1$ as well as $2n - i$ and $2n - i +1$ will stay in their respective rows of the corresponding tableaux after the $s_i^C$-local move, which consequently does not alter the either $T'$ or its mirror image. We inspect columns three and five in Table ~\ref{tab:mirrorchart} and observe that they are the same. So if two edges were part of a symmetric tree before we perform an $s_i^C$-local move on them, then they will still be part of a symmetric tree after the $s_i^C$-local move.

%%%%%%%%%%%%%%%%%%%%%%%%%%%%%%%%%%%%%%%%%%%%%%%%
%%%%%%%%%%%%%%%%%%%%%%% Mirror chart begins %%%%%%%%%%%%%%%
%%%%%%%%%%%%%%%%%%%%%%%%%%%%%%%%%%%%%%%%%%%%%%%%
\begin{table}[H]
\resizebox{\linewidth}{!}{
\begin{tabular}{ | c || c | c || c | c |}
  \hline
  $T'$ & Mirror image of $T'$ & $s_{i}^C$-local move on mirror image of $T'$ & $s_{i}^C$-local move on $T'$ & Mirror image of $s_{i}^C$-local move on $T'$ \\
    \hline \hline
%%%%%%%%%%%%%%%%%%%%%%%%%%%%%%%%%%%%%%%%%%%%%%%%%%%
%%%%%%%%%%%%%%%%%%%%%%%%%% 1 %%%%%%%%%%%%%%%%%%%%%%%%
%%%%%%%%%%%%%%%%%%%%%%%%%%%%%%%%%%%%%%%%%%%%%%%%%%%
    &  &  &  & \\
%%%%%%%%%%% T %%%%%%%%%%%%%%
\begin{tikzpicture}[level distance=2cm,
level 1/.style={sibling distance=3.5cm},
level 2/.style={sibling distance=1cm}]
\tikzstyle{every node}=[circle, draw, scale=1, inner sep=2pt]
\vspace{10pt}
\node[fill] (Root) {}
    child {
    node[fill] {}
    child { node[fill] {} 
	edge from parent 
	node[left, draw=none] {$i + 1$}
	node[right, draw=none]  {$j'$}    
    }
	edge from parent 
	node[left, draw=none] {$i$}
	node[right, draw=none]  {$j$}
};
\end{tikzpicture} & 
%%%%%%%%%% Mirror of T %%%%%%%%%%%%%%%
\begin{tikzpicture}[level distance=2cm,
level 1/.style={sibling distance=3.5cm},
level 2/.style={sibling distance=1cm}]
\tikzstyle{every node}=[circle, draw, scale=1, inner sep=2pt]
\vspace{10pt}
\node[fill] (Root) {}
    child {
    node[fill] {}
    child { node[fill] {} 
	edge from parent 
	node[left, draw=none] {$2n - j' + 1$}
	node[right, draw=none]  {$2n - i$}    
    }
	edge from parent 
	node[left, draw=none] {$2n - j + 1$}
	node[right, draw=none]  {$2n -i + 1$}
};
\end{tikzpicture} & 
%%%%%%%%%% Local move on mirror of T %%%%%%%%%%%%%%%
\begin{tikzpicture}[level distance=2cm,
level 1/.style={sibling distance=3.5cm},
level 2/.style={sibling distance=1cm}]
\tikzstyle{every node}=[circle, draw, scale=1, inner sep=2pt]
\vspace{10pt}
\node[fill] (Root) {}
    child {
    node[fill] {}
    child { node[fill] {} 
	edge from parent 
	node[left, draw=none] {$2n - j' + 1$}
	node[right, draw=none]  {$2n - i$}    
    }
	edge from parent 
	node[left, draw=none] {$2n - j + 1$}
	node[right, draw=none]  {$2n -i + 1$}
};
\end{tikzpicture} & 
%%%%%%%%%%% Local move on T %%%%%%%%%%%%%%%
\begin{tikzpicture}[level distance=2cm,
level 1/.style={sibling distance=3.5cm},
level 2/.style={sibling distance=1cm}]
\tikzstyle{every node}=[circle, draw, scale=1, inner sep=2pt]
\vspace{10pt}
\node[fill] (Root) {}
    child {
    node[fill] {}
    child { node[fill] {} 
	edge from parent 
	node[left, draw=none] {$i + 1$}
	node[right, draw=none]  {$j'$}    
    }
	edge from parent 
	node[left, draw=none] {$i$}
	node[right, draw=none]  {$j$}
};
\end{tikzpicture} & 
%%%%%%%%%% Mirror of local move on T %%%%%%%%%%%%%%%
\begin{tikzpicture}[level distance=2cm,
level 1/.style={sibling distance=3.5cm},
level 2/.style={sibling distance=1cm}]
\tikzstyle{every node}=[circle, draw, scale=1, inner sep=2pt]
\vspace{10pt}
\node[fill] (Root) {}
    child {
    node[fill] {}
    child { node[fill] {} 
	edge from parent 
	node[left, draw=none] {$2n - j' + 1$}
	node[right, draw=none]  {$2n - i$}    
    }
	edge from parent 
	node[left, draw=none] {$2n - j + 1$}
	node[right, draw=none]  {$2n -i + 1$}
};
\end{tikzpicture} \\
    &  &  &  & \\
    \hline
%%%%%%%%%%%%%%%%%%%%%%%%%%%%%%%%%%%%%%%%%%%%%%%%%%
%%%%%%%%%%%%%%%%%%%%%%%%%%% 2 %%%%%%%%%%%%%%%%%%%%%%
%%%%%%%%%%%%%%%%%%%%%%%%%%%%%%%%%%%%%%%%%%%%%%%%%%
    &  &  &  & \\
%%%%%%%%%%% T %%%%%%%%%%%%%%
\begin{tikzpicture}[level distance=2cm,
level 1/.style={sibling distance=3.5cm},
level 2/.style={sibling distance=1cm}]
\tikzstyle{every node}=[circle, draw, scale=1, inner sep=2pt]
\node (Root)[fill] {}
    child {
    node[fill] {}
    child { node[fill] {} 
	edge from parent 
	node[left, draw=none] {$j$}
	node[right, draw=none]  {$i$}    
    }
	edge from parent 
	node[left, draw=none] {$j'$}
	node[right, draw=none]  {$i + 1$}
};
\end{tikzpicture} & 
%%%%%%%%%% Mirror of T %%%%%%%%%%%%%%%
\begin{tikzpicture}[level distance=2cm,
level 1/.style={sibling distance=3.5cm},
level 2/.style={sibling distance=1cm}]
\tikzstyle{every node}=[circle, draw, scale=1, inner sep=2pt]
\node (Root)[fill] {}
    child {
    node[fill] {}
    child { node[fill] {} 
	edge from parent 
	node[left, draw=none] {$2n - i + 1$}
	node[right, draw=none]  {$2n - j + 1$}    
    }
	edge from parent 
	node[left, draw=none] {$2n - i$}
	node[right, draw=none]  {$2n - j' + 1$}
};
\end{tikzpicture} & 
%%%%%%%%%% Local move on mirror of T %%%%%%%%%%%%%%%
\begin{tikzpicture}[level distance=2cm,
level 1/.style={sibling distance=3.5cm},
level 2/.style={sibling distance=1cm}]
\tikzstyle{every node}=[circle, draw, scale=1, inner sep=2pt]
\node (Root)[fill] {}
    child {
    node[fill] {}
    child { node[fill] {} 
	edge from parent 
	node[left, draw=none] {$2n - i + 1$}
	node[right, draw=none]  {$2n - j + 1$}    
    }
	edge from parent 
	node[left, draw=none] {$2n - i$}
	node[right, draw=none]  {$2n - j' + 1$}
};
\end{tikzpicture} & 
%%%%%%%%%%% Local move on T %%%%%%%%%%%%%%%
\begin{tikzpicture}[level distance=2cm,
level 1/.style={sibling distance=3.5cm},
level 2/.style={sibling distance=1cm}]
\tikzstyle{every node}=[circle, draw, scale=1, inner sep=2pt]
\node (Root)[fill] {}
    child {
    node[fill] {}
    child { node[fill] {} 
	edge from parent 
	node[left, draw=none] {$j$}
	node[right, draw=none]  {$i$}    
    }
	edge from parent 
	node[left, draw=none] {$j'$}
	node[right, draw=none]  {$i + 1$}
};
\end{tikzpicture} & 
%%%%%%%%%% Mirror of local move on T %%%%%%%%%%%%%%%
\begin{tikzpicture}[level distance=2cm,
level 1/.style={sibling distance=3.5cm},
level 2/.style={sibling distance=1cm}]
\tikzstyle{every node}=[circle, draw, scale=1, inner sep=2pt]
\node (Root)[fill] {}
    child {
    node[fill] {}
    child { node[fill] {} 
	edge from parent 
	node[left, draw=none] {$2n - i + 1$}
	node[right, draw=none]  {$2n - j + 1$}    
    }
	edge from parent 
	node[left, draw=none] {$2n - i$}
	node[right, draw=none]  {$2n - j' + 1$}
};
\end{tikzpicture} \\
    &  &  &  & \\
    \hline
%%%%%%%%%%%%%%%%%%%%%%%%%%%%%%%%%%%%%%%%%%%%%%%%%%
%%%%%%%%%%%%%%%%%%%%%%%%%%%%% 3 %%%%%%%%%%%%%%%%%%%%
%%%%%%%%%%%%%%%%%%%%%%%%%%%%%%%%%%%%%%%%%%%%%%%%%%
    &  &  &  & \\
     %%%%%%%%%%% T %%%%%%%%%%%%%%
  \begin{tikzpicture}[level distance=2cm,
level 1/.style={sibling distance=3.5cm},
level 2/.style={sibling distance=1cm}]
\tikzstyle{every node}=[circle, draw, scale=1, inner sep=2pt]
\node (Root)[fill] {}
    child {
    node[fill] {}
    child { node[fill] {} 
	edge from parent 
	node[left, draw=none] {$i$}
	node[right, draw=none]  {$i + 1$}    
    }
	edge from parent 
	node[left, draw=none] {$j$}
	node[right, draw=none]  {$j'$}
};
\end{tikzpicture} & 
%%%%%%%%%% Mirror of T %%%%%%%%%%%%%%%
  \begin{tikzpicture}[level distance=2cm,
level 1/.style={sibling distance=3.5cm},
level 2/.style={sibling distance=1cm}]
\tikzstyle{every node}=[circle, draw, scale=1, inner sep=2pt]
\node (Root)[fill] {}
    child {
    node[fill] {}
    child { node[fill] {} 
	edge from parent 
	node[left, draw=none] {$2n - i$}
	node[right, draw=none]  {$2n - i + 1$}    
    }
	edge from parent 
	node[left, draw=none] {$2n - j' + 1$}
	node[right, draw=none]  {$2n - j + 1$}
};
\end{tikzpicture} & 
%%%%%%%%%% Local move on mirror of T %%%%%%%%%%%%%%%
  \begin{tikzpicture}[sloped, level distance=2cm,
level 1/.style={sibling distance=3.5cm},
level 2/.style={sibling distance=1cm}]
\tikzstyle{every node}=[circle, draw, scale=1, inner sep=2pt]
\node (Root)[fill] {}
    child {
    node[fill] {} 
    	edge from parent 
	node[ellipse, above, draw=none] {$2n - j' + 1$}
	node[ellipse, below,draw=none]  {$2n - i$}  
}
child {
    node[fill] {}
    	edge from parent 
	node[ellipse, below, draw=none] {$2n - i + 1$}
	node[ellipse, above, draw=none]  {$2n - j + 1$}  
};
\end{tikzpicture} &
%%%%%%%%%%% Local move on T %%%%%%%%%%%%%%%
  \begin{tikzpicture}[sloped, level distance=2cm,
level 1/.style={sibling distance=3.5cm},
level 2/.style={sibling distance=1cm}]
\tikzstyle{every node}=[circle, draw, scale=1, inner sep=2pt]
\node (Root)[fill] {}
    child {
    node[fill] {} 
    	edge from parent 
	node[ellipse, above, draw=none] {$j$}
	node[ellipse, below,draw=none]  {$i$}  
}
child {
    node[fill] {}
    	edge from parent 
	node[ellipse, below, draw=none] {$i + 1$}
	node[ellipse, above, draw=none]  {$j'$}  
};
\end{tikzpicture} & 
%%%%%%%%%% Mirror of local move on T %%%%%%%%%%%%%%%
  \begin{tikzpicture}[sloped, level distance=2cm,
level 1/.style={sibling distance=3.5cm},
level 2/.style={sibling distance=1cm}]
\tikzstyle{every node}=[circle, draw, scale=1, inner sep=2pt]
\node (Root)[fill] {}
    child {
    node[fill] {} 
    	edge from parent 
	node[ellipse, above, draw=none] {$2n - j' + 1$}
	node[ellipse, below,draw=none]  {$2n - i$}  
}
child {
    node[fill] {}
    	edge from parent 
	node[ellipse, below, draw=none] {$2n - i + 1$}
	node[ellipse, above, draw=none]  {$2n - j + 1$}  
};
\end{tikzpicture} \\
    &  &  &  & \\
      \hline
%%%%%%%%%%%%%%%%%%%%%%%%%%%%%%%%%%%%%%%%%%%%%%%%%%
%%%%%%%%%%%%%%%%%%%%%%%%%%% 4 %%%%%%%%%%%%%%%%%%%%%%
%%%%%%%%%%%%%%%%%%%%%%%%%%%%%%%%%%%%%%%%%%%%%%%%%%
    &  &  &  & \\
%%%%%%%%%%% T %%%%%%%%%%%%%%
  \begin{tikzpicture}[sloped, level distance=2cm,
level 1/.style={sibling distance=3.5cm},
level 2/.style={sibling distance=1cm}]
\tikzstyle{every node}=[circle, draw, scale=1, inner sep=2pt]
\node (Root)[fill] {}
    child {
    node[fill] {} 
    	edge from parent 
	node[ellipse, above, draw=none] {$j$}
	node[ellipse, below, draw=none]  {$i$}  
}
child {
    node[fill] {}
    	edge from parent 
	node[ellipse, below, draw=none] {$i + 1$}
	node[ellipse, above, draw=none]  {$j'$}  
};
\end{tikzpicture} & 
%%%%%%%%%% Mirror of T %%%%%%%%%%%%%%%
  \begin{tikzpicture}[sloped, level distance=2cm,
level 1/.style={sibling distance=3.5cm},
level 2/.style={sibling distance=1cm}]
\tikzstyle{every node}=[circle, draw, scale=1, inner sep=2pt]
\node (Root)[fill] {}
    child {
    node[fill] {} 
    	edge from parent 
	node[ellipse, above, draw=none] {$2n - j' + 1$}
	node[ellipse, below, draw=none]  {$2n - i$} 
}
child {
    node[fill] {}
    	edge from parent 
	node[ellipse, below, draw=none] {$2n - i + 1$}
	node[ellipse, above, draw=none]  {$2n - j + 1$}   
};
\end{tikzpicture} &
%%%%%%%%%% Local move on mirror of T %%%%%%%%%%%%%%%
  \begin{tikzpicture}[level distance=2cm,
level 1/.style={sibling distance=3.5cm},
level 2/.style={sibling distance=1cm}]
\tikzstyle{every node}=[circle, draw, scale=1, inner sep=2pt]
\node (Root)[fill] {}
    child {
    node[fill] {}
    child { node[fill] {} 
	edge from parent 
	node[left, draw=none] {$2n - i$}
	node[right, draw=none]  {$2n - i + 1$}
    }
	edge from parent 
	node[left, draw=none] {$2n - j' + 1$}    
	node[right, draw=none]  {$2n - j + 1$}
};
\end{tikzpicture} & 
%%%%%%%%%%% Local move on T %%%%%%%%%%%%%%%
  \begin{tikzpicture}[level distance=2cm,
level 1/.style={sibling distance=3.5cm},
level 2/.style={sibling distance=1cm}]
\tikzstyle{every node}=[circle, draw, scale=1, inner sep=2pt]
\node (Root)[fill] {}
    child {
    node[fill] {}
    child { node[fill] {} 
	edge from parent 
	node[left, draw=none] {$i$}
	node[right, draw=none]  {$i + 1$}    
    }
	edge from parent 
	node[left, draw=none] {$j$}
	node[right, draw=none]  {$j'$}
};
\end{tikzpicture} & 
%%%%%%%%%% Mirror of local move on T %%%%%%%%%%%%%%%
  \begin{tikzpicture}[level distance=2cm,
level 1/.style={sibling distance=3.5cm},
level 2/.style={sibling distance=1cm}]
\tikzstyle{every node}=[circle, draw, scale=1, inner sep=2pt]
\node (Root)[fill] {}
    child {
    node[fill] {}
    child { node[fill] {} 
	edge from parent 
	node[left, draw=none] {$2n - i$}
	node[right, draw=none]  {$2n - i + 1$}   
    }
	edge from parent 
	node[left, draw=none] {$2n - j' + 1$}
	node[right, draw=none]  {$2n - j + 1$}    
};
\end{tikzpicture} \\
    &  &  &  & \\
  \hline
%%%%%%%%%%%%%%%%%%%%%%%%%%%%%%%%%%%%%%%%%%%%%%%%%%%
%%%%%%%%%%%%%%%%%%%%%%%%%%%%%%%%%%%%%%%%%%%%%%%%%%%
%%%%%%%%%%%%%%%%%%%%%%%%%%%%%%%%%%%%%%%%%%%%%%%%%%%
\end{tabular}}
\caption{Identical results from $s_{i}^C$-local move on mirror image of $T'$ and mirror image of $s_{i}^C$-local move on $T'$ \label{tab:mirrorchart}}
\end{table}
%%%%%%%%%%%%%%%%%%%%%%%%%%%%%%%%%%%%%%%%%%%%%%%%%%%
%%%%%%%%%%%%%%%%%%%%%%% Mirror chart ends %%%%%%%%%%%%%%%%%%%
%%%%%%%%%%%%%%%%%%%%%%%%%%%%%%%%%%%%%%%%%%%%%%%%%%%

Since these are the only edges changed by the local move, all the other edges will still satisfy the symmetry condition.  We conclude that $s_i^C(T)$ is symmetric whenever $T$ is symmetric.  The result follows.
\end{proof}

Next we prove there is exactly one connected component of symmetric plane trees in $G^C$ by showing that each symmetric plane tree can be transformed via $s_i^C$-local moves to one with the leaf $e(1,2)$ and then using induction.

\begin{theorem}
If $T$ and $T'$ are symmetric plane trees then there is a finite sequence of $s_i^C$-local moves that transforms $T$ into $T'$. 
\end{theorem}

\begin{proof}
The proof is by induction on the total number $n$ of edges in a plane tree. 

There are two base cases.  The case when $n=2$ was addressed in Figure~\ref{fig:S2}  since $s_2^C=s_2$ in that setting; it is reproduced in type $C$ notation in Figure~\ref{fig:symmetrybasecase}(a).  The case when $n=3$ has three symmetric plane trees as shown in Figure~\ref{fig:symmetrybasecase}(b): the top and the middle are connected by the edge $s_3^C=s_3$ while the middle and the bottom are connected by $s_2^C=s_2s_4$.  

%%%%%%%%%%%%%%%%%%%%%%%%%%%%%%%%%%%%%%%%%%%%%%%%
%%%%%%%%%%%%%% Symmetry base case diagram begins %%%%%%%%%%%%%%%
%%%%%%%%%%%%%%%%%%%%%%%%%%%%%%%%%%%%%%%%%%%%%%%%
% Saving height of n = 3 orbit
\newsavebox{\tempboxSYMM}
\sbox{\tempboxSYMM}{%
\begin{tikzpicture}[scale = 0.5, every node/.style={scale=0.6}]
\node (lattice) {
\begin{tabular}{c c c c c}
%%%%%%%%%% Line 1 %%%%%%%%%%%
%%%%%%%%%% Line 1 %%%%%%%%%%%
& %%%%next column
& %%%%next column
%%%%%%%%%%%%%%%% Maximal tree begins %%%%%%%%%%%%%%%%%%%%%%%
$ \begin{gathered}
\\
\begin{tikzpicture}[level distance=1cm,
level 1/.style={sibling distance=1cm},
level 2/.style={sibling distance=1cm}]
level 3/.style={sibling distance=1cm}]
\tikzstyle{every node}=[circle, draw, scale=.8, inner sep=2pt]
\node[inner sep=2pt] (Root) {}
child {
    child {
    node[fill] {}
    child { 
    node[fill] {} 
	edge from parent 
	node[left, draw=none] {\tiny $3$}
	node[right, draw=none]  {\tiny $4$}    
    }
	edge from parent 
	node[left, draw=none] {\tiny $2$}
	node[right, draw=none]  {\tiny $5$}  
}	
	    node[fill] {}
	edge from parent 
	node[left, draw=none] {\tiny $1$}
	node[right, draw=none]  {\tiny $6$}    
};
\end{tikzpicture} \\
\tiny
\begin{tabular}{ | c | c | c | }
  \hline
  $1$ & $2$ & $3$ \\ \hline
  $4$ & $5$ & $6$ \\ \hline
\end{tabular}
\end{gathered} $
\vspace{10pt}
%%%%%%%%%%%%%%%% Maximal tree ends %%%%%%%%%%%%%%%%%%%%%%%%
& %%%%next column
& %%%%next column
\\
%%%%%%%%%% Line 2 %%%%%%%%%%%
%%%%%%%%%% Line 2 %%%%%%%%%%%
& %%%%next column
& %%%%next column
%%%%%%%%%%%%%%%%%%%% Arrow begins %%%%%%%%%%%%%%%%%%%%%%%%
$ \begin{gathered}
\begin{tikzpicture}
  [scale = 1.5]
  \node (n1) at (0, 0) {};
  \node (n2) at (0, -1) {};
  \node [overlay] at (0.15, -0.5) {\small $s_{3}^C$};
  \draw[-] (n1) -- (n2);
\end{tikzpicture}
\end{gathered} $
\vspace{10pt}
%%%%%%%%%%%%%%%%%%%% Arrow ends %%%%%%%%%%%%%%%%%%%%%%%%
& %%%%next column
& %%%%next column
\\  
%%%%%%%%%% Line 3 %%%%%%%%%%%
%%%%%%%%%% Line 3 %%%%%%%%%%%
& %%%%next column
& %%%%next column
%%%%%%%%%%%%%%%%%%% Y tree begins %%%%%%%%%%%%%%%%%%%%%%
$ \begin{gathered}
\begin{tikzpicture}[level distance=1cm,
level 1/.style={sibling distance=1cm},
level 2/.style={sibling distance=2cm}]
\tikzstyle{every node}=[circle, draw, scale=.8, inner sep=2pt]
\vspace{10pt}
\node[inner sep=2pt] (Root) {}
child {
    child {
    node[fill] {}
	edge from parent 
	node[sloped, ellipse, above, draw=none]  {\tiny $2$}
	node[sloped, ellipse, below, draw=none]  {\tiny $3$}  
}	
    child { 
    node[fill] {} 
	edge from parent 
	node[sloped, ellipse, above, draw=none] {\tiny $5$}
	node[sloped, ellipse, below, draw=none]  {\tiny $4$}    
    }
	    node[fill] {}
	edge from parent 
	node[left, draw=none] {\tiny $1$}
	node[right, draw=none]  {\tiny $6$}    
};
\end{tikzpicture} \\
\tiny
\begin{tabular}{ | c | c | c | }
  \hline
  $1$ & $2$ & $4$ \\ \hline
  $3$ & $5$ & $6$ \\ \hline
\end{tabular}
\end{gathered} $
\vspace{10pt}
%%%%%%%%%%%%%%%% Y tree ends %%%%%%%%%%%%%%%%%%%%%%%%%%%
& %%%%next column
& %%%%next column
\\  
%%%%%%%%%% Line 4 %%%%%%%%%%%
%%%%%%%%%% Line 4 %%%%%%%%%%%
& %%%%next column
& %%%%next column
%%%%%%%%%%%%%%%%%%%% Arrow begins %%%%%%%%%%%%%%%%%%%%%%%%
$ \begin{gathered}
\begin{tikzpicture}
  [scale = 1.5]
  \node (n1) at (0, 0) {};
  \node (n2) at (0, -1) {};
  \node [overlay] at (0.15, -0.5) {\small $s_{2}^C$};
  \draw[-] (n1) -- (n2);
\end{tikzpicture}
\end{gathered} $
\vspace{10pt}
%%%%%%%%%%%%%%%%%%%% Arrow ends %%%%%%%%%%%%%%%%%%%%%%%%
& %%%%next column
& %%%%next column
\\ 
%%%%%%%%%% Line 5 %%%%%%%%%%%
%%%%%%%%%% Line 5 %%%%%%%%%%%
& %%%%next column
& %%%%next column
%%%%%%%%%%%%%%%%%%% Minimal tree begins %%%%%%%%%%%%%%%%%%%%%
$ \begin{gathered}
\begin{tikzpicture}[level distance=1cm,
level 1/.style={sibling distance=1.5cm},
level 2/.style={sibling distance=1cm}]
\tikzstyle{every node}=[circle, draw, scale=.8, inner sep=2pt]
\node[inner sep=2pt] (root) {}
    child {
    node[fill] {} 
    	edge from parent 
	node[sloped, ellipse, above, draw=none] {\tiny $1$}
	node[sloped, ellipse, below, draw=none]  {\tiny $2$}  
    }
    child {
    node[fill] {}
    	edge from parent 
	node[left, draw=none] {\tiny $3$}
	node[right, draw=none]  {\tiny $4$}  
    }
    child {
    node[fill] {}
    	edge from parent 
	node[sloped, ellipse, below, draw=none] {\tiny $5$}
	node[sloped, ellipse, above, draw=none]  {\tiny $6$}  
    };
\end{tikzpicture} \\
\tiny
\begin{tabular}{ | c | c | c | }
  \hline
  $1$ & $3$ & $5$ \\ \hline
  $2$ & $4$ & $6$ \\ \hline
\end{tabular}
\end{gathered} $
\vspace{10pt}
%%%%%%%%%%%%%%%%%%% Minimal tree ends %%%%%%%%%%%%%%%%%%%%%
\end{tabular}
};
\end{tikzpicture}
}

%------------------------------------------------------------------------------------------%
%------------------------------------------------------------------------------------------%
%------------------------------------------------------------------------------------------%
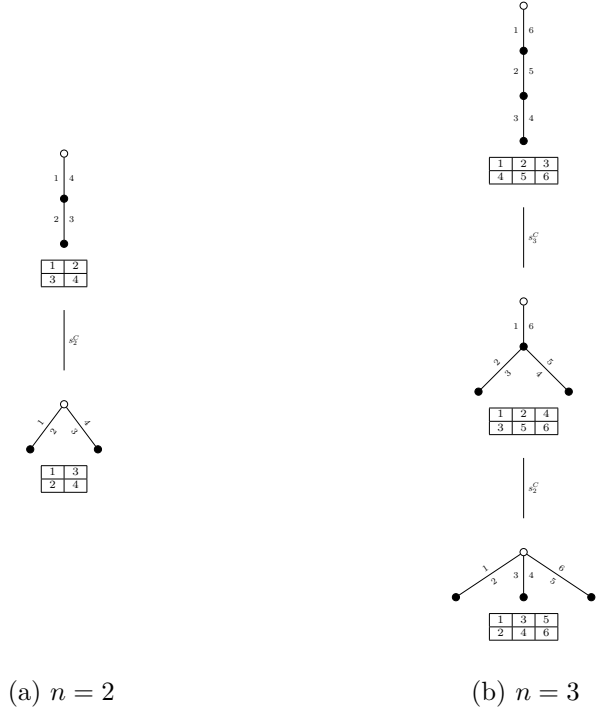
\begin{figure}[H]
\begin{center}
%%%%%%%%%%%%%%%%%%% n = 2 begins %%%%%%%%%%%%%%%%%%%%%
\begin{subfigure}[b]{0.4\textwidth}  
\centering
\vbox to\ht\tempboxSYMM{
\vfill
\begin{tikzpicture}[scale = 0.5, every node/.style={scale=0.6}]
\node (lattice) {
\begin{tabular}{c}
%%%%%%%%%% Line 1 %%%%%%%%%%%
%%%%%%%%%% Line 1 %%%%%%%%%%%
%%%%%%%%%%%%%%%% Maximal tree begins %%%%%%%%%%%%%%%%%%%%%%%
$ \begin{gathered}
\\
\begin{tikzpicture}[level distance=1cm,
level 1/.style={sibling distance=1cm},
level 2/.style={sibling distance=1cm}]
level 3/.style={sibling distance=1cm}]
\tikzstyle{every node}=[circle, draw, scale=.8, inner sep=2pt]
\node[inner sep=2pt] (Root) {}
child {
    child {
    node[fill] {}
	edge from parent 
	node[left, draw=none] {\tiny $2$}
	node[right, draw=none]  {\tiny $3$}  
}	
	node[fill] {}
	edge from parent 
	node[left, draw=none] {\tiny $1$}
	node[right, draw=none]  {\tiny $4$}    
};
\end{tikzpicture} \\
\tiny
\begin{tabular}{ | c | c | }
  \hline
  $1$ & $2$ \\ \hline
  $3$ & $4$ \\ \hline
\end{tabular}
\end{gathered} $
\vspace{10pt}
%%%%%%%%%%%%%%%% Maximal tree ends %%%%%%%%%%%%%%%%%%%%%%%%
\\
%%%%%%%%%% Line 2 %%%%%%%%%%%
%%%%%%%%%% Line 2 %%%%%%%%%%%
%%%%%%%%%%%%%%%%%%%% Arrow begins %%%%%%%%%%%%%%%%%%%%%%%%
$ \begin{gathered}
\begin{tikzpicture}
  [scale = 1.5]
  \node (n1) at (0, 0) {};
  \node (n2) at (0, -1) {};
  \node [overlay] at (0.15, -0.5) {\small $s_{2}^C$};
  \draw[-] (n1) -- (n2);
\end{tikzpicture}
\end{gathered} $
\vspace{10pt}
%%%%%%%%%%%%%%%%%%%% Arrow ends %%%%%%%%%%%%%%%%%%%%%%%%
\\
%%%%%%%%%% Line 5 %%%%%%%%%%%
%%%%%%%%%% Line 5 %%%%%%%%%%%
%%%%%%%%%%%%%%%%%%% Minimal tree begins %%%%%%%%%%%%%%%%%%%%%
$ \begin{gathered}
\begin{tikzpicture}[level distance=1cm,
level 1/.style={sibling distance=1.5cm},
level 2/.style={sibling distance=1cm}]
\tikzstyle{every node}=[circle, draw, scale=.8, inner sep=2pt]
\node[inner sep=2pt] (root) {}
    child {
    node[fill] {} 
    	edge from parent 
	node[sloped, ellipse, above, draw=none] {\tiny $1$}
	node[sloped, ellipse, below, draw=none]  {\tiny $2$}  
    }
    child {
    node[fill] {}
    	edge from parent 
	node[sloped, ellipse, below, draw=none] {\tiny $3$}
	node[sloped, ellipse, above, draw=none]  {\tiny $4$}  
    };
\end{tikzpicture} \\
\tiny
\begin{tabular}{ | c | c | }
  \hline
  $1$ & $3$ \\ \hline
  $2$ & $4$ \\ \hline
\end{tabular}
\end{gathered} $
\vspace{10pt}
%%%%%%%%%%%%%%%%%%% Minimal tree ends %%%%%%%%%%%%%%%%%%%%%
\end{tabular}
};
\end{tikzpicture}
\vfill
}
\caption{$n = 2$ \label{fig:n2symm}}
%%%%%%%%%%%%%%%%%%% n = 2 ends %%%%%%%%%%%%%%%%%%%%%%
%%%%%%%%%%%%%%%%%%%%%%%%%%%%%%%%%%%%%%%%%%
%%%%%%%%%%%%%%%%%%%%%%%%%%%%%%%%%%%%%%%%%%
%%%%%%%%%%%%%%%%%%%%%%%%%%%%%%%%%%%%%%%%%%
%%%%%%%%%%%%%%%%%%%%%%%%%%%%%%%%%%%%%%%%%%
%%%%%%%%%%%%%%%%%%% n = 3 begins %%%%%%%%%%%%%%%%%%%%%
\end{subfigure}
\begin{subfigure}[b]{0.4\textwidth}  
\centering
\usebox{\tempboxSYMM}
\caption{$n = 3$ \label{fig:n3symm}} 
\end{subfigure}
%%%%%%%%%%%%%%%%%%% n = 3 ends %%%%%%%%%%%%%%%%%%%%%%
\end{center}
\caption{Type $C$ base cases for symmetric plane trees \label{fig:symmetrybasecase}}
\end{figure}
%%%%%%%%%%%%%%%%%%%%%%%%%%%%%%%%%%%%%%%%%%%%%%%%
%%%%%%%%%%%%%% Symmetry base case diagram ends %%%%%%%%%%%%%%%
%%%%%%%%%%%%%%%%%%%%%%%%%%%%%%%%%%%%%%%%%%%%%%%%

For the induction step, assume that any two symmetric plane trees with at most $n-1$ edges can be transformed into each other by a sequence of $s_i^C$-local moves.   Now consider a symmetric plane tree with $n$ edges.

First we show that there is a path of $s_i^C$-local moves from each plane tree $T$ to a plane tree containing the edge $e(1,2)$.  If $T$ does not have the edge $e(1,2)$ then it has the edge $e(1,j)$ for some $j \geq 3$.  This means that $1$ and $2$ are both in the top row of the tableau $\phi(T)$.  Let $k$ be the first integer {\em not} in the top row of $\phi(T)$. Since $\phi(T)$ has shape $(n,n)$ we know that $k \leq n+1$. Proposition~\ref{proposition: connected} showed that the standard tableau $s_2s_3\cdots s_{k-1}(\phi(T))$ has $2$ in the bottom row by way of $s_i$-local moves of type $A$.  We now confirm that $s_2^Cs_3^C\cdots s_{k-1}^C$ also moves $2$ to the bottom row.  If $k=n+1$ then the top row of the tableau is filled with the integers from $1$ through $n$ and $s_{k-1}^C=s_n$ simply exchanges $n$ and $n+1$.  For $j \leq n$  we know that $s_2^Cs_3^C\cdots s_{j-1}^C$ permutes numbers within the disjoint sets $\{1,2,\ldots,j\}$ and $\{j+1,\ldots,2n\}$ independently.  So the tableau $s_2^Cs_3^C\cdots s_{k-1}^C(\phi(T))$ has $2$ on the bottom row for all $k \leq n+1$.  We therefore conclude that $T$ is in the same connected component of $G^C$ as a plane tree with the edge $e(1,2)$.

We next show that all symmetric plane trees are in the same connected component of $G^C$.  Suppose $T$ and $T'$ are both symmetric plane trees.  By the previous argument, we can assume that they each contain the leaf $e(1,2)$ and hence by symmetry the leaf $e(2n-1,2n)$. Since these edges are both leaves, they can be erased without disconnecting the two trees.  Consider the subtrees $T_1$ and $T_1'$ consisting respectively of all the edges of $T$ and $T'$ except $e(1,2)$ and $e(2n-1,2n)$.  The two subtrees are still symmetric but have only $n-2$ edges.  By the inductive hypothesis we can transform $T_1$ into $T_1'$ with a sequence of $s_i^C$-local moves, which also transforms $T$ into $T'$.  By induction the claim is proven.
\end{proof}

The proof for asymmetric plane trees is somewhat similar but more subtle.

\begin{theorem}
If $T$ and $T'$ are asymmetric plane trees then there is a finite sequence of $s_i^C$-local moves that transforms $T$ into $T'$. 
\end{theorem}

\begin{proof}
The proof is by induction on the total number of edges $n$ in a plane tree. 

The base cases for asymmetric plane trees occur when $n=3$ and when $n=4$.  There are two asymmetric plane trees with three edges, and these trees are related by $s_2^C=s_2s_4$ as shown in Figure~\ref{fig:asymmetrybasecase}(a).  There are eight asymmetric plane trees with four edges, and these trees are related by the $s_i^C$-local moves shown in Figure~\ref{fig:asymmetrybasecase}(b).

%%%%%%%%%%%%%%%%%%%%%%%%%%%%%%%%%%%%%%%%%%%%%%%%
%%%%%%%%%%%%%% Asymmetry base case diagram begins %%%%%%%%%%%%%%
%%%%%%%%%%%%%%%%%%%%%%%%%%%%%%%%%%%%%%%%%%%%%%%%
% Saving height of n = 4 orbit
\newsavebox{\tempboxASYMM}
\sbox{\tempboxASYMM}{%
\begin{tikzpicture}[scale = 0.5, every node/.style={scale=0.6}]
\node (lattice) {
\begin{tabular}{c c c}
%%%%%%%%%% Line 1 %%%%%%%%%%%
%%%%%%%%%% Line 1 %%%%%%%%%%%
& %%%%next column
& %%%%next column
%%%%%%%%%%%%%%%%%%% Top tree begins %%%%%%%%%%%%%%%%%%%%%%
$ \begin{gathered}
\\
\begin{tikzpicture}[level distance=1cm,
level 1/.style={sibling distance=2cm},
level 2/.style={sibling distance=2cm}]
\tikzstyle{every node}=[circle, draw, scale=.8, inner sep=2pt]
\node[inner sep=2pt] (Root) {}
child {
    child {
    	child {
	node[fill] {}
	edge from parent 
	node[left, draw=none]  {\tiny $3$}
	node[right, draw=none]  {\tiny $4$}  
	}
    	node[fill] {}
	edge from parent 
	node[sloped, ellipse, above, draw=none]  {\tiny $2$}
	node[sloped, ellipse, below, draw=none]  {\tiny $5$}  
}	
    child { 
    	node[fill] {} 
	edge from parent 
	node[sloped, ellipse, above, draw=none] {\tiny $7$}
	node[sloped, ellipse, below, draw=none]  {\tiny $6$}    
    }
	node[fill] {}
	edge from parent 
	node[left, draw=none] {\tiny $1$}
	node[right, draw=none]  {\tiny $8$}    
};
\end{tikzpicture} \\
\tiny
\begin{tabular}{ | c | c | c | c | }
  \hline
  $1$ & $2$ & $3$ & $6$ \\ \hline
  $4$ & $5$ & $7$ & $8$ \\ \hline
\end{tabular}
\end{gathered} $
\vspace{10pt}
%%%%%%%%%%%%%%%% Top tree ends %%%%%%%%%%%%%%%%%%%%%%%%%%%
\\
%%%%%%%%%% Line 2 %%%%%%%%%%%
%%%%%%%%%% Line 2 %%%%%%%%%%%
& %%%%next column
%%%%%%%%%%%%%%%%%%%% Arrow begins %%%%%%%%%%%%%%%%%%%%%%%%
$ \begin{gathered}
\begin{tikzpicture}
  [scale = 1.5]
  \node (n1) at (0, 0) {};
  \node (n2) at (1, 1) {};
  \draw[-] (n1) -- (n2) node[sloped, midway, above] {\small $s_{3}^C$};
\end{tikzpicture}
\end{gathered} $
\vspace{10pt}
%%%%%%%%%%%%%%%%%%%% Arrow ends %%%%%%%%%%%%%%%%%%%%%%%%
& %%%%next column
%%%%%%%%%%%%%%%%%%%% Arrow begins %%%%%%%%%%%%%%%%%%%%%%%%
$ \begin{gathered}
\begin{tikzpicture}
  [scale = 1.5]
  \node (n1) at (0, 0) {};
  \node (n2) at (0, -1) {};
  \node [overlay] at (0.15, -0.5) {\small $s_{2}^C$};
  \draw[-] (n1) -- (n2);
\end{tikzpicture}
\end{gathered} $
\vspace{10pt}
%%%%%%%%%%%%%%%%%%%% Arrow ends %%%%%%%%%%%%%%%%%%%%%%%%
\\
%%%%%%%%%% Line 3 %%%%%%%%%%%
%%%%%%%%%% Line 3 %%%%%%%%%%%
%%%%%%%%%%%%%%%%%%% Left tree begins %%%%%%%%%%%%%%%%%%%%%%
$ \begin{gathered}
\begin{tikzpicture}[level distance=1cm,
level 1/.style={sibling distance=2cm},
level 2/.style={sibling distance=2cm}]
\tikzstyle{every node}=[circle, draw, scale=.8, inner sep=2pt]
\node[inner sep=2pt] (Root) {}
child {
    child {
    	node[fill] {}
	edge from parent 
	node[sloped, ellipse, above, draw=none]  {\tiny $2$}
	node[sloped, ellipse, below, draw=none]  {\tiny $3$}  
	}	
    child { 
        	child {
		node[fill] {}
		edge from parent 
		node[left, draw=none]  {\tiny $5$}
		node[right, draw=none]  {\tiny $6$}  
		}
    	node[fill] {} 
	edge from parent 
	node[sloped, ellipse, above, draw=none] {\tiny $7$}
	node[sloped, ellipse, below, draw=none]  {\tiny $4$}    
    	}
	node[fill] {}
	edge from parent 
	node[left, draw=none] {\tiny $1$}
	node[right, draw=none]  {\tiny $8$}    
};
\end{tikzpicture} \\
\tiny
\begin{tabular}{ | c | c | c | c | }
  \hline
  $1$ & $2$ & $4$ & $5$ \\ \hline
  $3$ & $6$ & $7$ & $8$ \\ \hline
\end{tabular}
\end{gathered} $
\vspace{10pt}
%%%%%%%%%%%%%%%%%%% Left tree ends %%%%%%%%%%%%%%%%%%%%%%
& %%%%next column
& %%%%next column
%%%%%%%%%%%%%%%%%%% Right tree begins %%%%%%%%%%%%%%%%%%%%%%
$ \begin{gathered}
\begin{tikzpicture}[level distance=1cm,
level 1/.style={sibling distance=2cm},
level 2/.style={sibling distance=1cm}]
\tikzstyle{every node}=[circle, draw, scale=.8, inner sep=2pt]
\node[inner sep=2pt] (Root) {}
child {
	child {
		child {
			node[fill] {}
			edge from parent 
			node[left, draw=none]  {\tiny $3$}
			node[right, draw=none]  {\tiny $4$}  		
		}
		node[fill] {}
		edge from parent 
		node[left, draw=none]  {\tiny $2$}
		node[right, draw=none]  {\tiny $5$}  
	}
    	node[fill] {}
	edge from parent 
	node[sloped, ellipse, above, draw=none]  {\tiny $1$}
	node[sloped, ellipse, below, draw=none]  {\tiny $6$}  
}
child {
	node[fill] {}
	edge from parent 
	node[sloped, ellipse, below, draw=none] {\tiny $7$}
	node[sloped, ellipse, above, draw=none]  {\tiny $8$}    
};
\end{tikzpicture} \\
\tiny
\begin{tabular}{ | c | c | c | c | }
  \hline
  $1$ & $2$ & $3$ & $7$ \\ \hline
  $4$ & $5$ & $6$ & $8$ \\ \hline
\end{tabular}
\end{gathered} $
\vspace{10pt}
%%%%%%%%%%%%%%%%%%% Right tree ends %%%%%%%%%%%%%%%%%%%%%%
\\
%%%%%%%%%% Line 4 %%%%%%%%%%%
%%%%%%%%%% Line 4 %%%%%%%%%%%
%%%%%%%%%%%%%%%%%%%% Arrow begins %%%%%%%%%%%%%%%%%%%%%%%%
$ \begin{gathered}
\begin{tikzpicture}
  [scale = 1.5]
  \node (n1) at (0, 0) {};
  \node (n2) at (0, -1) {};
  \node [overlay] at (0.15, -0.5) {\small $s_{2}^C$};
  \draw[-] (n1) -- (n2);
\end{tikzpicture}
\end{gathered} $
\vspace{10pt}
%%%%%%%%%%%%%%%%%%%% Arrow ends %%%%%%%%%%%%%%%%%%%%%%%%
& %%%%next column
& %%%%next column
%%%%%%%%%%%%%%%%%%%% Arrow begins %%%%%%%%%%%%%%%%%%%%%%%%
$ \begin{gathered}
\begin{tikzpicture}
  [scale = 1.5]
  \node (n1) at (0, 0) {};
  \node (n2) at (0, -1) {};
  \node [overlay] at (0.15, -0.5) {\small $s_{3}^C$};
  \draw[-] (n1) -- (n2);
\end{tikzpicture}
\end{gathered} $
\vspace{10pt}
%%%%%%%%%%%%%%%%%%%% Arrow ends %%%%%%%%%%%%%%%%%%%%%%%%
\\  
%%%%%%%%%% Line 5 %%%%%%%%%%%
%%%%%%%%%% Line 5 %%%%%%%%%%%
%%%%%%%%%%%%%%%%%%% Left tree begins %%%%%%%%%%%%%%%%%%%%%%
$ \begin{gathered}
\begin{tikzpicture}[level distance=1cm,
level 1/.style={sibling distance=2cm},
level 2/.style={sibling distance=1cm}]
\tikzstyle{every node}=[circle, draw, scale=.8, inner sep=2pt]
\node[inner sep=2pt] (Root) {}
child {
	node[fill] {}
	edge from parent 
	node[sloped, ellipse, below, draw=none] {\tiny $2$}
	node[sloped, ellipse, above, draw=none]  {\tiny $1$}    
}
child {
	child {
		child {
			node[fill] {}
			edge from parent 
			node[left, draw=none]  {\tiny $5$}
			node[right, draw=none]  {\tiny $6$}  		
		}
		node[fill] {}
		edge from parent 
		node[left, draw=none]  {\tiny $4$}
		node[right, draw=none]  {\tiny $7$}  
	}
    	node[fill] {}
	edge from parent 
	node[sloped, ellipse, above, draw=none]  {\tiny $8$}
	node[sloped, ellipse, below, draw=none]  {\tiny $3$}  
};
\end{tikzpicture} \\
\tiny
\begin{tabular}{ | c | c | c | c | }
  \hline
  $1$ & $3$ & $4$ & $5$ \\ \hline
  $2$ & $6$ & $7$ & $8$ \\ \hline
\end{tabular}
\end{gathered} $
\vspace{10pt}
%%%%%%%%%%%%%%%% Left tree ends %%%%%%%%%%%%%%%%%%%%%%%%%%%
& %%%%next column
& %%%%next column
%%%%%%%%%%%%%%%%%%% Right tree begins %%%%%%%%%%%%%%%%%%%%%%
$ \begin{gathered}
\begin{tikzpicture}[level distance=1cm,
level 1/.style={sibling distance=2cm},
level 2/.style={sibling distance=2cm}]
\tikzstyle{every node}=[circle, draw, scale=.8, inner sep=2pt]
\node[inner sep=2pt] (Root) {}
child {
	child {
		child { % Dummy child to keep graphics vertically aligned
			node[draw=none] {}
			edge from parent[draw=none]
		}
		node[fill] {}
		edge from parent 
		node[sloped, ellipse, above, draw=none]  {\tiny $2$}
		node[sloped, ellipse, below, draw=none]  {\tiny $3$}  
	}
	child {
		node[fill] {}
		edge from parent 
		node[sloped, ellipse, above, draw=none]  {\tiny $5$}
		node[sloped, ellipse, below, draw=none]  {\tiny $4$}  		
	}
    	node[fill] {}
	edge from parent 
	node[sloped, ellipse, above, draw=none]  {\tiny $1$}
	node[sloped, ellipse, below, draw=none]  {\tiny $6$}  
}
child {
	node[fill] {}
	edge from parent 
	node[sloped, ellipse, below, draw=none] {\tiny $7$}
	node[sloped, ellipse, above, draw=none]  {\tiny $8$}    
};
\end{tikzpicture} \\
\tiny
\begin{tabular}{ | c | c | c | c | }
  \hline
  $1$ & $2$ & $4$ & $7$ \\ \hline
  $3$ & $5$ & $6$ & $8$ \\ \hline
\end{tabular}
\end{gathered} $
\vspace{10pt}
%%%%%%%%%%%%%%%% Right tree ends %%%%%%%%%%%%%%%%%%%%%%%%%%%
\\  
%%%%%%%%%% Line 6 %%%%%%%%%%%
%%%%%%%%%% Line 6 %%%%%%%%%%%
%%%%%%%%%%%%%%%%%%%% Arrow begins %%%%%%%%%%%%%%%%%%%%%%%%
$ \begin{gathered}
\begin{tikzpicture}
  [scale = 1.5]
  \node (n1) at (0, 0) {};
  \node (n2) at (0, -1) {};
  \node [overlay] at (0.15, -0.5) {\small $s_{3}^C$};
  \draw[-] (n1) -- (n2);
\end{tikzpicture}
\end{gathered} $
\vspace{10pt}
%%%%%%%%%%%%%%%%%%%% Arrow ends %%%%%%%%%%%%%%%%%%%%%%%%
& %%%%next column
%%%%%%%%%%%%%%%%%%%% Arrow begins %%%%%%%%%%%%%%%%%%%%%%%%
$ \begin{gathered}
\begin{tikzpicture}
  [scale = 1.5]
  \node (n1) at (0, 0) {};
  \node (n2) at (-1, -1) {};
  \draw[-] (n1) -- (n2) node[sloped, midway, above] {\small $s_{2}^C$};
\end{tikzpicture}
\end{gathered} $
\vspace{10pt}
%%%%%%%%%%%%%%%%%%%% Arrow ends %%%%%%%%%%%%%%%%%%%%%%%%
& %%%%next column
%%%%%%%%%%%%%%%%%%%% Arrow begins %%%%%%%%%%%%%%%%%%%%%%%%
$ \begin{gathered}
\begin{tikzpicture}
  [scale = 1.5]
  \node (n1) at (0, 0) {};
  \node (n2) at (0, -1) {};
  \node [overlay] at (0.15, -0.5) {\small $s_{4}^C$};
  \draw[-] (n1) -- (n2);
\end{tikzpicture}
\end{gathered} $
\vspace{10pt}
%%%%%%%%%%%%%%%%%%%% Arrow ends %%%%%%%%%%%%%%%%%%%%%%%%
\\ 
%%%%%%%%%% Line 7 %%%%%%%%%%%
%%%%%%%%%% Line 7 %%%%%%%%%%%
%%%%%%%%%%%%%%%%%%% Left tree begins %%%%%%%%%%%%%%%%%%%%%%
$ \begin{gathered}
\begin{tikzpicture}[level distance=1cm,
level 1/.style={sibling distance=2cm},
level 2/.style={sibling distance=2cm}]
\tikzstyle{every node}=[circle, draw, scale=.8, inner sep=2pt]
\node[inner sep=2pt] (Root) {}
child {
	node[fill] {}
	edge from parent 
	node[sloped, ellipse, below, draw=none] {\tiny $2$}
	node[sloped, ellipse, above, draw=none]  {\tiny $1$}    
}
child {
	child {
		node[fill] {}
		edge from parent 
		node[sloped, ellipse, above, draw=none]  {\tiny $4$}
		node[sloped, ellipse, below, draw=none]  {\tiny $5$}  
	}
	child {
		node[fill] {}
		edge from parent 
		node[sloped, ellipse, above, draw=none]  {\tiny $7$}
		node[sloped, ellipse, below, draw=none]  {\tiny $6$}  		
	}
    	node[fill] {}
	edge from parent 
	node[sloped, ellipse, above, draw=none]  {\tiny $8$}
	node[sloped, ellipse, below, draw=none]  {\tiny $3$}  
};
\end{tikzpicture} \\
\tiny
\begin{tabular}{ | c | c | c | c | }
  \hline
  $1$ & $3$ & $4$ & $6$ \\ \hline
  $2$ & $5$ & $7$ & $8$ \\ \hline
\end{tabular}
\end{gathered} $
\vspace{10pt}
%%%%%%%%%%%%%%%% Left tree ends %%%%%%%%%%%%%%%%%%%%%%%%%%%
& %%%%next column
& %%%%next column
%%%%%%%%%%%%%%%%%%% Right tree begins %%%%%%%%%%%%%%%%%%%%%%
$ \begin{gathered}
\begin{tikzpicture}[level distance=1cm,
level 1/.style={sibling distance=1.5cm},
level 2/.style={sibling distance=2cm}]
\tikzstyle{every node}=[circle, draw, scale=.8, inner sep=2pt]
\node[inner sep=2pt] (Root) {}
child {
	child {
		node[fill] {}
		edge from parent 
		node[left, draw=none]  {\tiny $2$}
		node[right, draw=none]  {\tiny $3$}  
	}
    	node[fill] {}
	edge from parent 
	node[sloped, ellipse, above, draw=none]  {\tiny $1$}
	node[sloped, ellipse, below, draw=none]  {\tiny $4$}  
}
child {
	node[fill] {}
	edge from parent 
	node[left, draw=none] {\tiny $5$}
	node[right, draw=none]  {\tiny $6$}    
}
child {
	node[fill] {}
	edge from parent 
	node[sloped, ellipse, below, draw=none] {\tiny $7$}
	node[sloped, ellipse, above, draw=none]  {\tiny $8$}    
};
\end{tikzpicture} \\
\tiny
\begin{tabular}{ | c | c | c | c | }
  \hline
  $1$ & $2$ & $5$ & $7$ \\ \hline
  $3$ & $4$ & $6$ & $8$ \\ \hline
\end{tabular}
\end{gathered} $
\vspace{10pt}
%%%%%%%%%%%%%%%% Right tree ends %%%%%%%%%%%%%%%%%%%%%%%%%%%
\\  
%%%%%%%%%% Line 8 %%%%%%%%%%%
%%%%%%%%%% Line 8 %%%%%%%%%%%
& %%%%next column
%%%%%%%%%%%%%%%%%%%% Arrow begins %%%%%%%%%%%%%%%%%%%%%%%%
$ \begin{gathered}
\begin{tikzpicture}
  [scale = 1.5]
  \node (n1) at (0, 0) {};
  \node (n2) at (1, -1) {};
  \draw[-] (n1) -- (n2) node[sloped, midway, above] {\small $s_{4}^C$};
\end{tikzpicture}
\end{gathered} $
\vspace{10pt}
%%%%%%%%%%%%%%%%%%%% Arrow ends %%%%%%%%%%%%%%%%%%%%%%%%
& %%%%next column
%%%%%%%%%%%%%%%%%%%% Arrow begins %%%%%%%%%%%%%%%%%%%%%%%%
$ \begin{gathered}
\begin{tikzpicture}
  [scale = 1.5]
  \node (n1) at (0, 0) {};
  \node (n2) at (0, -1) {};
  \node [overlay] at (0.15, -0.5) {\small $s_{2}^C$};
  \draw[-] (n1) -- (n2);
\end{tikzpicture}
\end{gathered} $
\vspace{10pt}
%%%%%%%%%%%%%%%%%%%% Arrow ends %%%%%%%%%%%%%%%%%%%%%%%%
\\ 
%%%%%%%%%% Line 9 %%%%%%%%%%%
%%%%%%%%%% Line 9 %%%%%%%%%%%
& %%%%next column
& %%%%next column
%%%%%%%%%%%%%%%%%%% Bottom tree begins %%%%%%%%%%%%%%%%%%%%%%
$ \begin{gathered}
\begin{tikzpicture}[level distance=1cm,
level 1/.style={sibling distance=1.5cm},
level 2/.style={sibling distance=2cm}]
\tikzstyle{every node}=[circle, draw, scale=.8, inner sep=2pt]
\node[inner sep=2pt] (Root) {}
child {
	node[fill] {}
	edge from parent 
	node[sloped, ellipse, below, draw=none] {\tiny $2$}
	node[sloped, ellipse, above, draw=none]  {\tiny $1$}    
}
child {
	node[fill] {}
	edge from parent 
	node[left, draw=none] {\tiny $3$}
	node[right, draw=none]  {\tiny $4$}    
}
child {
	child {
		node[fill] {}
		edge from parent 
		node[right, draw=none]  {\tiny $7$}
		node[left, draw=none]  {\tiny $6$}  
	}
    	node[fill] {}
	edge from parent 
	node[sloped, ellipse, above, draw=none]  {\tiny $8$}
	node[sloped, ellipse, below, draw=none]  {\tiny $5$}  
};
\end{tikzpicture} \\
\tiny
\begin{tabular}{ | c | c | c | c | }
  \hline
  $1$ & $3$ & $5$ & $6$ \\ \hline
  $2$ & $4$ & $7$ & $8$ \\ \hline
\end{tabular}
\end{gathered} $
\vspace{10pt}
%%%%%%%%%%%%%%%% Bottom tree ends %%%%%%%%%%%%%%%%%%%%%%%%%%%
\end{tabular}
};
\end{tikzpicture}
}

%------------------------------------------------------------------------------------------%
%------------------------------------------------------------------------------------------%
%------------------------------------------------------------------------------------------%
\begin{figure}[H]
\begin{center}
%%%%%%%%%%%%%%%%%%% n = 3 begins %%%%%%%%%%%%%%%%%%%%%
\begin{subfigure}[b]{0.33\textwidth}  
\centering
\vbox to\ht\tempboxASYMM{
\vfill
\begin{tikzpicture}[scale = 0.5, every node/.style={scale=0.6}]
\node (lattice) {
\begin{tabular}{c}
%%%%%%%%%% Line 1 %%%%%%%%%%%
%%%%%%%%%% Line 1 %%%%%%%%%%%
%%%%%%%%%%%%%%%%%%%% Left arm begins %%%%%%%%%%%%%%%%%%%%%%%%
$ \begin{gathered}
\\
\begin{tikzpicture}[level distance=1cm,
level 1/.style={sibling distance=2cm},
level 2/.style={sibling distance=1cm}]
\tikzstyle{every node}=[circle, draw, scale=.8, inner sep=2pt]
\node[inner sep=2pt] (Root) {}
    child {
    node[fill] {}
    child { 
    node[fill] {} 
	edge from parent 
	node[left, draw=none] {\tiny $2$}
	node[right, draw=none]  {\tiny $3$}    
    }
	edge from parent 
	node[sloped, ellipse, above, draw=none] {\tiny $1$}
	node[sloped, ellipse, below, draw=none]  {\tiny $4$}  
	}
child {
    node[fill] {}
    	edge from parent 
	node[sloped, ellipse, below, draw=none] {\tiny $5$}
	node[sloped, ellipse, above, draw=none]  {\tiny $6$}  
};
\end{tikzpicture} \\
\tiny
\begin{tabular}{ | c | c | c | }
  \hline
  $1$ & $2$ & 5 \\ \hline
  $3$ & $4$ & 6 \\ \hline
\end{tabular}
\end{gathered} $
\vspace{10pt}
%%%%%%%%%%%%%%%%%%%% Left arm ends %%%%%%%%%%%%%%%%%%%%%%%%
\\  
%%%%%%%%%% Line 2 %%%%%%%%%%%
%%%%%%%%%% Line 2 %%%%%%%%%%%
%%%%%%%%%%%%%%%%%%%% Arrow begins %%%%%%%%%%%%%%%%%%%%%%%%
$ \begin{gathered}
\begin{tikzpicture}
  [scale = 1.5]
  \node (n1) at (0, 0) {};
  \node (n2) at (0, -1) {};
  \node [overlay] at (0.15, -0.5) {\small $s_{2}^C$};
  \draw[-] (n1) -- (n2);
\end{tikzpicture}
\end{gathered} $
\vspace{10pt}
%%%%%%%%%%%%%%%%%%%% Arrow ends %%%%%%%%%%%%%%%%%%%%%%%%
\\ 
%%%%%%%%%% Line 3 %%%%%%%%%%%
%%%%%%%%%% Line 3 %%%%%%%%%%%
%%%%%%%%%%%%%%%%%%%% Right arm begins %%%%%%%%%%%%%%%%%%%%%%%%
$ \begin{gathered}
\begin{tikzpicture}[level distance=1cm,
level 1/.style={sibling distance=2cm},
level 2/.style={sibling distance=1cm}]
\tikzstyle{every node}=[circle, draw, scale=.8, inner sep=2pt]
\node[inner sep=2pt] (Root) {}
child {
    node[fill] {}
    	edge from parent 
	node[sloped, ellipse, above, draw=none] {\tiny $1$}
	node[sloped, ellipse, below, draw=none]  {\tiny $2$}  
}
    child {
    node[fill] {}
    child { 
    node[fill] {} 
	edge from parent 
	node[left, draw=none] {\tiny $4$}
	node[right, draw=none]  {\tiny $5$}    
    }
	edge from parent 
	node[sloped, ellipse, below, draw=none] {\tiny $3$}
	node[sloped, ellipse, above, draw=none]  {\tiny $6$}  
	};
\end{tikzpicture} \\
\tiny
\begin{tabular}{ | c | c | c | }
  \hline
  $1$ & $3$ & 4 \\ \hline
  $2$ & $5$ & 6 \\ \hline
\end{tabular}
\end{gathered} $
\vspace{10pt}
%%%%%%%%%%%%%%%%%%%% Right arm ends %%%%%%%%%%%%%%%%%%%%%%%%
\end{tabular}
};
\end{tikzpicture}
\vfill
}
\caption{$n = 3$ \label{fig:n3asymm}}
%%%%%%%%%%%%%%%%%%% n = 3 ends %%%%%%%%%%%%%%%%%%%%%%
%%%%%%%%%%%%%%%%%%%%%%%%%%%%%%%%%%%%%%%%%%
%%%%%%%%%%%%%%%%%%%%%%%%%%%%%%%%%%%%%%%%%%
%%%%%%%%%%%%%%%%%%%%%%%%%%%%%%%%%%%%%%%%%%
%%%%%%%%%%%%%%%%%%%%%%%%%%%%%%%%%%%%%%%%%%
%%%%%%%%%%%%%%%%%%% n = 4 begins %%%%%%%%%%%%%%%%%%%%%
\end{subfigure}
\begin{subfigure}[b]{0.66\textwidth}  
\centering
\usebox{\tempboxASYMM}
\caption{$n = 4$ \label{fig:n4asymm}} 
\end{subfigure}
%%%%%%%%%%%%%%%%%%% n = 4 ends %%%%%%%%%%%%%%%%%%%%%%
\end{center}
\caption{Type $C$ base cases for asymmetric plane trees \label{fig:asymmetrybasecase}}
\end{figure}
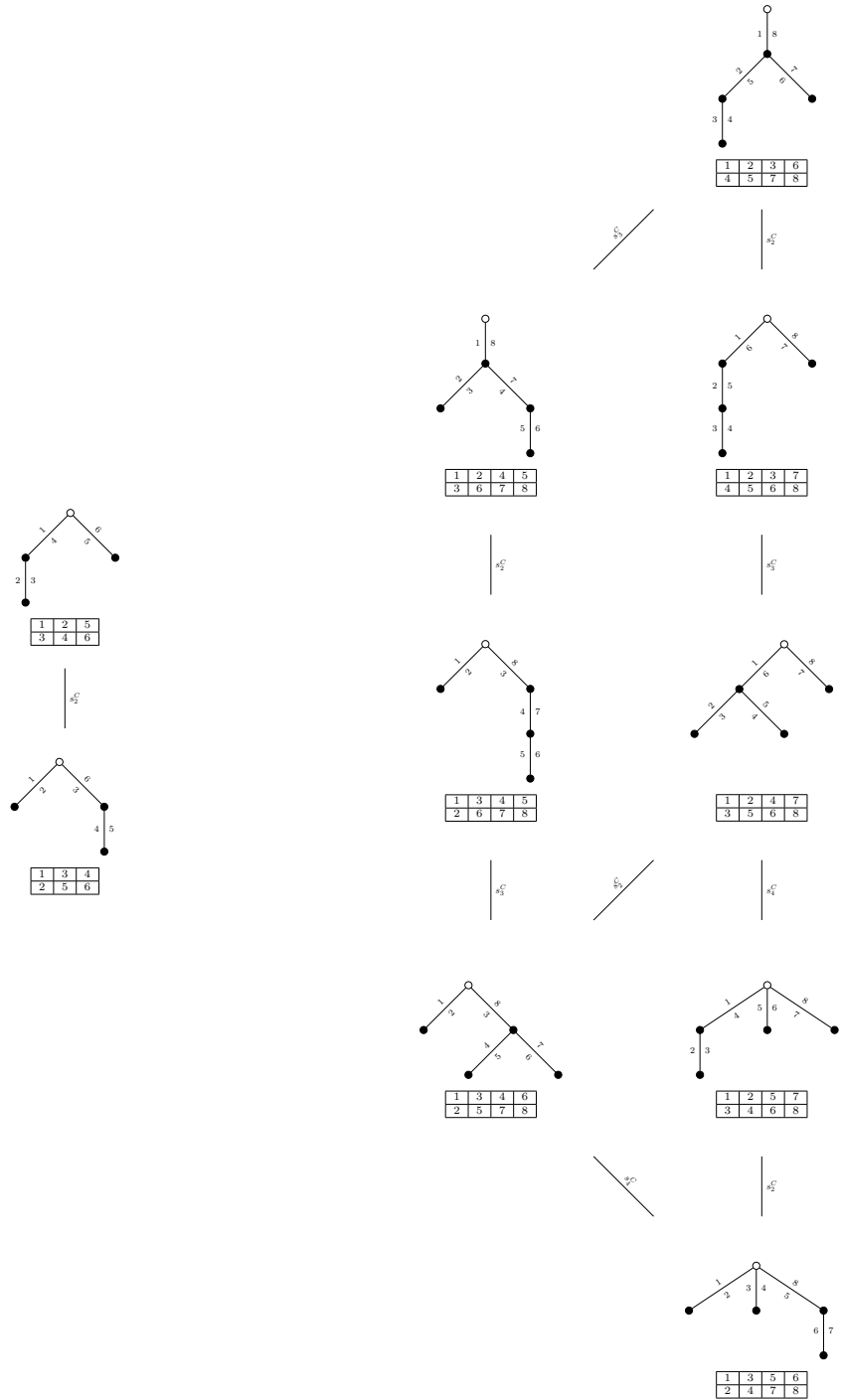
%%%%%%%%%%%%%%%%%%%%%%%%%%%%%%%%%%%%%%%%%%%%%%%%
%%%%%%%%%%%%%% Asymmetry base case diagram ends %%%%%%%%%%%%%%%
%%%%%%%%%%%%%%%%%%%%%%%%%%%%%%%%%%%%%%%%%%%%%%%%

For the induction step, let $n \geq 4$ and assume that any two asymmetric plane trees with at most $n-1$ edges can be transformed into each other by a sequence of $s_i^C$-local moves.   

Let $T$ be an arbitrary asymmetric plane tree with $n$ edges.  We describe an algorithm to obtain a sequence of $s_i^C$-local moves from $T$ to a plane tree with only the edge $e(1,2n)$ incident to the root.  (Note the special case of plane trees with three edges, for which there are no asymmetric trees containing the edge $e(1, 2n) = e(1,6)$.)  Figure~\ref{fig:bigstar} gives a schematic of $T$ with notation for the half-edges $j_1<j_1 +1<j_2< j_2+1< \cdots < j_{k-1}+1<j_k$ and the possibly-empty subtrees $a_i$.

%%%%%%%%%%%%%%%%%%%%%%%%%%%%%%%%%%%%%%%%%%%%%%%%%%
%%%%%%%%%%%%%%%%%% Big star diagram begins %%%%%%%%%%%%%%%%%%%
%%%%%%%%%%%%%%%%%%%%%%%%%%%%%%%%%%%%%%%%%%%%%%%%%%
\begin{center}
\begin{figure}[H]
\tikzset{blob/.style={draw, dashed, kite, rounded corners, shape border rotate=180, minimum size=1.4cm, xshift=0cm, yshift=1.4cm}} % Subtree blob style
\begin{tikzpicture} [level distance=2cm,
level 1/.style={sibling distance=3cm},
level 2/.style={sibling distance=1cm}]
\tikzstyle{every node}=[circle, draw, scale=.8, inner sep=2pt]
\node (root) [circle, draw] {}
  child {node[fill] (1) {}
  	child{ node[blob] {\small $a_{1}$} } % Leaf blob
    	edge from parent 
	node[sloped, ellipse, above, draw=none] {\small $1$}
	node[sloped, ellipse, below, draw=none]  {\small $j_{1}$}}  
  child {node[fill] (2) {}
  	child{ node[blob] {\small $a_{2}$} } % Leaf blob
     	edge from parent 
	node[sloped, ellipse, above, draw=none] {\small $j_{1} + 1$}
	node[sloped, ellipse, below, draw=none]  {\small $j_{2}$}}  
  child {node[fill] (3) {}
  	child{ node[blob] {\small $a_{i}$} } % Leaf blob
     	edge from parent 
	node[sloped, ellipse, above, draw=none] {\small $j_{i}$}
	node[sloped, ellipse, below, draw=none]  {\small $j_{i - 1} + 1$}}  
  child {node[fill] (4) {}
  	child{ node[blob] {\small $a_{k}$} } % Leaf blob
     	edge from parent 
	node[sloped, ellipse, above, draw=none] {\small $j_{k} = 2n$}
	node[sloped, ellipse, below, draw=none]  {\small $j_{k - 1} + 1$}} ;

\path (2) -- (3) node[draw=none, midway] () {\LARGE $\cdots$}; % Dashed line
\path (3) -- (4) node[draw=none, midway] () {\LARGE $\cdots$}; % Dashed line
\draw [decorate,decoration={brace,amplitude=2mm}] ($(4.west)+(1mm,-1.75cm)$) -- ($(1.east)+(-1mm,-1.75cm)$) ; % Bottom brace
\node [draw=none, anchor=center, yshift=-5.5cm] {$k$ subtrees} ; % Bottom label
\end{tikzpicture} \vspace{-0.75cm}
\caption{Edges incident to the root in $T$ \label{fig:bigstar}}
\end{figure}
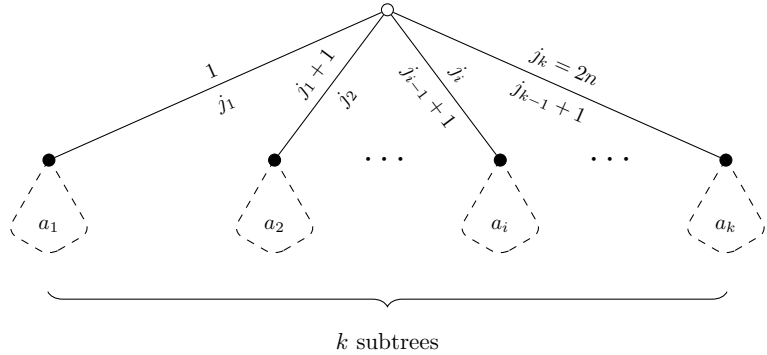
\end{center}
%%%%%%%%%%%%%%%%%%%%%%%%%%%%%%%%%%%%%%%%%%%%%%%%%%
%%%%%%%%%%%%%%%%%% Big star diagram ends %%%%%%%%%%%%%%%%%%%%
%%%%%%%%%%%%%%%%%%%%%%%%%%%%%%%%%%%%%%%%%%%%%%%%%%

Since $s_{j_1}^C = s_{j_1}s_{2n - j_1}$ we can use an $s_{j_1}^C$-local move on edges $e(1, j_1)$ and $e(j_1 + 1, j_2)$ to form edges $e(1, j_2)$ and $e(j_1, j_1 + 1)$.  Repeat this process for each $j_p$ with $j_p \leq n$. 

We can continue this process for any edge $e(1,j)$ with $j>n$ as long as we are not in the case of Figure~\ref{fig:problems}.  The problem in that case is that the local move that collapses $2n-j_p$ and $2n-j_p+1$ simultaneously triggers a type (1) local move on half-edges $j_p$ and $j_p+1$ and reinserts a lower-indexed branch into the root. (Note that $j_p < n < 2n-j_p$ by our convention on the labeling of the half-edges in the plane tree.)  

To address the case in Figure~\ref{fig:problems}, we apply the sequence $s_{j_p-1}^C s_{j'_{q-1}}^C \cdots s_{j'_2}^C s_{j'_1}^C$ of $s_i^C$-local moves.  Since $j_p<n$ the sequence of local moves permutes indices in the sets $\{j'_1,\ldots,j_p\}$ and $\{2n-j_p+1,2n-j_p+2,\ldots,2n-j'_1+1\}$ independently.   Thus after applying those $s_i^C$-local moves, the tree contains both of the edges $e(1,2n-j_p)$ and $e(2,j_p+1)$.  Applying $s_{j_p}^C$ to that tree results in a plane tree with edge $e(1,k)$ for $k \geq 2n-j_p+2$ as desired.  Continuing this process, we obtain in all cases a sequence of $s_i^C$-local moves that transforms an arbitrary asymmetric plane tree to one containing the edge $e(1,2n)$.

%%%%%%%%%%%%%%%%%%%%%%%%%%%%%%%%%%%%%%%%
%%%%%%%%%%%%%%% Workaround diagram begins %%%%%%%%%%%
%%%%%%%%%%%%%%%%%%%%%%%%%%%%%%%%%%%%%%%%
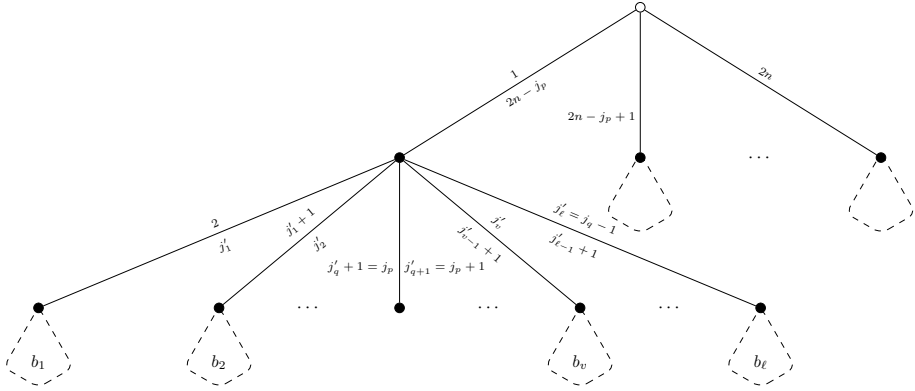
\begin{figure}[H]
\begin{center}
\tikzset{blob/.style={draw, dashed, kite, rounded corners, minimum size=1.45cm}} % Subtree blob style
\begin{tikzpicture}[scale = 0.8, every node/.style={scale=0.8}]
    %%%%%%%%%%%%%%%%%%%%
    % set up node for tree 1
    \node[] (tree1) at (0, 0) {
\begin{tikzpicture} [level distance=2.5cm,
level 1/.style={sibling distance=4cm},
level 2/.style={sibling distance=3cm},
level 3/.style={sibling distance=2cm}]
\tikzstyle{every node}=[circle, draw, scale=.8, inner sep=2pt]
\node (root) [circle, draw] {}
  child {node[fill] (1) {}
  	child{ node[fill] (11) {}
		child{ node[blob, shape border rotate=180, xshift=0cm, yshift=2cm] {\small $b_1$} } % Subtree
		edge from parent 
		node[sloped, ellipse, above, draw=none] {\tiny $2$}
		node[sloped, ellipse, below, draw=none]  {\tiny $j'_{1}$}
	} % Leaf
	child{ node[fill] (12) {}
		child{ node[blob, shape border rotate=180, xshift=0cm, yshift=2cm] {\small $b_2$} } % Subtree
		edge from parent 
		node[sloped, ellipse, above, draw=none] {\tiny $j'_{1} + 1$}
		node[sloped, ellipse, below, draw=none]  {\tiny $j'_{2}$}
	} % Leaf
	child{ node[fill] (13) {}
%		child{ node[blob, shape border rotate=180, xshift=0cm, yshift=2cm] {\small $b_u$} } % Subtree	
		edge from parent 
		node[right, draw=none, pos=0.75] {\tiny $j'_{q+1}=j_p+1$}
		node[left, draw=none, pos=0.75]  {\tiny $j'_q+1=j_p$}
	} % Leaf
	child{ node[fill] (14) {}
		child{ node[blob, shape border rotate=180, xshift=0cm, yshift=2cm] {\small $b_v$} } % Subtree	
		edge from parent 
		node[sloped, ellipse, above, draw=none] {\tiny $j'_v$}
		node[sloped, ellipse, below, draw=none]  {\tiny $j'_{v-1} + 1$}
	} % Leaf
	child{ node[fill] (15) {}  
		child{ node[blob, shape border rotate=180, xshift=0cm, yshift=2cm] {\small $b_{\ell}$} } % Subtree	
		edge from parent 
		node[sloped, ellipse, below, draw=none] {\tiny $j'_{\ell - 1} + 1$}
		node[sloped, ellipse, above, draw=none]  {\tiny $j'_{\ell} = j_q - 1$}
	} % Leaf
     	edge from parent 
	node[sloped, ellipse, above, draw=none] {\tiny $1$}
	node[sloped, ellipse, below, draw=none]  {\tiny $2n-j_p$}
	} % Branch of focus
  child {node[fill] (2) {}
    	child{ node[blob, shape border rotate=180, xshift=0cm, yshift=2cm] {\small $$} } % Subtree
     	edge from parent 
	node[left, draw=none, pos=0.75] {\tiny $2n-j_p+ 1$}
	node[right, draw=none, pos=0.75] {\tiny $$}}  
  child {node[fill] (k) {}
    	child{ node[blob, shape border rotate=180, xshift=0cm, yshift=2cm] {\small $$} } % Subtree
     	edge from parent 
	node[sloped, ellipse, above, draw=none] {\tiny $2n$}
	node[sloped, ellipse, below, draw=none]  {\tiny $$}} ;

\path (2) -- (k) node[draw=none, midway] () {$\cdots$}; % Dashed line
\path (12) -- (13) node[draw=none, midway] () {$\cdots$}; % Dashed line
\path (13) -- (14) node[draw=none, midway] () {$\cdots$}; % Dashed line
\path (14) -- (15) node[draw=none, midway] () {$\cdots$}; % Dashed line
\end{tikzpicture} 
    };

  \end{tikzpicture}
\end{center}
\caption{Problematic case \label{fig:problems}}
\end{figure}
%%%%%%%%%%%%%%%%%%%%%%%%%%%%%%%%%%%%%%%%
%%%%%%%%%%%%%% Workaround diagram ends %%%%%%%%%%%%
%%%%%%%%%%%%%%%%%%%%%%%%%%%%%%%%%%%%%%%%

Finally we show that all asymmetric plane trees are in the same connected component of $G^C$.  Suppose $T$ and $T'$ are both asymmetric plane trees with at least $4$ edges.  By the previous argument, we can assume that they each contain the edge $e(1,2n)$. Consider the subtrees $T_1$ and $T_1'$ consisting of all the edges of $T$ and respectively $T'$ except $e(1,2n)$.  The two subtrees are still asymmetric but have only $n-1$ edges.  By the inductive hypothesis we can transform $T_1$ into $T_1'$ with a sequence of $s_i^C$-local moves, which also transforms $T$ into $T'$.  By induction the claim is proven.
\end{proof}

The main result is a simple corollary of the previous results.

\begin{corollary}\label{corollary: two components}
The graph $G^C$ has exactly two connected components: one containing exactly the symmetric plane trees and the other containing exactly the asymmetric plane trees.
\end{corollary}

%%%%%%%%%%%%%%%%%%%% arXiv only %%%%%%%%%%%%%%%%%%%%%
Appendix A gives larger examples of Corollary~\ref{corollary: two components} for $n = 5, 6, 7$.
%%%%%%%%%%%%%%%%%%%%%%%%%%%%%%%%%%%%%%%%%%%%%%

We conclude with a formula for the size of each connected component in $G^C$, namely the number of symmetric plane trees and the number of asymmetric plane trees.

\begin{proposition}\label{proposition: count}
Given $\mathcal{T}_n$ the number of symmetric plane trees is $$r=\sum_{m}^{} \sum_{k_1 + k_2 + \cdots + k_{m+1} = \frac{n-m}{2}}^{} \prod_{j=1}^{m+1} C_{k_j}$$ where $m$ varies over odd numbers between $0$ and $n$ when $n$ is odd and over even numbers between $0$ and $n$ when $n$ is even.  The number of asymmetric plane trees is $C_n - r$.
\end{proposition}

\begin{proof}
We use the fact that the total number of plane trees with $n$ edges is the Catalan number $C_n = \frac{1}{n+1} \binom{2n}{n}$.

Define the {\em middle path graph} of a symmetric plane tree in $\mathcal{T}_n$ to be the maximal set of edges of the form $e(i,2n+1-i)$ for some $i$ with $1 \leq i \leq n$.  Let $m$ be the number of edges in the middle path graph of a symmetric plane tree.  To be a symmetric plane tree, any descendants to the left of a vertex in the middle path graph have their mirror image to the right of the same vertex.  Thus the set of all symmetric plane trees can be constructed by all possible ways to attach plane trees to the left of the middle path graph, together with the mirror images on the right.  There are $m+1$ vertices in the middle path graph; suppose that for each $i$ with $1 \leq i \leq m+1$ the $i^{th}$ vertex from the root in the middle path graph has a plane tree with $k_i$ edges to its left.  The sum $k_1+k_2+\cdots +k_{m+1}$ must satisfy
\[k_1+k_2+\cdots +k_{m+1} = \frac{n-m}{2}\]
since there are $n$ total edges in the tree, $m$ edges on the middle path graph, and another $k_1+k_2+\cdots +k_{m+1}$ edges in the mirror images of the subtrees to the left of the middle path graph.  By examining parity, we see that $m$ varies over odd numbers from $0$ to $n$ if $n$ is odd and over even numbers from $0$ to $n$ if $n$ is even.  For any such partition $k_1+k_2+\cdots +k_{m+1}$ we can independently take any of the $C_{k_i}$ plane trees on $k_i$ vertices to attach to the left of the $i^{th}$ vertex on the middle path graph, with its mirror image on the right.  Thus the total number of symmetric plane trees with $n$ edges is
\begin{align}
\sum_{m}^{} \sum_{k_1 + k_2 + \cdots + k_{m+1} = \frac{n-m}{2}}^{} \prod_{j=1}^{m+1} C_{k_j} \label{eq:symm}
\end{align}
as desired.  The number of asymmetric plane trees is simply the number of all plane trees minus the number of symmetric plane trees.
\end{proof}

\begin{question}
Do other properties of $G^A$ hold for the components of $G^C$ as well?  For instance is each component of $G^C$ graded by a function with a straightforward description?
\end{question}

\section{Remarks on classical types $B$ and $D$ and possible biological interpretations}

We conclude this paper with remarks and questions in two directions. First we discuss whether $s_i$-local moves could be reasonably extended to Weyl groups of other classical Lie types.  At the end we discuss speculative connections to biology.

\subsection{Extending $s_i$-local moves combinatorially to other classical Lie types}

There are two other Weyl groups of classical types, namely the Weyl groups of type $B$ and type $D$.  Both can be described as a subgroup of a sufficiently large permutation group.  

We think the Weyl group of type $D$ is unlikely to extend fruitfully to the setting of plane trees.  The problem is that the generators of the Weyl group of type $D$ cannot be written as a product of disjoint simple transpositions $(i,i+1)$.  Indeed, one generator must contain a transposition like $(n,n+2)$.  Within the permutation group, that transposition equals both 
\[(n,n+1)(n+1,n+2)(n,n+1) = (n+1,n+2)(n,n+1)(n+1,n+2).\]  
However the $s_i$-local moves do not form a group action; as we discussed in Remark \ref{remark: not group action} there is no consistent way to define $(n,n+2)$.

By contrast the Weyl group of type $B$ may lead to meaningful biological and combinatorial implications. 
The maps of type $B$ are the involutions defined by:
\begin{center}
$s_1^B=s_1s_{2n}$ corresponding to the reflection $(1,2)(2n,2n+1)$\\
$s_2^B=s_2s_{2n-1}$ corresponding to the reflection $(2,3)(2n-1,2n)$\\
\qquad \vdots \\
 $s_{n-1}^B=s_{n-1}s_{n+2}$ corresponding to the reflection $(n-1,n)(n+2,n+3)$\\
 $s_n^B = s_n$ corresponding to the reflection $(n, n+2)$.
\end{center}
Note that $s_n^B$ is different from the other permutations, much like $s_n^C$. (Also like the Weyl group of type $C$, we only have $s^{B}_{i}$ for $i \in \{1, 2, \ldots, n\}$.)  Though it is not a simple transposition, the fact that $n+1$ is fixed by all of the other generators $s_i^B$ means that we can define an unambiguous action on standard tableaux of shape $(n+1,n)$.  In this action, the map $s_n^B$ exchanges $n$ and $n+2$ and the other maps $s_i^B$ act as the corresponding product of type $A$ $s_i$-local moves.  

The type $B$ involutions do not act on plane trees since plane trees must have an even number of half-edges.  However they do act on objects like plane trees that have $n$ whole edges and an unpaired half-edge labeled $n+1$. This half-edge forms a small loop or bulge between the half-edges labeled $n$ and $n+2$. Figure~\ref{fig:B} gives an example.  As with the maps $s_i^B$ on tableaux, the action on these modified plane trees always fixes the bulge $n+1$.  

%%%%%%%%%%%%%%%%%%%%%%%%%%%%%%%%%%%%%%%%%%%%%%%%%%%
%%%%%%%%%%%%%%%%%%%% Type B begins %%%%%%%%%%%%%%%%%%%%%%%
%%%%%%%%%%%%%%%%%%%%%%%%%%%%%%%%%%%%%%%%%%%%%%%%%%%
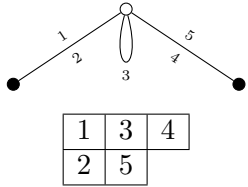
\begin{figure}[H]
$ \begin{gathered}
\begin{tikzpicture}[level distance=1cm,
level 1/.style={sibling distance=3cm},
level 2/.style={sibling distance=1cm}]
\tikzstyle{every node}=[circle, draw, scale=.8, inner sep=2pt]
\tikzstyle{every loop}=[-, shorten >=0pt, looseness=50, min distance=5mm]
\node[inner sep=2pt] (root) {}
    child {
    node[fill] {} 
    	edge from parent 
	node[sloped, ellipse, above, draw=none] {\tiny $1$}
	node[sloped, ellipse, below, draw=none]  {\tiny $2$}  
    }
    child {
    node[fill] {}
    	edge from parent 
	node[sloped, ellipse, below, draw=none] {\tiny $4$}
	node[sloped, ellipse, above, draw=none]  {\tiny $5$}  
    };
    % Loop
    \draw (root) edge [loop below] node[draw=none] {\tiny $3$} ();
\end{tikzpicture} \\
\begin{tabular}{ | c | c | c | }
  \hline
  $1$ & $3$ & 4 \\ \hline
  $2$ & $5$ & \multicolumn{1}{c}{} \\ \cline{1-2}
\end{tabular}
\end{gathered} $
\caption{Type-B model when $n=2$ \label{fig:B}}
\end{figure}
%%%%%%%%%%%%%%%%%%%%%%%%%%%%%%%%%%%%%%%%%%%%%%%%%%%
%%%%%%%%%%%%%%%%%%%%% Type B ends %%%%%%%%%%%%%%%%%%%%%%%
%%%%%%%%%%%%%%%%%%%%%%%%%%%%%%%%%%%%%%%%%%%%%%%%%%%

We leave these investigations for future research, for instance in the following questions.

\begin{question}
What are the orbits of the action of involutions $s_i^B$?  What is a natural collection of involutions to represent mutations on strands with several bulges (namely fixing several integers)?
\end{question}

\subsection{Speculative connections between Weyl groups of classical types and biology}
We extended local moves combinatorially from $S_n$ to other Weyl groups of classical types.  We end with speculative comments and questions about whether the maps we defined are observed in any biological contexts.

We begin with possible biological interpretations of type $C$ local moves.  The product of DNA transcription, messenger RNA (mRNA) carries genetic information contained in DNA from the cell nucleus to the cytoplasm, where protein synthesis takes place.  During the normal process of translation, a ribosome reads an mRNA strand from the 5' end of the base sequence to the 3' end, decoding three bases into one amino acid molecule at a time. Whereas type $A$ local moves act by twisting RNA strands at a particular location, we think of a type $C$ local move as exchanging two triples of base pairs at some point in the translation process, a development that may completely change the sequence of amino acids.

We conjecture that the type $C$ local moves may correspond to certain RNA mutations. When $i = n$ the map $s_i^C$ replaces stacked bases with their Watson-Crick complement; otherwise, the maps $s_i^C$ exchange adjacent sets of stacked bases while preserving their bonds.  Figure~\ref{fig:CRNA} illustrates an example of the twisting mechanism when $s_2^C$ is applied for $n=4$.  (Applying $s_4^C$ in this example would exchange $4$ and $5$ instead.) %%%

%%%%%%%%%%%%%%%%%%%%%%%%%%%%%%%%%%%%%%%%
%%%%%%%%%%%%%%% C RNA diagram begins %%%%%%%%%%%%%
%%%%%%%%%%%%%%%%%%%%%%%%%%%%%%%%%%%%%%%%
\begin{center}
\begin{figure}[H]
\begin{tikzpicture}
%%%%%%%%%%%%%%%%%%%%%%%%%%%%%%%%%%%%%%%%
%%%%%%%%%%%%%%%%%% Left begins %%%%%%%%%%%%%%%%
\node (L) at (0, 0) {
\begin{tikzpicture}[x=1cm, y=3cm]
    \def \steps{4}
    \def \width{0.75}
    \pgfmathsetmacro \stepsize{1/\steps}
    %%%%%%%%%%%%%%%%%%%%
    \draw [thick, dashed] (\width, \steps*\stepsize) arc[radius = 0.5cm, start angle= -40, end angle=220]; % Hairpin bulge
    \draw (0, -0.05) node {\LARGE $\vdots$}; % Left leg
    \draw (\width, -0.05) node {\LARGE $\vdots$}; % Right leg
    %%%%%%%%%%%%%%%%%%%%    
    \foreach \i in {1, 2,..., \steps} {
        \draw [thick]
        (0, {(\i-1)*\stepsize}) -- (0, {\i*\stepsize}) % Left vertical rail
        (\width, {(\i-1)*\stepsize}) -- (\width, {\i*\stepsize}); % Right vertical rail
    }
    %%%%%%%%%%%%%%%%%%%%
    \foreach \i in {1, ..., \steps} {
	\pgfmathtruncatemacro \diff{\steps - \i + 1}
    	\draw (0, {\i*\stepsize}) node [circle, inner sep=1.75pt, fill] {}; % Left side nodes
	\draw (-0.3, {\i*\stepsize}) node [ellipse, inner sep=1.75pt] {\tiny $\i$}; % Left side labels
    	
	\draw (\width, {\i*\stepsize}) node [circle, inner sep=1.75pt, fill] {}; % Right side nodes
	\pgfmathtruncatemacro{\var}{\steps+\diff}
	\draw (1.05, {\i*\stepsize}) node [ellipse, inner sep=1.75pt] {\tiny $\var$}; % Right side labels
	
        \draw [dashed] (0, {\i*\stepsize}) edge (\width, {\i*\stepsize}); % Rungs on ladder
    }
\end{tikzpicture}
};
%%%%%%%%%%%%%%%%%% Left ends %%%%%%%%%%%%%%%%
%%%%%%%%%%%%%%%%%%%%%%%%%%%%%%%%%%%%%%%%

%%%%%%%%%%%%%%%%%%%%%%%%%%%%%%%%%%%%%%%%
%%%%%%%%%%%%%%%%%% Right begins %%%%%%%%%%%%%%%
\node (R) at (5, 0) {
\begin{tikzpicture}[x=1cm, y=3cm]
    \def \steps{4}
    \def \width{0.75}
    \pgfmathsetmacro \stepsize{1/\steps}
    %%%%%%%%%%%%%%%%%%%%
    \draw [thick, dashed] (\width, \steps*\stepsize) arc[radius = 0.5cm, start angle= -40, end angle=220]; % Hairpin bulge
    \draw (0, -0.05) node {\LARGE $\vdots$}; % Left leg
    \draw (\width, -0.05) node {\LARGE $\vdots$}; % Right leg
    %%%%%%%%%%%%%%%%%%%%    
    \foreach \i in {1, 2,..., \steps} {
        \draw [thick]
        (0, {(\i-1)*\stepsize}) -- (0, {\i*\stepsize}) % Left vertical rail
        (\width, {(\i-1)*\stepsize}) -- (\width, {\i*\stepsize}); % Right vertical rail
    }
    %%%%%%%%%%%%%%%%%%%%
    \foreach \i in {1, ..., \steps} {
    	\draw (0, {\i*\stepsize}) node [circle, inner sep=1.75pt, fill] {}; % Left side nodes	
    	\draw (\width, {\i*\stepsize}) node [circle, inner sep=1.75pt, fill] {}; % Right side nodes	
        \draw [dashed] (0, {\i*\stepsize}) edge (\width, {\i*\stepsize}); % Rungs on ladder
    } 
    	% Left side labels
	\draw (-0.3, 1) node [ellipse, inner sep=2pt] {\tiny $4$}; 
	\draw (-0.3, 0.75) node [ellipse, inner sep=2pt] {\tiny $2$}; 
	\draw (-0.3, 0.5) node [ellipse, inner sep=2pt] {\tiny $3$}; 
	\draw (-0.3, 0.25) node [ellipse, inner sep=2pt] {\tiny $1$};
	
	% Right side labels
	\draw (1.05, 1) node [ellipse, inner sep=2pt] {\tiny $5$}; 
	\draw (1.05, 0.75) node [ellipse, inner sep=2pt] {\tiny $7$}; 
	\draw (1.05, 0.5) node [ellipse, inner sep=2pt] {\tiny $6$}; 
	\draw (1.05, 0.25) node [ellipse, inner sep=2pt] {\tiny $8$};
\end{tikzpicture}
};
%%%%%%%%%%%%%%%%%% Right ends %%%%%%%%%%%%%%%%
%%%%%%%%%%%%%%%%%%%%%%%%%%%%%%%%%%%%%%%%

% draw arrows and text between nodes
\draw[shorten >=0.5cm, shorten <=0.5cm, -] (L)--(R) node [midway, above] {\tiny $s^{C}_{2} = s_2s_6$};
\end{tikzpicture}
\caption{The map $s_i^C$ for an element of the Weyl group of type $C$ acting on RNA base pairs \label{fig:CRNA}}
\end{figure}
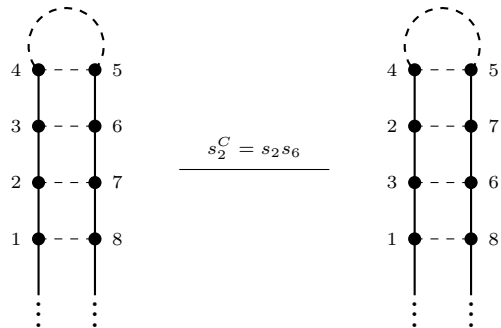
\end{center}
%%%%%%%%%%%%%%%%%%%%%%%%%%%%%%%%%%%%%%%%
%%%%%%%%%%%%%%% C RNA diagram ends %%%%%%%%%%%%%%
%%%%%%%%%%%%%%%%%%%%%%%%%%%%%%%%%%%%%%%%

Like the Weyl group of type $C$, the elements of the Weyl group of type $B$ correspond to mutations on an RNA strand.  But the type-$B$ model is different because the stacked bases now contain a bulge, namely the sequence of unmatched nucleotides corresponding to the half-edge $n+1$. 

\begin{question} Are any processes like this observed biologically?
\end{question}

%We expect additional changes depending on which strand the unmatched nucleotide is attached to as only one strand is being read during translation. We suspect the nucleotide bond's chances of being broken and reformed with a base belonging to the opposite nucleotide strand to vary according to the stability of the secondary and tertiary structures. Another question we leave for future (biological) exploration is the unmatched nucleotide's influence on the RNA's susceptibility to mutations.

%Our final open question concerns the implementation of "non-traditional" nucleotide pairing in the application of the Weyl groups. Base pairs $A-U$ and $C-G$ are the most common, however, sometimes bases that chemically have lower affinity happen to form bondings. As the biological implications merit some exploration the adjustments this observation brings to our application of the Weyl groups will be highly interesting. 

%%%%%%%%%%%%%%%%%%%%%%%%%%%%%%%%%%%%%%%%%%%%%%%
%%%%%%%%%%%%%%%%%%%% arXiv only %%%%%%%%%%%%%%%%%%%%%
%%%%%%%%%%%%%%%%% Appendix A begins %%%%%%%%%%%%%%%%%%%%
%%%%%%%%%%%%%%%%%%%%%%%%%%%%%%%%%%%%%%%%%%%%%%%
\clearpage
\appendix
\section{type $C$ Orbits for Higher-Order Plane Trees} \label{App:A}

Figures~\ref{fig:A5} to \ref{fig:A7} summarize the main results of Section~\ref{section: type C} by showing the type $C$ orbits for plane trees with $n = 5$, $6$, and $7$ edges, respectively. The edge labels indicate Weyl group permutations of type $C$.  Each graph has two connected components: one of symmetric plane trees and one of asymmetric plane trees. Note that the number of asymmetric plane trees becomes much greater than the number of symmetric plane trees as $n$ grows larger. Proposition~\ref{proposition: count} supplies the formulas that calculate the number of symmetric plane trees and the number of asymmetric plane trees given each $n$. The {\em Mathematica} notebook that generates these orbits is publicly available online at \cite{Wang15}.

%%%%%%%%%%%%%%%%%%%%%%%%%%%%%%%%%%%%%%%%%%%%%%%
%%%%%%%%%%%%%%%%%%% n = 5 orbits begin %%%%%%%%%%%%%%%%%%%
%%%%%%%%%%%%%%%%%%%%%%%%%%%%%%%%%%%%%%%%%%%%%%%
\newcommand{\gscaleFIVE}{0.25}
\newlength\smallheightFIVE
\newlength\bigheightFIVE
\settoheight\smallheightFIVE{\includegraphics[scale = \gscaleFIVE]{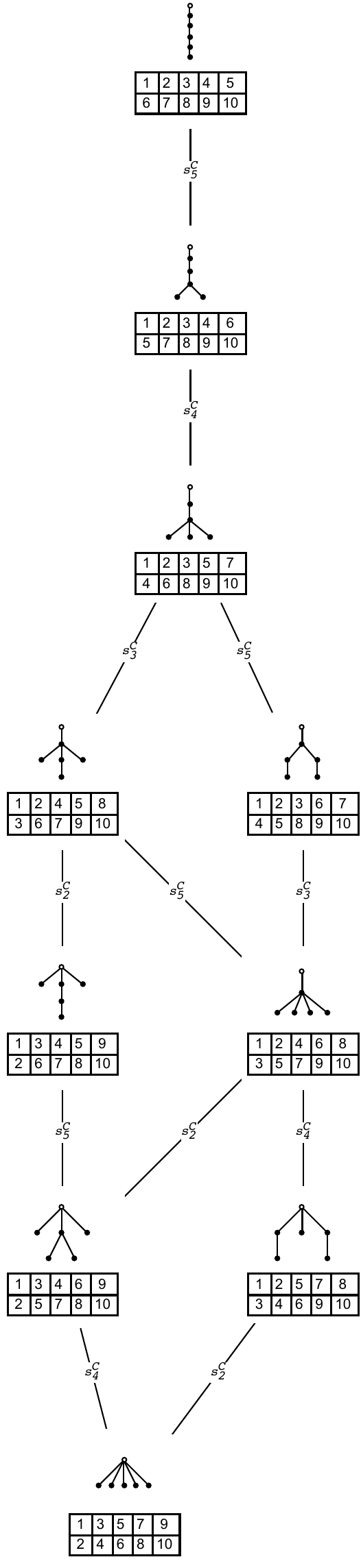}}
\settoheight\bigheightFIVE{\includegraphics[scale = \gscaleFIVE]{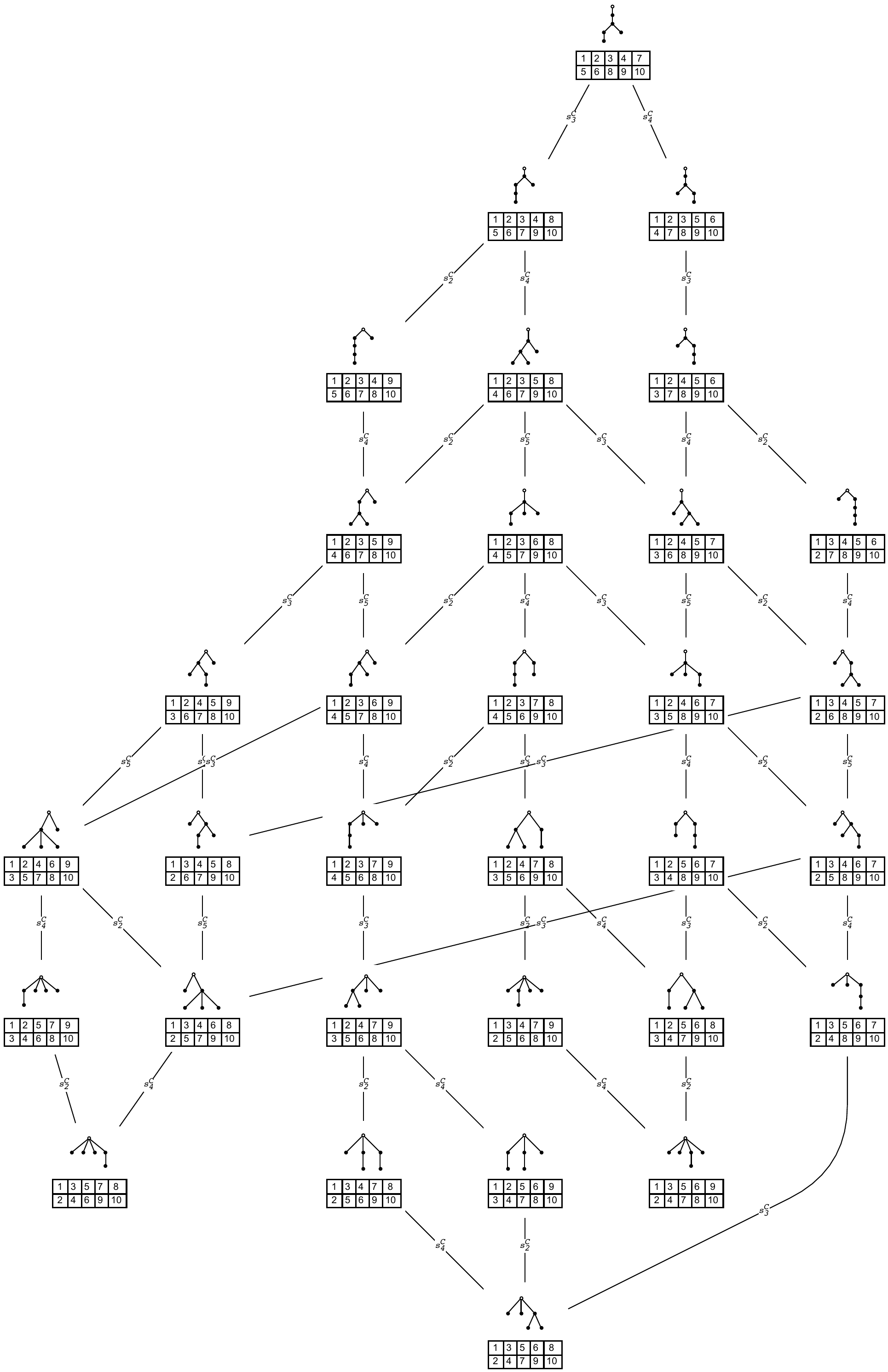}}
\pgfmathsetmacro \adjustheightFIVE{(\the\bigheightFIVE - \the\smallheightFIVE) / 2}

\begin{figure}[H]
\begin{center}
%%%%%%%%%%%%%%%%%%% Symmetric orbit begins %%%%%%%%%%%%%%%%
\begin{subfigure}[b]{0.25\textwidth}  
\centering
\raisebox{\adjustheightFIVE pt}{\includegraphics[scale = \gscaleFIVE]{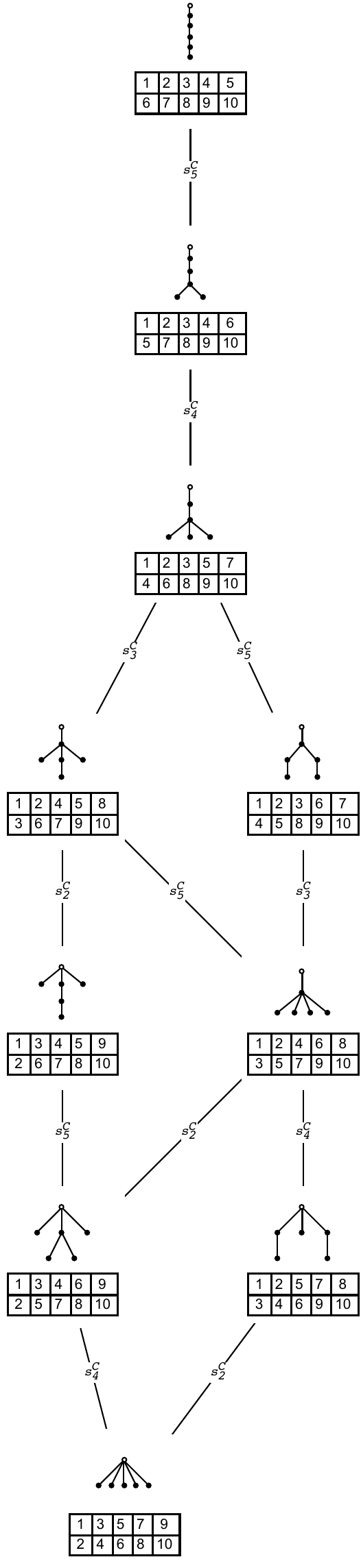}}
\caption{Symmetric orbit \label{fig:n5symm}}
%%%%%%%%%%%%%%%%% Symmetric orbit  ends %%%%%%%%%%%%%%%%%%%
%%%%%%%%%%%%%%%%%%% Asymmetric orbit begins %%%%%%%%%%%%%%%%
\end{subfigure}
\begin{subfigure}[b]{0.7\textwidth}  
\centering
\includegraphics[scale = \gscaleFIVE]{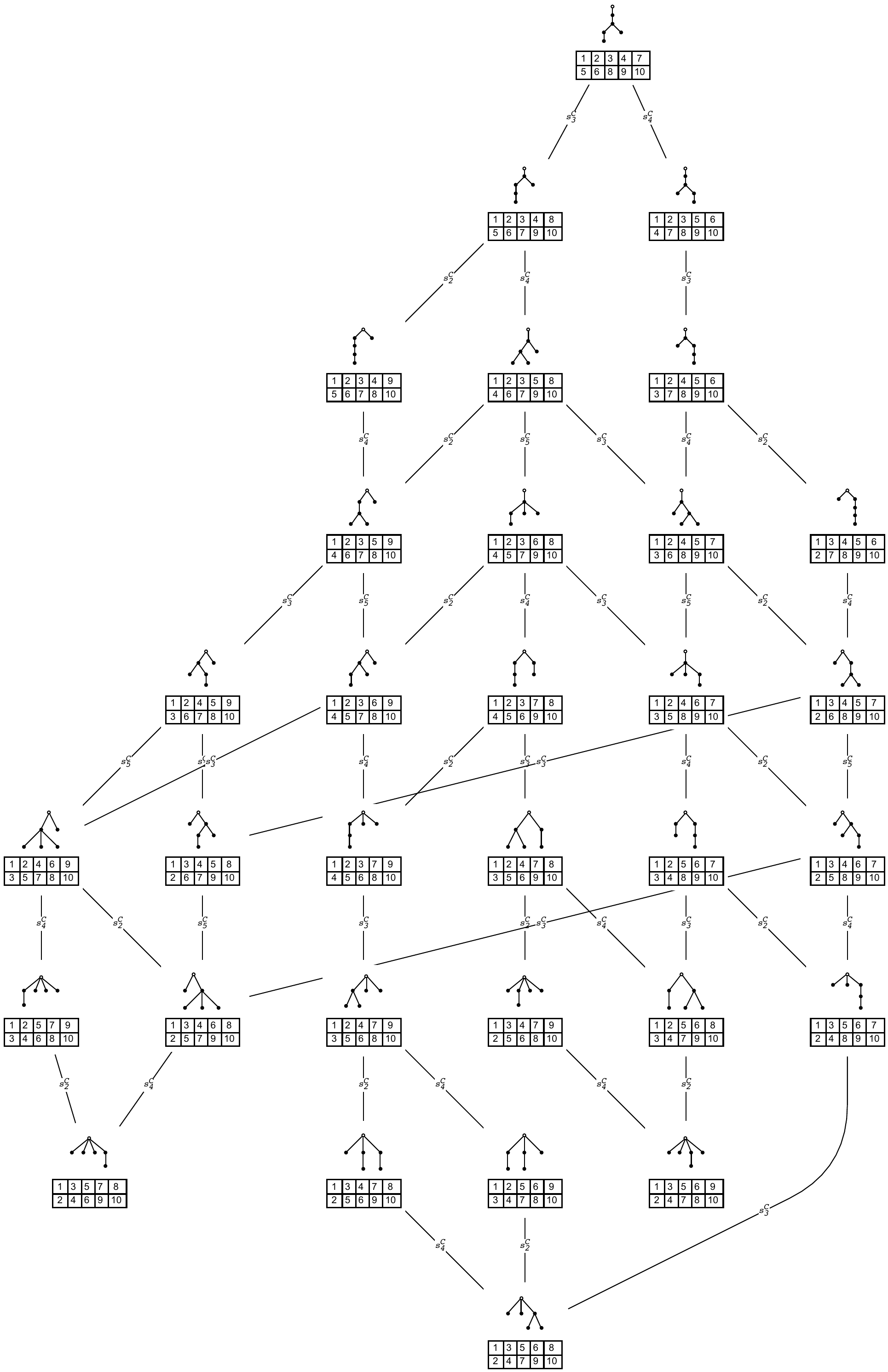}
\caption{Asymmetric orbit \label{fig:n5asymm}} 
\end{subfigure}
%%%%%%%%%%%%%%%%% Asymmetric orbit ends %%%%%%%%%%%%%%%%%%%
\end{center}
\caption{Type $C$ local moves for plane trees with $n=5$ edges \label{fig:A5}}
\end{figure}
%%%%%%%%%%%%%%%%%%%%%%%%%%%%%%%%%%%%%%%%%%%%%%%
%%%%%%%%%%%%%%%%%%% n = 5 orbits end %%%%%%%%%%%%%%%%%%%%
%%%%%%%%%%%%%%%%%%%%%%%%%%%%%%%%%%%%%%%%%%%%%%%
%%%%%%%%%%%%%%%%%%%%%%%%%%%%%%%%%%%%%%%%%%%%%%%

%%%%%%%%%%%%%%%%%%%%%%%%%%%%%%%%%%%%%%%%%%%%%%%
%%%%%%%%%%%%%%%%%%% n = 6 orbits begin %%%%%%%%%%%%%%%%%%%
%%%%%%%%%%%%%%%%%%%%%%%%%%%%%%%%%%%%%%%%%%%%%%%
\newcommand{\gscaleSIX}{0.125}
\newlength\smallheightSIX
\newlength\bigheightSIX
\settoheight\smallheightSIX{\includegraphics[scale = \gscaleSIX]{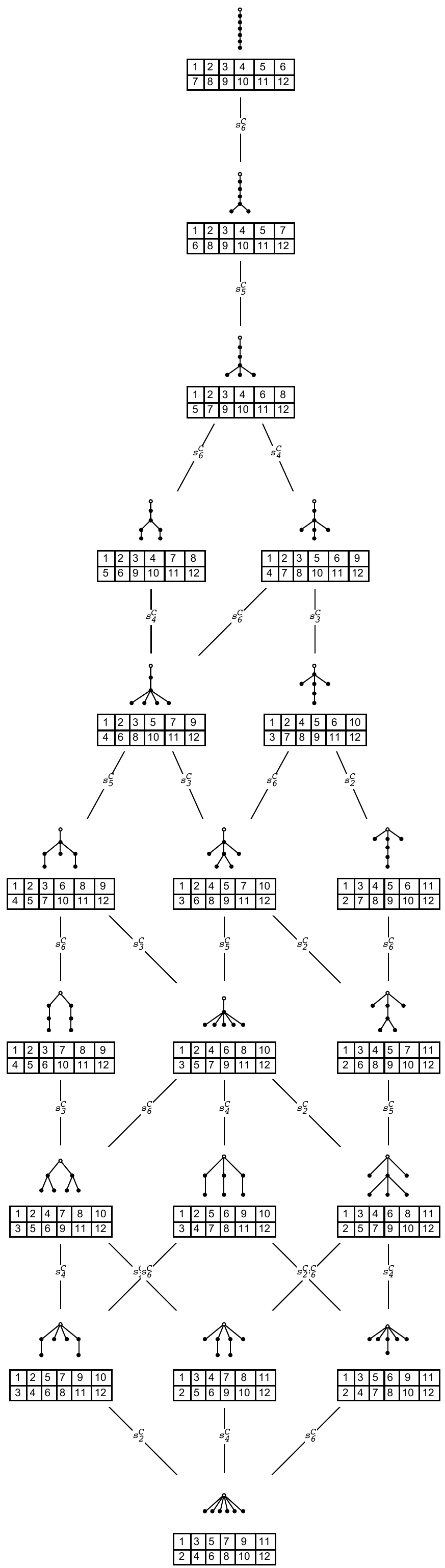}}
\settoheight\bigheightSIX{\includegraphics[scale = \gscaleSIX]{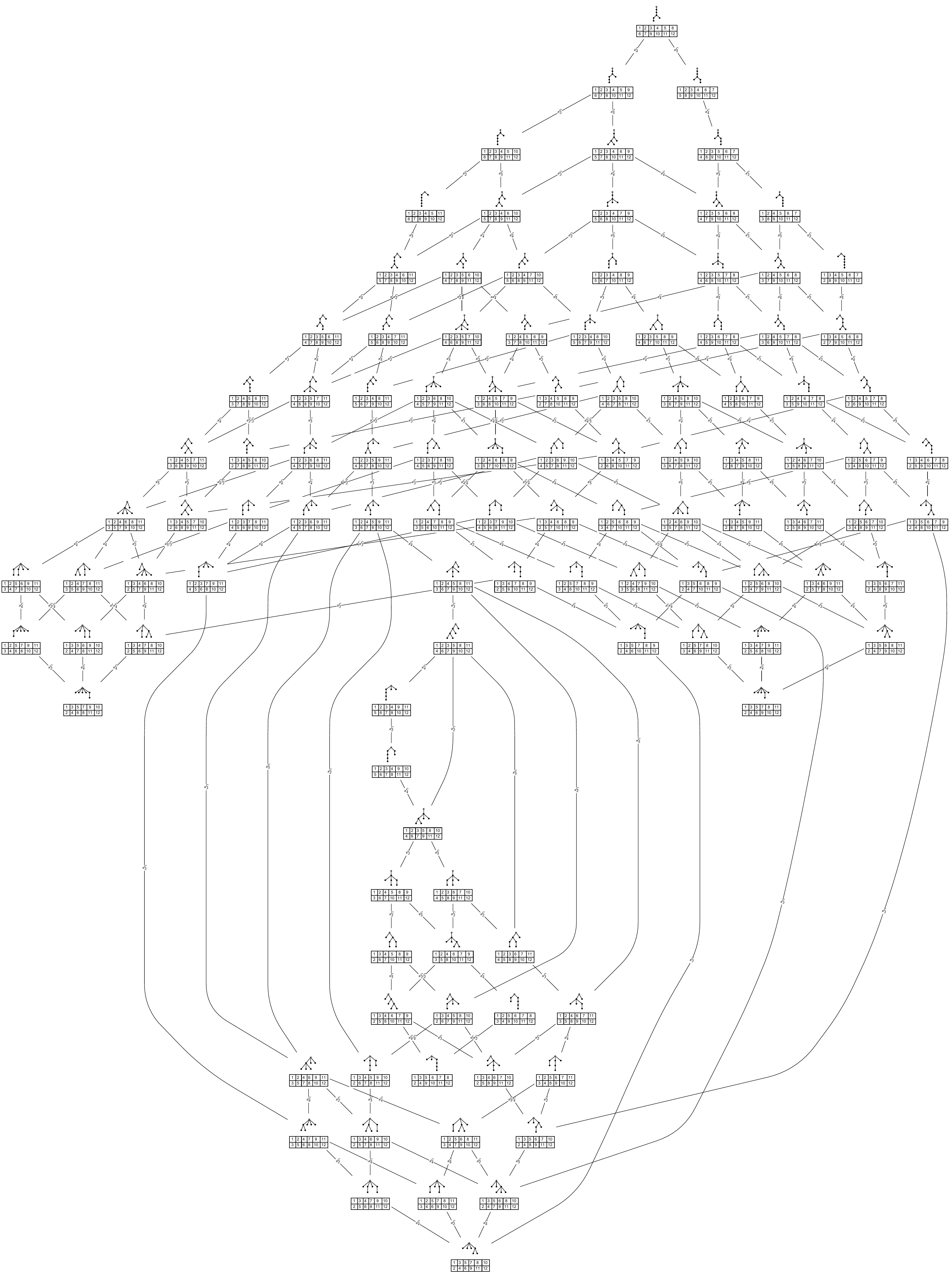}}
\pgfmathsetmacro \adjustheightSIX{(\the\bigheightSIX - \the\smallheightSIX) / 2}

\begin{figure}[H]
\begin{center}
%%%%%%%%%%%%%%%%%%% Symmetric orbit begins %%%%%%%%%%%%%%%%
\begin{subfigure}[b]{0.25\textwidth}  
\centering
\raisebox{\adjustheightSIX pt}{\includegraphics[scale = \gscaleSIX]{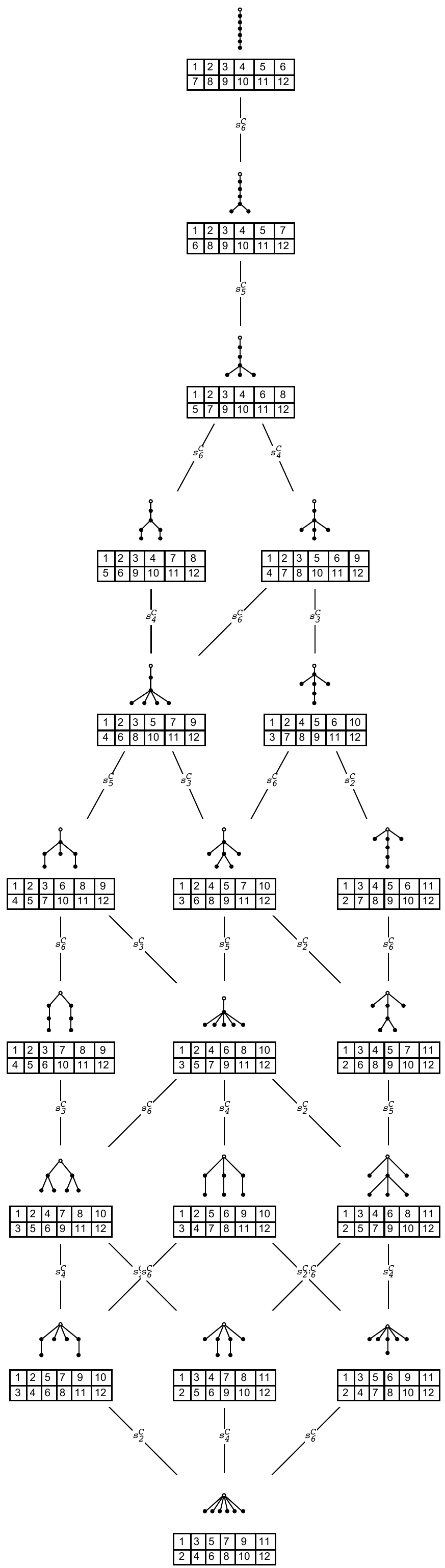}}
\caption{Symmetric orbit \label{fig:n6symm}}
%%%%%%%%%%%%%%%%% Symmetric orbit  ends %%%%%%%%%%%%%%%%%%%
%%%%%%%%%%%%%%%%%%% Asymmetric orbit begins %%%%%%%%%%%%%%%%
\end{subfigure}
\begin{subfigure}[b]{0.7\textwidth}  
\centering
\includegraphics[scale = \gscaleSIX]{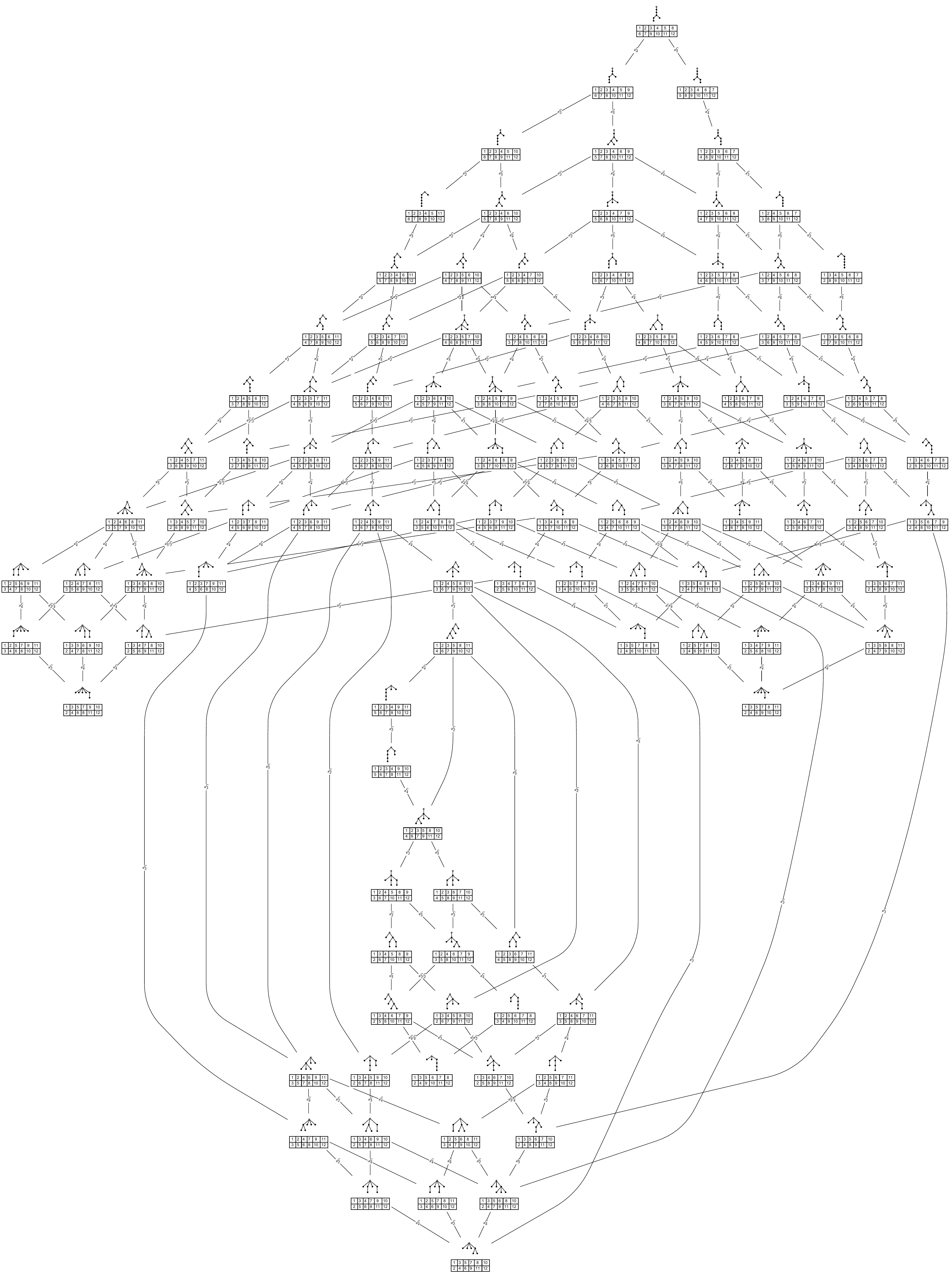}
\caption{Asymmetric orbit \label{fig:n6asymm}} 
\end{subfigure}
%%%%%%%%%%%%%%%%% Asymmetric orbit ends %%%%%%%%%%%%%%%%%%%
\end{center}
\caption{Type $C$ local moves for plane trees with $n=6$ edges \label{fig:A6}}
\end{figure}
%%%%%%%%%%%%%%%%%%%%%%%%%%%%%%%%%%%%%%%%%%%%%%%
%%%%%%%%%%%%%%%%%%% n = 6 orbits end %%%%%%%%%%%%%%%%%%%%
%%%%%%%%%%%%%%%%%%%%%%%%%%%%%%%%%%%%%%%%%%%%%%%
%%%%%%%%%%%%%%%%%%%%%%%%%%%%%%%%%%%%%%%%%%%%%%%

%%%%%%%%%%%%%%%%%%%%%%%%%%%%%%%%%%%%%%%%%%%%%%%
%%%%%%%%%%%%%%%%%%% n = 7 orbits begin %%%%%%%%%%%%%%%%%%%
%%%%%%%%%%%%%%%%%%%%%%%%%%%%%%%%%%%%%%%%%%%%%%%
\newcommand{\gscaleSEVEN}{0.055}
\newlength\smallheightSEVEN
\newlength\bigheightSEVEN
\settoheight\smallheightSEVEN{\includegraphics[scale = \gscaleSEVEN]{C7symmetric.pdf}}
\settoheight\bigheightSEVEN{\includegraphics[scale = \gscaleSEVEN]{C7asymmetric.pdf}}
\pgfmathsetmacro \adjustheightSEVEN{(\the\bigheightSEVEN - \the\smallheightSEVEN) / 2}

\begin{figure}[H]
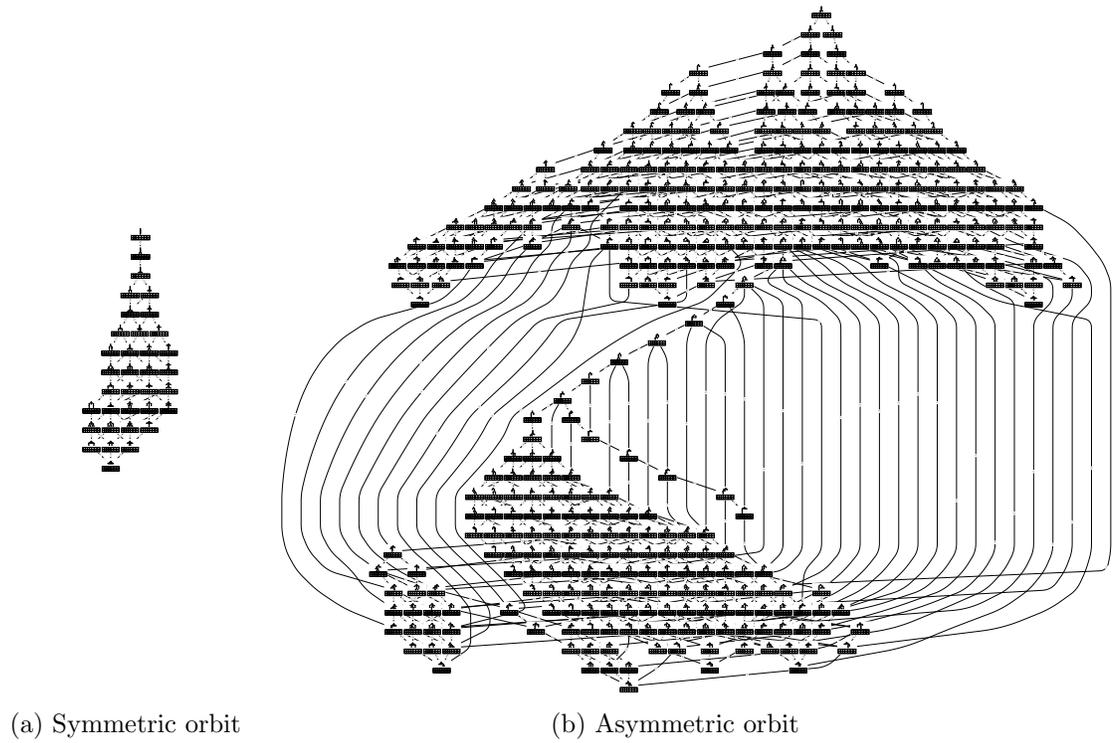

\begin{center}
%%%%%%%%%%%%%%%%%%% Symmetric orbit begins %%%%%%%%%%%%%%%%
\begin{subfigure}[b]{0.25\textwidth}  
\centering
\raisebox{\adjustheightSEVEN pt}{\includegraphics[scale = \gscaleSEVEN]{C7symmetric.pdf}}
\caption{Symmetric orbit \label{fig:n7symm}}
%%%%%%%%%%%%%%%%% Symmetric orbit  ends %%%%%%%%%%%%%%%%%%%
%%%%%%%%%%%%%%%%%%% Asymmetric orbit begins %%%%%%%%%%%%%%%%
\end{subfigure}
\begin{subfigure}[b]{0.7\textwidth}  
\centering
\includegraphics[scale = \gscaleSEVEN]{C7asymmetric.pdf}
\caption{Asymmetric orbit \label{fig:n7asymm}} 
\end{subfigure}
%%%%%%%%%%%%%%%%% Asymmetric orbit ends %%%%%%%%%%%%%%%%%%%
\end{center}
\caption{Type $C$ local moves on plane trees with $n=7$ edges \label{fig:A7}}
\end{figure}
%%%%%%%%%%%%%%%%%%%%%%%%%%%%%%%%%%%%%%%%%%%%%%%
%%%%%%%%%%%%%%%%%%% n = 7 orbits end %%%%%%%%%%%%%%%%%%%%
%%%%%%%%%%%%%%%%%%%%%%%%%%%%%%%%%%%%%%%%%%%%%%%
%%%%%%%%%%%%%%%%%%%%%%%%%%%%%%%%%%%%%%%%%%%%%%%

%%%%%%%%%%%%%%%%%% Appendix A ends %%%%%%%%%%%%%%%%%%%%
%%%%%%%%%%%%%%%%%%%% arXiv only %%%%%%%%%%%%%%%%%%%%%
%%%%%%%%%%%%%%%%%%%%%%%%%%%%%%%%%%%%%%%%%%%%%%%

\end{document}